\documentclass[11pt,oneside]{amsart}

\usepackage{fullpage}

\usepackage{amsthm,amsfonts,amssymb,amsmath,amsxtra,mathtools}
\usepackage[dvipsnames]{xcolor}
\usepackage[all]{xy}
\SelectTips{cm}{}
\usepackage{tikz}
\usepackage{verbatim}
\usepackage{todonotes}
\usepackage{mathrsfs}
\usepackage{booktabs,makecell}
\usepackage{diagbox}

\RequirePackage{xspace}
\RequirePackage{etoolbox}
\RequirePackage{varwidth}
\RequirePackage{enumitem}
\RequirePackage{tensor}
\RequirePackage{mathtools}
\RequirePackage{longtable}
\RequirePackage{multirow}
\RequirePackage{makecell}
\RequirePackage{bigints}

\usepackage[backref=page]{hyperref}%

\newcommand{\sH}{\ensuremath{\mathscr{H}}\xspace}
\newcommand{\sV}{\ensuremath{\mathscr{V}}\xspace}
\newcommand{\sS}{\ensuremath{\mathscr{S}}\xspace}

\DeclareMathOperator{\Gal}{Gal}
\DeclareMathOperator{\Hom}{Hom}
\DeclareMathOperator{\End}{End}
\DeclareMathOperator{\Res}{Res}
\newcommand{\wit}{\widetilde}
\newcommand{\Sh}{\mathrm{Sh}}

\DeclareMathOperator{\Int}{\ensuremath{\mathrm{Int}}\xspace}

\DeclareMathOperator{\Lie}{Lie}

\newcommand{\val}{{\mathrm{val}}}

\DeclareMathOperator{\tr}{Tr}

\newcommand{\U}{\mathrm{U}}

\newcommand{\SO}{\mathrm{SO}}
\DeclareMathOperator{\vol}{vol}

\newcommand{\isoarrow}{%
   \ifbool{@display}{\overset{\sim}{\longrightarrow}}{\xrightarrow\sim}%
   }

\newcommand{\Fb}{{\breve F}}

\newcommand{\OFb}{{O_{\breve F}}}
\newcommand{\Herm}{\mathrm{Herm}}

\newcommand{\rd}{\mathrm{d}}

\newcommand{\Den}{\mathrm{Den}}

\newcommand{\Hor}{\mathrm{Hor}}
\newcommand{\pDen}{\partial\mathrm{Den}}

\newcommand{\Intp}{\Int_{L^\flat,\sV}^\perp}
\newcommand{\pDenp}{\pDen_{L^\flat,\sV}^\perp}
\newcommand{\pPden}{\partial \mathrm{Pden}}

\newcommand{\KG}{K}

\newcommand{\pEis}{\partial\mathrm{Eis}}
\newcommand{\Eis}{\mathrm{Eis}}
\newcommand{\Diff}{\mathrm{Diff}}

\newcommand{\wdeg}{\widehat\deg}

\newcommand{\sz}{\mathsf{z}}
\newcommand{\sx}{\mathsf{x}}
\newcommand{\sy}{\mathsf{y}}

\DeclareFontFamily{U}{matha}{\hyphenchar\font45}
\DeclareFontShape{U}{matha}{m}{n}{
      <5> <6> <7> <8> <9> <10> gen * matha
      <10.95> matha10 <12> <14.4> <17.28> <20.74> <24.88> matha12
      }{}
\DeclareSymbolFont{matha}{U}{matha}{m}{n}
\DeclareFontFamily{U}{mathx}{\hyphenchar\font45}
\DeclareFontShape{U}{mathx}{m}{n}{
      <5> <6> <7> <8> <9> <10>
      <10.95> <12> <14.4> <17.28> <20.74> <24.88>
      mathx10
      }{}
\DeclareSymbolFont{mathx}{U}{mathx}{m}{n}

\DeclareMathSymbol{\obot}         {2}{matha}{"6B}

\newcommand{\kb}{{\bar\kappa}}

\newcommand{\bX}{\mathbb{X}}

\newcommand{\bV}{\mathbb{V}}
\newcommand{\bF}{\mathbf{F}}
\newcommand{\bY}{\mathbb{Y}}

\newcommand{\bP}{\mathbb{P}}

\newcommand{\bW}{\mathbb{W}}
\newcommand{\bA}{\mathbb{A}}
\newcommand{\bfV}{\mathbf{V}}

\newcommand{\rF}{\mathrm{F}}
\newcommand{\Gr}{\mathrm{Gr}}
\newcommand{\SGr}{\mathrm{SGr}}

\newcommand{\Fil}{\mathrm{Fil}}

\newcommand{\bL}{\mathbb{L}}

\newcommand{\Q}{\mathbb{Q}}
\newcommand{\Z}{\mathbb{Z}}

\newcommand{\C}{\mathbb{C}}

\newcommand{\F}{\mathbb{F}}

\newcommand{\spa}{\mathrm{Span}}

\newcommand{\Oo}{\mathcal{O}}

\newcommand{\cZ}{\mathcal{Z}}
\newcommand{\cD}{\mathcal{D}}

\newcommand{\cH}{H}

\newcommand{\tZ}{\tilde{\mathcal{Z}}}

\newcommand{\cG}{\mathcal{G}}
\newcommand{\cU}{\mathcal{U}}
\newcommand{\cN}{\mathcal{N}}

\newcommand{\cM}{\mathcal{M}}

\newcommand{\cF}{\mathcal{F}}
\newcommand{\Exc}{\mathrm{Exc}}

\newcommand{\GL}{\mathrm{GL}}

\newcommand{\Nilp}{\mathrm{Nilp}\, }

\newcommand{\Spec}{\mathrm{Spec}\, }
\newcommand{\Spf}{\mathrm{Spf}\, }
\newcommand{\SpfOF}{{\mathrm{Spf}\,\mathcal{O}_{\breve{F}} }}

\newcommand{\q}{q}
\newcommand{\pden}{\partial \mathrm{Den}}
\newcommand{\red}{\mathrm{red}}

\newcommand{\diag}{\mathrm{Diag}}

\newcommand{\Pden}{\mathrm{Pden}}
\newcommand{\ppden}{\partial \mathrm{Pden}}
\newcommand{\den}{\mathrm{Den}}

\newcommand{\m}{\mathrm{m}}

\newcommand{\ff}{\mathrm{if }}

\newtheorem{theorem}{Theorem}[section]
\newtheorem{corollary}[theorem]{Corollary}
\newtheorem{proposition}[theorem]{Proposition}
\newtheorem{lemma}[theorem]{Lemma}
\newtheorem{remark}[theorem]{Remark}

\theoremstyle{definition}
\newtheorem{definition}[theorem]{Definition}
\newtheorem{example}[theorem]{Example}

\numberwithin{equation}{section}

\title{A proof of the Kudla--Rapoport conjecture for Kr\"amer models}

\author[Qiao He]{Qiao He}
\address{Department of Mathematics, University of Wisconsin Madison, Van Vleck Hall,
Madison, WI 53706, USA}
\email{qhe36@wisc.edu}

\author[Chao Li]{Chao Li}
\address{Columbia University, Department of Mathematics, 2990 Broadway,	New York, NY 10027, USA}
\email{chaoli@math.columbia.edu}

\author[Yousheng Shi]{Yousheng Shi}
\address{Department of Mathematics, University of Wisconsin Madison, Van Vleck Hall,
Madison, WI 53706, USA}
\email{shi58@wisc.edu}

\author[Tonghai Yang]{Tonghai Yang}
\address{Department of Mathematics, University of Wisconsin Madison, Van Vleck Hall,
Madison, WI 53706, USA}
\email{thyang@math.wisc.edu}

\date{\today}

\begin{document}

\begin{abstract}
We prove the Kudla--Rapoport conjecture for Kr\"amer models of unitary Rapoport--Zink spaces at ramified places. It is a precise identity between arithmetic intersection numbers of special cycles on Kr\"amer models and modified derived local densities of hermitian forms. As an application, we relax the local assumptions at ramified places in the arithmetic Siegel--Weil formula for unitary Shimura varieties, which is in particular applicable to unitary Shimura vartieties associated to unimodular hermitian lattices over imaginary quadratic fields.
\end{abstract}

\maketitle{}

\tableofcontents{}

\section{Introduction}

\subsection{Background}
The classical \emph{Siegel--Weil formula} (\cite{Sie35,Siegel1951,Weil1965}) relates certain Siegel Eisenstein series to the arithmetic of quadratic forms, namely it expresses special \emph{values} of these series as theta functions --- generating series of representation numbers of quadratic forms. Kudla (\cite{Kudla97, Kudla2004}) initiated an influential program to establish the \emph{arithmetic Siegel--Weil formula} relating certain Siegel Eisenstein series to objects in arithmetic geometry, which among others, aims to express the \emph{central derivative} of these series as the arithmetic analogue of theta functions --- generating series of arithmetic intersection numbers of $n$ special divisors on Shimura varieties associated to $\SO(n-1,2)$ or $\U(n-1,1)$.

For  $\U(n-1,1)$--Shimura varieties with hyperspecial level at an \emph{unramified} place,  Kudla--Rapoport  \cite{KR1} conjectured a local arithmetic Siegel--Weil formula, now known as the (local) \emph{Kudla--Rapoport conjecture}. It is a precise identity between the \emph{central derivative} of local representation densities of hermitian forms (the analytic side) and the arithmetic intersection number of special cycles on unitary Rapoport--Zink spaces (the geometric side). This conjecture was recently proved by Zhang and one of us \cite{LZ}, and we refer to the introduction of \cite{LZ} for more background and related results.

It is a natural question, which is also important for global applications, to formulate and prove an analogue of the Kudla--Rapoport conjecture at a \emph{ramified} place. At a ramified place,
there are two well-studied level structures for unitary Rapoport--Zink spaces, one gives rise to the \emph{exotic smooth model} which has good reduction, and the other one gives rise to the \emph{Kr\"amer model} which has bad (semistable) reduction. For the even dimensional exotic smooth model, the analogue of Kudla--Rapoport conjecture was formulated and proved by Liu and one of us \cite{LL2} using a strategy similar to \cite{LZ}.

For the Kr\"amer model, however, the situation is more complicated --- it is expected that the analytic side of the conjecture requires nontrivial modification, by a certain linear combination of \emph{central values} of local representation densities. The necessity of such modification in the presence of bad reduction was first discovered by Kudla--Rapoport  \cite{KRshimuracurve}   via explicit computation in the context of the Drinfeld $p$-adic half plane. In \cite{HSY3}, three of us formulated \emph{the Kudla--Rapoport conjecture for Kr\"amer models} (recalled in \S\ref{sec:kudla-rapop-conj}) by providing a conceptual recipe for the precise modification needed for the analytic side. Moreover, this conjecture was proved for $n=2$ (based on the previous works \cite{Shi2,HSY}) and $n=3$ in \cite{HSY3}.

The main theorem of the present paper settles this conjecture for any $n$ (and the proof is new even for $n=2,3$).  As a first application, we relax the local assumptions in the arithmetic Siegel--Weil formula for $\U(n-1,1)$--Shimura varieties by allowing Kr\"amer models at ramified places.  The main theorem should also be useful to  relax the local assumptions at ramified places in the arithmetic inner product formula \cite{LL2020,LL2} and its $p$-adic avatar by Disegni--Liu \cite{DL22}.

\subsection{Kudla--Rapoport conjecture for Kr\"amer models}\label{sec:kudla-rapop-conj}

Let $p$ be an odd prime. Let $F_0$ be a finite extension of $\mathbb{Q}_p$ with residue field $\kappa=\mathbb{F}_q$. Let $F$ be a ramified quadratic extension of $F_0$. Let $\pi$ be a uniformizer of $F$ such that $\tr_{F/F_0}(\pi)=0$. Then $\pi_0=\pi^2$ is a uniformizer of $F_0$. Let $\breve F$ be the completion of the maximal unramified extension of $F$. Let $ O_F, \OFb$ be the ring of integers of $F,\Fb$ respectively.

Let $n\ge2$ be an integer. To define the Kr\"amer model of the unitary Rapoport--Zink space, we fix  a (principally polarized) supersingular hermitian $O_F$-modules $\mathbb{X}$ of signature $(1,n-1)$ over $\kb$ (Definition \ref{def:hermitian modules}). The \emph{Kr\"amer model} $\mathcal{N}=\mathcal{N}_n$ is the formal scheme over $\Spf O_{\breve F}$ parameterizing hermitian formal $O_F$-modules $X$ of signature $(1,n-1)$ within the quasi-isogeny class of $\mathbb{X}$, together with a rank 1 filtration $\mathcal{F}\subset \Lie X$ satisfying the Kr\"amer condition (Definition \ref{def:NKra}). The space $\mathcal{N}$ is locally of finite type, and semistable of relative dimension $n-1$ over $\Spf \OFb$. There are two choices of the framing object $\mathbb{X}$ (up to quasi-isogeny), giving rise to two non-isomorphic (resp. isomorphic) spaces $\mathcal{N}$ when $n$ is even (resp. odd) (\S\ref{subsec:associated hermitian space}).

Let $\mathbb{Y}$ be the framing hermitian $O_F$-modules of signature $(0,1)$ over $\kb$ defined as in Definition~\ref{def:hermitian modules}. The \emph{space of quasi-homomorphisms} $\mathbb{V}=\mathbb{V}_n\coloneqq \Hom_{O_F}(\mathbb{Y}, \mathbb{X}) \otimes_{O_F}F$ carries a natural $F/F_0$-hermitian form, which makes $\mathbb{V}$ a non-degenerate $F/F_0$-hermitian space of dimension $n$ (\S\ref{subsec:associated hermitian space}). The two choices of the framing object $\mathbb{X}$ exactly correspond to the two isomorphism classes of $\mathbb{V}$, classified by $\chi(\mathbb{V}):=\chi((-1)^{\frac{n(n-1)}2}\det(\mathbb{V}))\in\{\pm1\}$, where $\chi: F_0^\times\rightarrow \{\pm1\}$ is the quadratic character associated to $F/F_0$. For any subset $L\subset \mathbb{V}$, the \emph{special cycle} $\mathcal{Z}(L)$ (\S\ref{subsec:specialcycles}) is a closed formal subscheme of $\mathcal{N}$, over which each quasi-homomorphism $x\in L$ deforms to a homomorphism.

Let $L\subset \mathbb{V}$ be an $O_F$-lattice (of full rank $n$). We will associate to $L$ two integers: the \emph{arithmetic intersection number} $\Int(L)$ and the \emph{modified derived local density} $\pDen(L)$.

\begin{definition}
Let $L\subset \mathbb{V}$ be an $O_F$-lattice. Let $x_1,\ldots, x_n$ be an $O_F$-basis of $L$. Define the \emph{arithmetic intersection number}
\begin{equation}
  \label{eq:IntL}
  \Int(L)\coloneqq \chi(\mathcal{N},\mathcal{O}_{\mathcal{Z}(x_1)} \otimes^\mathbb{L}\cdots \otimes^\mathbb{L}\mathcal{O}_{\mathcal{Z}(x_n)} )\in \mathbb{Z},
\end{equation}
 where $\mathcal{O}_{\mathcal{Z}(x_i)}$ denotes the structure sheaf of the special divisor $\mathcal{Z}(x_i)$, $\otimes^\mathbb{L}$ denotes the derived tensor product of coherent sheaves on $\mathcal{N}$, and $\chi$ denotes the Euler--Poincar\'e characteristic (Definition \ref{def:Int L}). By  Howard \cite[Corollary D]{Ho2}), we know that $\Int(L)$ is independent of the choice of the basis $x_1,\ldots, x_n$ and hence is a well-defined invariant of $L$ itself.
\end{definition}

For  another hermitian $O_F$-lattice $M$ (of arbitrary rank), denote by $\Herm_{L,M}$ the $O_{F_0}$-scheme of hermitian $O_F$-module homomorphisms from $L$ to $M$ (Definition \ref{def: Den}) and define its \emph{local density} to be $$\Den(M,L)\coloneqq \lim_{d\rightarrow +\infty}\frac{|\Herm_{L,M}(O_{F_0}/\pi_0^{d})|}{q^{N\cdot d_{L,M}}},$$ where $d_{L,M}$ is the dimension of $\Herm_{L,M} \otimes_{O_{F_0}}F_0$. Let $H$ be the standard hyperbolic hermitian $O_F$-lattice of rank 2 (given by the hermitian matrix $\left(\begin{smallmatrix}0 & \pi^{-1} \\ -\pi^{-1} & 0\end{smallmatrix}\right)$). It is well-known that  there exists a \emph{local density polynomial} $\den(M,L,X)\in \mathbb{Q}[X]$  such that for any integer $k\ge0$,
\begin{equation}\label{eq:Den poly introduction}
	\Den(M, L, \q^{-2k})=\Den(H^k\obot M,L).
\end{equation}
        When $M$  has also rank $n$ and $\chi(M)=-\chi(L)$, we have $\Den(M,L)=0$ (Lemma \ref{lem: 0 vanish of error term}) and in this case we write $$\den'(M, L):=-2\cdot \frac{\rd}{\rd X}\bigg|_{X=1} \den(M,L, X).$$ Define the (normalized) \emph{derived local density}
        \begin{equation}\label{eq: def of Den'}
        \Den'(L):=\frac{\Den'(I_n, L)}{\Den(I_n,I_n)}\in \mathbb{Q}.\end{equation}
        Here $I_n$ is the unimodular lattice of rank $n$ with $\chi(I_n)=-\chi(L)$. Recall that a hermitian $O_F$-lattice $L$ is \emph{unimodular}\footnote{We refrain from using the terminology \emph{self-dual} in the ramified case to avoid possible confusion with a lattice $L$ such that $L=L^\vee$, where $L^\vee$ is the dual lattice with respect to the underlying quadratic form, see \S\ref{subsec:hermfourier}.} if $L=L^\sharp$, where $L^\sharp$ is the dual lattice of $L$ with respect to the hermitian form (see \S\ref{subsec: notation} for notation).

        The naive analogue of the Kudla--Rapoport conjecture for Kr\"amer model states that $$\Int(L)\stackrel{?}=\Den'(L).$$ However, as explained in \cite{HSY3} this naive analogue does not hold for trivial reasons. In fact, by definition $\Int(L)$ vanishes unless $L$ is integral (i.e., $L\subset L^\sharp$), while $\Den'(L)$ does not vanish for non-integral lattices $L$ which are dual to vertex lattices. More precisely, recall that an integral $O_F$-lattice $\Lambda\subset \mathbb{V}$ is called a \emph{vertex lattice (of type $t$)}  if $\Lambda^\sharp/\Lambda$ is a $\kappa$-vector space (of dimension $t$). For a vertex lattice $\Lambda\subset \mathbb{V}$ of type $t>0$,  $\Lambda^\sharp$ is non-integral and so  $\Int(\Lambda^\sharp)=0$, while $\Den'(\Lambda^\sharp)\ne0$ in general (see e.g. \eqref{eq: Den' Lambda tmax}). In general, we define  the type $t(L)$ of $L$ to be the number of positive fundamental invariants of $L$ (see \S \ref{subsec: notation}).

        To account for these discrepancies, we will define $\pDen(L)$ by modifying $\Den'(L)$ with a linear combination of the (normalized) \emph{local densities} (Corollary \ref{cor: int of pden_{2t}})
        \begin{equation}
          \label{eq: def of den(L)}
          \Den_{t}(L):=\frac{\Den(\Lambda_t^\sharp, L)}{\Den(\Lambda_t^\sharp,\Lambda_t^\sharp)}\in \mathbb{Z}.
        \end{equation}
 Here $\Lambda_t\subset \mathbb{V}$ is a vertex lattice of type $t$  (in particular $\chi(\Lambda_t^\sharp)=\chi(L)$). Recall that the possible vertex type $t$ is given by any even integer such that $0\le t\le t_\mathrm{max}$, where $$t_\mathrm{max}=
        \begin{cases}
          n, &\ff \, n\text{ even, }  \chi(\mathbb{V})=+1, \\
          n-1, &\ff \,  n\text{ odd}, \\
          n-2, &\ff \, n\text{ even, } \chi(\mathbb{V})=-1.
        \end{cases}$$

 \begin{definition}
Let $L\subset \mathbb{V}$ be an $O_F$-lattice. Define the \emph{modified derived local density} (Corollary \ref{cor: int of pden})
\begin{equation}
  \label{eq: def of pdenL}
  \pDen(L):=\Den'(L)+\sum_{j=1}^{t_\mathrm{max}/2}c_{2j}\cdot\Den_{2j}(L)\in \mathbb{Z}.
\end{equation}
 Here the coefficients $c_{2j}\in\mathbb{Q}$ are chosen to satisfy
 \begin{equation} \label{eq:coeff}
\pDen(\Lambda_{2i}^\sharp)=0,\quad 1\le i\le t_\mathrm{max}/2,
\end{equation}
which turns out to be a linear system in $(c_2,c_4,\ldots,c_{t_\mathrm{max}})$ with a unique solution  (\cite[Theorem 6.1]{HSY3}.
 \end{definition}

 The main purpose of this paper is to  prove the following local arithmetic Siegel-Weil formula, settling the main conjecture  of \cite{HSY3}. We will prove this theorem in \S\ref{sec: pf of main theorem}.

\begin{theorem}[Kudla--Rapoport conjecture for Kr\"amer models]\label{conj:main}
  Let $L\subset \mathbb{V}$ be an $O_F$-lattice. Then $$\Int(L)=\pDen(L).$$
\end{theorem}

\subsection{The arithmetic Siegel--Weil formula}\label{subsec:global introduction} Next let us describe some global applications of our main theorem, following the setting of \cite[\S1.3]{LZ}. We now switch to global notations. Let $F$ be a CM number field with maximal total real subfield $F_0$.  Fix an embedding $\overline{\mathbb{Q}}\hookrightarrow \mathbb{C}$ and  fix  a CM type $\Phi\subset \Hom(F, \overline{\mathbb{Q}}) =\Hom(F, \C) $ of $F$. We also identify the CM type $\Phi$ with the set of  archimedean places of $F_0$. Let $V$ be an $F/F_0$-hermitian space of dimension $n$ and $G=\Res_{F_0/\mathbb{Q}}\U(V)$. Assume the signatures of $V$ are $\{(n-1,1)_{\phi_0},(n,0)_{\phi\in\Phi-\{\phi_0\}}\}$ for some distinguished element $\phi_0\in \Phi$. Define a torus $Z^\mathbb{Q}=\{z\in \Res_{F/\mathbb{Q}}\mathbb{G}_m: \mathrm{Nm}_{F/F_0}(z)\in \mathbb{G}_m\}$. Associated to $\wit G\coloneqq Z^\mathbb{Q}\times G$ there is a natural Shimura datum $(\wit G,\{h_{\wit G}\})$ of PEL type (\cite[\S11.1]{LZ}). Let $K=K_{Z^\mathbb{Q}}\times K_G\subset \wit G(\mathbb{A}_f)$ be a compact open subgroup. Then the associated Shimura variety $\Sh_{\KG}=\Sh_{\KG}(\wit G,\{h_{\wit G}\})$ is of dimension $n-1$ and has a canonical model over its reflex field $E$.

Assume that $K_{Z^\mathbb{Q}}\subset Z^\mathbb{Q}(\mathbb{A}_f)$ is the unique maximal open compact subgroup. Assume that $K_G=\prod_v K_{G,v}$, where $v$ runs over the finite places of $F_0$ such that $K_{G,v}\subset \U(V)(F_{0,v})$ is given by
\begin{itemize}
\item the stabilizer of a self-dual or almost self-dual lattice $\Lambda_v\subset V_v$ if $v$ is inert in $F$,
\item the stabilizer of a unimodular lattice $\Lambda_v\subset V_v$ if $v$ is ramified in $F$,
\item a principal congruence subgroup of $\U(V)(F_{0,v})\simeq\GL_n(F_{0,v})$ if $v$ is split in $F$.
\end{itemize}
Let $\mathcal{V}_\mathrm{ram}$ (resp. $\mathcal{V}_\mathrm{asd}$) be the set of finite places $v$ of $F_0$ such that $v$ is ramified in $F$ (resp. $v$ is inert in $F$ and $\Lambda_v$ is almost self-dual).
Further assume that all places of $E$ above $\mathcal{V}_\mathrm{ram}\cup\mathcal{V}_\mathrm{asd}$ are unramified over $F$. Then we obtain a global regular integral model $\mathcal{M}_K$ of $\Sh_K$ over $O_E$
as in \cite[\S14.1-14.2]{LZ}, which is semistable at all places of $E$ above $\mathcal{V}_\mathrm{ram}\cup \mathcal{V}_\mathrm{asd}$ (for more precise technical conditions required, see \ref{item:G0}-\ref{item:G5}). When $K_G$ is the stabilizer of a global unimodular lattice, the regular integral model $\mathcal{M}_K$ recovers that  in \cite{HSY3}, and that in  \cite{BHKRY1} if $F_0=\mathbb{Q}$.

Let $\mathbb{V}$ be the \emph{incoherent} hermitian space over $\mathbb{A}_{F}$ associated to $V$, namely $\mathbb{V}$ is totally positive definite and $\mathbb{V}_v \cong V_v$ for all finite places $v$. Let $\varphi_K\in \sS(\mathbb{V}^n_f)$ be a $K$-invariant (where $K$ acts on $\mathbb{V}_f$ via the second factor $K_G$) factorizable Schwartz function such that $\varphi_{K,v}=\mathbf{1}_{(\Lambda_v)^n}$ at all $v$ nonsplit in $F$. Let $T\in \Herm_n(F_0)$ be a nonsingular $F/F_0$-hermitian matrix of size $n$. Associated to $(T,\varphi_K)$ we have \emph{arithmetic special cycles} $\mathcal{Z}(T,\varphi_K)$ over $\mathcal{M}_K$ (\cite[\S14.3]{LZ}) generalizing the $\mathcal{Z}(T)$ in  \cite{KR2}. Analogous to the local situation (\ref{eq:IntL}), we can define its local arithmetic intersection numbers $\Int_{T,v}(\varphi_K)$ at finite places $v$. Using the star product of Kudla's Green functions, we can also define its local arithmetic intersection number  $\Int_{T,v}(\sy,\varphi_K)$ at infinite places (\cite[\S15.3]{LZ}), which depends on a parameter $\sy\in \Herm_n(F_{\infty})_{>0}$  where  $F_{\infty}=F \otimes_{\mathbb{Q}}\mathbb{R}$. Combining all the local arithmetic numbers together, define the \emph{global arithmetic intersection number}, or the \emph{arithmetic degree} of the special cycle $\mathcal{Z}(T,\varphi_K)$ in the arithmetic Chow group of $\mathcal{M}_K$, $$\wdeg_T(\sy,\varphi_K)\coloneqq \sum_{v\nmid\infty}\Int_{T,v}(\varphi_K)+\sum_{v\mid \infty}\Int_{T,v}(\sy,\varphi_K).$$

On the other hand, associated to $\varphi:=\varphi_K \otimes \varphi_\infty\in\sS(\mathbb{V}^n)$, where $\varphi_{\infty}$ is the Gaussian function, there is a classical \emph{incoherent Eisenstein series} $E(\sz, s,\varphi)$ (\cite[\S12.4]{LZ}) on the hermitian upper half space $$\mathbb{H}_n=\{\sz=\sx+i\sy:\ \sx\in\Herm_n(F_{\infty}),\ \sy\in\Herm_n(F_{\infty})_{>0}\}.$$ This is essentially the Siegel Eisenstein series associated to a standard Siegel--Weil section of the degenerate principal series (\cite[\S12.1]{LZ}). The Eisenstein series here has a meromorphic continuation and a functional equation relating $s\leftrightarrow -s$.  The central value $E(\sz, 0, \varphi)=0$ by the incoherence. We thus consider its \emph{central derivative} $$\Eis'(\sz, \varphi_K)\coloneqq \frac{\rd}{\rd s}\bigg|_{s=0}E(\sz, s,\varphi).$$

Analogous to the local situation, we need to modify $\Eis'(\sz, \varphi_K)$ by central values of coherent Eisenstein series. For $v\in\mathcal{V}_\mathrm{ram}\cup\mathcal{V}_\mathrm{asd}$, let  $^v\mathbb{V}$ be the \emph{coherent} hermitian space over $\mathbb{A}_{F}$ nearby $\mathbb{V}$ at $v$, namely $(^v\mathbb{V})_w\simeq \mathbb{V}_w$ exactly for all places $w\ne v$. For any vertex lattice $\Lambda_{t,v} \subset (^v\mathbb{V})_v$ of type $t$, the Schwartz function $\varphi^v \otimes \mathbf{1}_{(\Lambda_{t,v}^{\sharp})^n}\in\sS((^v\mathbb{V})^n)$ gives a classical \emph{coherent Eisenstein series} $E(\sz,s,\varphi^v \otimes \mathbf{1}_{(\Lambda_{t,v}^{\sharp})^n})$. Analogous to (\ref{eq: def of den(L)}),  define the (normalized) \emph{central values}
\begin{equation}\label{eq:central value}
   ^v\Eis_t(\sz,\varphi_K):=\frac{\vol(K_{G,v})}{\vol(K_{\Lambda_{t,v}^{\sharp}})}\cdot E(\sz,0,\varphi^v \otimes \mathbf{1}_{(\Lambda_{t,v}^{\sharp})^n}).
\end{equation}
Here $K_{\Lambda_{t,v}^{\sharp}}\subset \U(^v\mathbb{V})(F_{0,v})$ is the stabilizer of $\Lambda_{t,v}^{\sharp}$, and the volumes are taken with respect to the Haar measures on $\U(V)(F_{0,v})$ and $\U(^v\mathbb{V})(F_{0,v})$ as defined in
\cite[Definition 3.8]{LL2020}. When $v\in\mathcal{V}_\mathrm{ram}$, analogous to (\ref{eq: def of pdenL}), define the linear combination
\begin{equation}
   ^v\Eis(\sz,\varphi_K):=\sum_{j=1}^{t_{\mathrm{max},v}/2}c_{2j,v}\cdot {}^v\Eis_{2j}(\sz,\varphi_K)\cdot\log q_v,
\end{equation}
where $q_v$ is the size of the residue field of $F_{0,v}$, and $t_{\mathrm{max},v}$ and $c_{2j,v}$ are the numbers $t_\mathrm{max}$ and $c_{2j}$ respectively in (\ref{eq: def of pdenL}) for the local hermitian space $(^v\mathbb{V})_v$ over the ramified extension $F_v/F_{0,v}$. When $v\in\mathcal{V}_\mathrm{asd}$, define
\begin{equation}\label{eq: modified asd term}
   ^v\Eis(\sz,\varphi_K):=c_{0,v}\cdot{}^v\Eis_{0}(\sz,\varphi_K)\cdot \log q_v^2,
\end{equation}
where $c_{0,v}=-\frac{1}{1+q_v}$. Define the \emph{modified central derivative}
\begin{equation}\label{eq:modified central derivative}
\pEis(\sz,\varphi_K):=\Eis'(\sz,\varphi_K)+(-1)^n\sum_{v\in\mathcal{V}_\mathrm{ram}\cup\mathcal{V}_\mathrm{asd}}{}^v\Eis(\sz,\varphi_K).
\end{equation}
Associated to an additive character $\psi: \mathbb{A}_{F_0}/F_0\rightarrow \mathbb{C}^\times$ (as explained in \cite[\S 12.2]{LZ} we assume that $\psi$ is unramified outside the set of finite places of $F_0$ split in $F$), it has a decomposition into Fourier coefficients
\begin{equation}\label{eq:Fourier coefficients of central derivative}
    \pEis(\sz,\varphi_K)=\sum_{T\in \Herm_n(F_0)}\pEis_T(\sz,\varphi_K).
\end{equation}

The following result asserts an identity between the arithmetic degrees of special cycles and the nonsingular Fourier coefficients of the modified central derivative of the incoherent Eisenstein series, which generalizes \cite[Theorem 1.3.1]{LZ} from inert places to all nonsplit places. In  particular, when
 $F$ is an imaginary quadratic field of  discriminant $d \equiv 1 \pmod 8$, we have an unconditional arithmetic Siegel-Weil formula for all unimodular lattices of signature $(n-1, 1)$ at non-singular coefficients, i.e.,  \cite[Theorem 1.5]{HSY3} holds without conditions.

\begin{theorem}[Arithmetic Siegel--Weil formula: nonsingular terms]
Let $\Diff(T, \mathbb{V})$ be the set of places $v$ such that $\mathbb{V}_{v}$ does not represent $T$ (\cite[\S12.3]{LZ}).  Let $T\in\Herm_n(F_0)$ be nonsingular such that $\Diff(T,\mathbb{V})=\{v\}$ where $v$ is nonsplit in $F$ and not above 2. Then $$\wdeg_T(\sy, \varphi_K)q^T=c_K\cdot \pEis_T(\sz,\varphi_K),$$ where $q^T\coloneqq\psi_\infty(\tr T\sz)$, $c_K$ is a nonzero constant independent of $T$ and $\varphi_K$ (to be specified in Theorem \ref{thm:Arithmetic Siegel--Weil formula: nonsingular terms}).
\end{theorem}

We form the \emph{generating series of arithmetic degrees}
\begin{equation}\label{eq:generating series of arithmetic degree}
    \wdeg(\sz, \varphi_K)\coloneqq \sum_{T\in\Herm_n(F_0)\atop \det T\ne0}\wdeg_T(\sy,\varphi_K) q^T.
\end{equation}
The following result relates this generating series to the modified central derivative of the incoherent Eisenstein series, which removes the assumption that $F/F_0$ is unramified at all finite places from \cite[Theorem 1.3.2]{LZ}.

\begin{theorem}[Arithmetic Siegel--Weil formula]
  Assume that $F/F_0$ is split at all places above 2. Further assume that $\varphi_K$ is nonsingular (\cite[\S12.3]{LZ}) at two places split in $F$. Then $$\wdeg(\sz, \varphi_K)=c_K\cdot \pEis(\sz,\varphi_K).$$ In particular, $\wdeg(\sz, \varphi_K)$ is a nonholomorphic hermitian modular form of genus $n$.
\end{theorem}

\subsection{Strategy and novelty of the proof of the Main Theorem \ref{conj:main}}\label{sec:whats-new} Our general strategy is closest to the unramified orthogonal case proved in \cite{LZ2}.   More precisely, fix an $O_F$-lattice $L^\flat\subset \mathbb{V}$ of rank $n-1$ and denote by $\mathbb{W}=(L^\flat_F)^\perp\subset \mathbb{V}$. Consider functions on $\mathbb{V}\setminus L^\flat_F$, $$\Int_{L^\flat}(x)\coloneqq \Int(L^\flat+\langle x\rangle),\quad \pDen_{L^\flat}(x)\coloneqq \pDen(L^\flat+\langle x\rangle).$$ Then it remains to show the equality of the two functions $\Int_{L^\flat}=\pDen_{L^\flat}$. To show this equality, we find a decomposition $$\Int_{L^\flat}=\Int_{L^\flat,\sH}+\Int_{L^\flat,\sV},\quad \pDen_{L^\flat}=\pDen_{L^\flat,\sH}+\pDen_{L^\flat,\sV}$$ into ``horizontal'' and ``vertical'' parts such that the horizontal identity $\Int_{L^\flat,\sH}=\pDen_{L^\flat,\sH}$ holds and that the vertical parts $\Int_{L^\flat,\sV}$ and $\pDen_{L^\flat,\sV}$ behaves well under Fourier transform along $L^\flat_F$.

The horizontal identity essentially reduces to the horizontal computation for $n=2$ in \cite{Shi2,HSY}. For the vertical identity, we perform a \emph{partial Fourier transform} along $L^\flat_F$ and consider new functions on $\mathbb{W}\setminus\{0\}$, $$\Intp(x):=\int_{L^\flat_F}\Int_{L^\flat, \sV}(y+x)\rd y,\quad \pDenp(x):=\int_{L^\flat_F}\pDen_{L^\flat, \sV}(y+x)\rd y.$$  The key is to show that $\Intp$ and $\pDenp$ are both \emph{constant} on $\mathbb{W}^{\ge0}\setminus\{0\}:=\{x\in \mathbb{W}\setminus \{0\}: \val(x)\ge0\}$ of $\mathbb{W}$  (see \S\ref{subsec: notation} for notation) as in Theorem \ref{lem:Int vertical is schwartz} and Theorem \ref{prop: part FT of denLflat}. Using an induction on the valuation of $L^\flat$, we show that the difference $\Intp-\pDenp$ vanishes on $\mathbb{W}^{\le0}:=\{x\in \mathbb{W}: \val(x)\le0\}$, and hence it vanishes identically and allows us to conclude that $\Int_{L^\flat,\sV}=\pDen_{L^\flat,\sV}$.

On the geometric side, we prove a \emph{Bruhat--Tits stratification} for the Kr\"amer model (Theorem \ref{thm:Bruhat Tits stratification}), analogous to the case of the Pappas model treated in Rapoport--Terstiege--Wilson \cite{RTW}. We make use of the linear invariance of special cycles \cite{Ho2} to express $\Int_{L^\flat,\sV}$ as a linear combination of functions on $\mathbb{V}$ which are \emph{translation invariant} under vertex lattices.  A new observation in our ramified case is that the translation invariance already allows us to control the support of its Fourier transform well enough (Lemma \ref{lem:Lambda invariant function is geq -1}) to conclude the desired key constancy of $\Intp$ on $\mathbb{W}^{\ge0}$. Compared to the unramified case, we \emph{completely avoid} the Tate conjecture of generalized Deligne-Lusztig varieties and explicit computation of their intersections with special divisors. It is not clear that the Deligne-Lusztig subvarieties span the Tate classes in this case.  \S\ref{sec:Bruhat-Tits} studies the structure of $\cN_{\mathrm{red}}$ and special cycles, and   should be of independent interest (in addition to preparation for \S\ref{sect:FourierGeo}).

On the analytic side, we make use of the \emph{primitive decomposition} of the local density polynomial into primitive local density polynomials and obtain a decomposition.
\begin{equation}
\label{eq:decompostionP}\pDen(L)=\sum_{L\subset L'}\pPden(L'),
\end{equation}
where $L'$ runs over $O_F$-lattices in $L_F$ containing $L$, and the symbol $\Pden$ stands for the primitive version of $\Den$ (Corollary \ref{cor: decomp of pden(L)}). Unlike the unramified or exotic smooth case, the primitive local density \emph{polynomial} itself seems rather complicated (see e.g. Corollary \ref{cor: Pden}).
Nevertheless we manage to prove a simple formula for its \emph{modified central derivative} $\pPden(L)$, which we find quite striking.
\begin{theorem}[Theorem \ref{thm: formula of ppden}]
Let $L\subset \bV$ be an $O_F$-lattice (of full rank $n$).
    \begin{enumerate}
        \item If $L$ is not integral, then
        $\ppden(L)=0.$
        \item If $L$ is unimodular, then $$\ppden(L)=\begin{cases}
        1, & \text{if $n$ is odd},\\
        0, & \text{if $n$ is even}.
        \end{cases}$$
        \item If $L$ is integral and of type $t>0$, then
    $$\ppden(L)=\begin{cases}
		\displaystyle\prod_{\ell=1}^{\frac{t-1}{2}}(1-\q^{2\ell}), & \text{  if $t$ is odd},\\
				\displaystyle(1-\chi(L')\q^{\frac{t}{2}})\prod_{\ell=1}^{\frac{t}{2}-1}(1-q^{2\ell}), & \text{ if  $t$ is even}.
			\end{cases}$$ Here we write $L\simeq I_{n-t}\obot L'$ with $I_{n-t}$ unimodular of rank $n-t$.
    \end{enumerate}
\end{theorem}

The proof of this theorem occupies the entire \S\ref{sec: formula of ppden} and \S\ref{sec:pPden}, and is our major technical innovation. One key difference between our case and the unramified or exotic smooth case is that in our case $I_n$ and $H$ (see \eqref{eq:Den poly introduction} and \eqref{eq: def of Den'}) have different fundamental invariants, hence it is not clear how to reduce the calculation of $\ppden$ into the embedding-counting problems over finite fields in the style of \cite[\S3]{CY}.
To deal with this difficulty, we first decompose $\ppden(L)$ according to orbits of Hermitian embeddings (Theorem \ref{thm: decom of Pden}).
Now a new observation is that the primitive local density polynomial becomes simpler when $L$ is ``very integral'' (i.e., when its fundamental invariants are all $\ge 1$, see Proposition \ref{prop: prim local density full type}) in which case the decomposition in Theorem \ref{thm: decom of Pden} is simple.  The primitive local density polynomial vanishes when $L$ is ``very non-integral'' (e.g.,  when one of its fundamental invariants is $\le-2$, see the proof of  Lemma \ref{lem: vanish val le -1}). When $L$ is the dual of a vertex lattice of positive type, this is just our assumption (\ref{eq:coeff}).  The remaining cases (in particular the unimodular lattice case) are much trickier  to handle, whose proof occupies most of \S\ref{sec:pPden} and is summarized in \S\ref{subsec:proofstra}. The proof relies on a series of non-trivial polynomial identities (e.g., Lemma \ref{lem: g(r,r,X)} and Lemma \ref{lem: sum of g})  involving algebraic combinatorics of quadratic spaces over finite fields, which should be of independent  interest.

With the simple formula for $\pPden(L)$ at hand, we finally prove the desired key constancy of $\pDenp$ on $\mathbb{W}^{\ge0}\setminus \{0\}$ via involved lattice-theoretic computation in \S\ref{sec: partial FT}, in a fashion similar to \cite{LZ2}. The techniques developed here on the analytic side seem to have wide applicability and we hope that they may shed new light on the Kudla--Rapoport conjecture in the context of more general level structures, e.g., for minuscule parahoric levels at unramified places formulated by Cho \cite{Cho}.

\subsection{Notation and terminology}\label{subsec: notation}
In this paper, a lattice means a hermitian $O_F$-lattice without explicit mentioning. Unless otherwise stated, $L$ always means a non-degenerate lattice of rank $n$ with a hermitian form $(\, , \,)$.
\begin{itemize}
\item We say a sublattice $L^\flat$ of a hermitian space is non-degenerate if the restriction of the hermitian form to it is non-degenerate.
    \item We define $L^{\sharp}$ to be the dual lattice of $L$ with respect to the hermitian form $(\, , \,)$. If $L\subset L^\sharp$, we say $L$ is integral.
    \item
Following \cite[Definition 2.11]{LL2}, for a lattice $L$ with hermitian form $(\, ,\, )$, we say that a basis $\left\{\ell_{1}, \ldots, \ell_{n}\right\}$ of $L$ is a normal basis (which always exists by \cite[Lemma 2.12]{LL2}) if its moment matrix $T=\left(\left(\ell_{i}, \ell_{j}\right)\right)_{1\le i, j\le n}$ is conjugate to
\begin{align*}
\left(\beta_{1} \pi^{2 b_{1}}\right) \oplus \cdots \oplus\left(\beta_{s} \pi^{2 b_{s}}\right) \oplus\left(\begin{array}{cc}
0 & \pi^{2 c_{1}+1} \\
-\pi^{2 c_{1}+1} & 0
\end{array}\right) \oplus \cdots \oplus\left(\begin{array}{cc}
0 & \pi^{2 c_{t}+1} \\
-\pi^{2 c_{t}+1} & 0
\end{array}\right)
\end{align*}
by a permutation matrix, for some $\beta_{1}, \ldots, \beta_{s} \in O_{F_0}^{\times}$ and $b_{1}, \ldots, b_{s}, c_{1}, \ldots, c_{t} \in \mathbb{Z}$. Moreover, we define its (unitary) fundamental invariants $(a_1,\cdots,a_n)$ to be the unique nondecreasing rearrangement of $(2 b_1,\cdots,2b_s,2c_1+1,\cdots,2c_t+1)$.
\item We define the type $t(L)$ of $L$ to be the number of positive fundamental invariants of $L$.  We use $r(L)$ to denote the rank of $L$ and call $L$ a {\it full type} lattice if $r(L)=t(L)$.
\item We define the valuation of $L$ to be $\mathrm{val}(L)\coloneqq\sum_{i=1}^{n}a_i$, where $(a_1,\cdots,a_n)$ are the fundamental invariants of $L$.  For $x\in L$, we define $\val(x)=\val((x,x))$, where $\val(\pi_0)=1$.
\item For a hermitian space $\bV$, we let $\bV^{? i}\coloneqq \{x\in \bV\mid \val(x)? i\}$ where $?$ can be $\ge$, $\leq$ or $=$.
\item For a ring $R$, we use $\langle \ell_1,\cdots, \ell_n\rangle_{R}$ to denote $\mathrm{Span}_R\{\ell_1,\cdots,\ell_n\}$. When $R=O_F$, we simply write $\langle \ell_1,\cdots, \ell_n\rangle$. We use $L_F$ to denote $L\otimes_{O_F} F$.
\item  For a hermitian lattice of rank $n$, we define its sign as
	\begin{equation*}
		\chi(L) \coloneqq \chi((-1)^{\frac{n(n-1)}{2}}\mathrm{det}(L)) =\pm 1
	\end{equation*}
	where $\chi$ is the quadratic character of $F_0^\times$ associated to $F/F_0$. Without explicit mentioning, we always use $\epsilon$ to denote $\chi(L)$.
	\item Let $I_{m}^{\epsilon}$ denote a unimodular lattice of rank $m$ with $\chi(I_{m}^{\epsilon})=\epsilon$.  We also simply denote a unimodular lattice of rank $m$ by $I_m$ if we do not need to remember its sign or its sign is clear in the context. In particular, when we consider $\Den'(I_n,L)$, we mean $I_n=I_{n}^{-\epsilon}$.
	\item We call a sublattice $N\subset M$ primitive in $M$ if $\mathrm{dim}_{\mathbb F_\q}\overline{N}=r(N)$, where $\overline{N}=(N+\pi M)/\pi M$. We also use $\overline{L}$ to denote $L\otimes_{O_F} O_F/(\pi).$
	\item For two lattices $L,L'$ of same rank, let $n(L',L)=\#\{L''\subset L_F\mid L\subset L'', L''\cong L'\}. $
	\item  We let $\delta_{\mathrm{odd}}(n)=1$ if $n$ is an odd integer and $\delta_{\mathrm{odd}}(n)=0$ if $n$ is an even integer.
\end{itemize}

{\bf Acknowledgments:} C.~L.~was partially supported by the NSF grant DMS-2101157. T.Y. ~was partially supported  by the Dorothy Gollmar Chair's Fund and Van Vleck research fund. Q.H. and T.Y. are partially supported by a graduate school grant of  UW-Madison. We would like to thank the referee for the careful reading and many valuable suggestions. Part of the work was done while Q.H., C.L. and T.Y. attending the "Algebraic Cycles, $L$-Values, and Euler Systems`` program held in MSRI. We would like to thank MSRI for the excellent work condition, financial support,  and hospitality.

	\section{Kr\"amer models of Rapoport-Zink spaces and special cycles}\label{sec: RZ space}
	We denote $\bar{a}$ the Galois conjugate of $a\in F$ over $F_0$. Denote by  $\Nilp O_{\breve F}$ be the category of $O_{\breve F}$-schemes $S$ such that $\pi$ is locally nilpotent on $S$. For such an $S$, denote its special fiber $S\times_{\SpfOF} \Spec \bar\kappa$ by $\bar S$. Let $\sigma\in \mathrm{Gal}(\breve{F}_0/F_0)$ be the Frobenius element. We fix an injection of rings $i_0:O_{F_0}\rightarrow \Oo_{\breve F_0}$ and an injection $i:O_F\rightarrow O_{\breve F}$ extending $i_0$.
	Denote by $\bar{i}:O_F\rightarrow O_{\breve F}$ the map $a\mapsto i(\bar{a})$.
	
	\subsection{RZ spaces}\label{subsec:RZspaces}
	Let $S\in\Nilp O_{\breve F}$. A $p$-divisible strict $O_{F_0}$-module over $S$ is a $p$-divisible group over $S$ with an $O_{F_0}$-action whose induced action on its Lie algebra is via  $O_{F_0}\xrightarrow[]{i_0}O_{\breve {F_0}}\rightarrow\Oo_S$.
	\begin{definition}\label{def:hermitian modules}
		A formal hermitian $O_F$-module of dimension $n$ over $S$ is a triple $(X,\iota,\lambda)$ where $X$ is a supersingular $p$-divisible strict $O_{F_0}$-module over $S$  of dimension $n$ and $F_0$-height $2n$ (supersingular means the $O_{F_0}$-relative Dieudonn\'e module of $X$ at each geometric point of $S$ has slope $\frac{1}{2}$), $\iota:O_F\rightarrow \End(X)$ is an $O_F$-action and $\lambda:X\rightarrow X^\vee$ is a principal polarization in the category of strict $O_{F_0}$-modules such that the Rosati involution induced by $\lambda$ is the Galois conjugation of $F/F_0$ when restricted on $O_F$.
	\end{definition}

	\begin{definition}\label{def:NKra}
		Fix a formal hermitian $O_F$-module $(\bX,\iota_\bX,\lambda_\bX)$ of dimension $n$ over $\bar\kappa$. The moduli space $\cN_n$ is the functor such that $\cN_n(S)$ for any $S\in \Nilp O_{\breve F}$ is the set of isomorphism classes of quintuples $ (X,\iota,\lambda,\rho,\cF)$ such that
		\begin{enumerate}
		    \item $(X,\iota,\lambda)$ is a formal hermitian $O_F$-module over $S$.
		    \item $\rho:X\times_{S}\bar{S} \rightarrow \bX \times_{\Spec \bar\kappa} \bar{S}$ is a quasi-isogeny of formal $O_F$-modules of height $0$.
		    \item $\cF$ satisfies Kr\"amer's (\cite{Kr}) signature condition: it is a local direct summand of $\Lie X$ of rank $n-1$ as an $\Oo_S$-module such that $O_F$ acts on $\cF$ by  $O_F\xrightarrow[]{\bar i}O_{\breve F}\rightarrow \Oo_S$ and acts on $\Lie X/\cF$ by $O_F\xrightarrow[]{i}O_{\breve F}\rightarrow \Oo_S$.
		\end{enumerate}
		An isomorphism between two such quintuples $(X,\iota,\lambda,\rho,\cF)$ and $(X',\iota',\lambda',\rho',\cF')$ is an isomorphism $\alpha:X\rightarrow X'$ such that $\rho'\circ(\alpha\times_S \bar{S})=\rho $, $\alpha^*(\lambda')$ is a $O_{F_0}^\times$-multiple of $\lambda$ and $\alpha_*(\cF)=\cF'$.
	\end{definition}
	Notice that $\cN_n$ is a relative Rapoport-Zink space in the sense of \cite{mihatsch2022relative}.
	By \cite[Proposition 2.2]{Ho2}, $\cN_n$ is representable by a  flat formal scheme  of relative dimension $n-1$ over $\Spf O_{\breve F}$.  We drop the  subscript $n$ in $\cN_n$ when there is no ambiguity.

	\subsection{Associated hermitian spaces}\label{subsec:associated hermitian space}
	For a strict $O_{F_0}$-module $X$ over $\bar\kappa$, let $M(X)$ be the $O_{F_0}$-relative Dieudonn\'e module of $X$.
	Let $(\bX,\iota_\bX,\lambda_\bX)$ be the framing object as in Definition \ref{def:NKra}, and $N=M(\bX) \otimes_{O_{F_0}} F_0$ be its rational relative Dieudonne module.  Then $N$ is a $2n$-dimensional $\breve{F}_0$-vector space equipped with a $\sigma$-linear operator $\bF$ and a $\sigma^{-1}$-linear operator $\bfV$. The $O_F$-action  $\iota_\bX:O_F\rightarrow \End(\bX)$ induces on $N$ an $O_F$-action  commuting with $\bF$ and $\bfV$.  We still denote this induced action by $\iota_\bX$ and denote $\iota_\bX(\pi)$ by $\Pi$. The principal polarization of $\bX$ induces a skew-symmetric $\breve{F}_0$-bilinear form $\langle\, ,\, \rangle$ on $N$ satisfying
	\[\langle \bF x, y\rangle=\langle x, \bfV y\rangle^\sigma, \quad \langle \iota(a) x, y\rangle=\langle x, \iota(\bar{a}) y\rangle,\]
	for any $x,y\in N$, $a\in O_F$.
	Then $N$ is an $n$-dimensional $\breve F$-vector space equipped with a $\breve{F}/\breve{F}_0$-hermitian form $(\, ,\, )$ defined by (see \cite[(2.6)]{Shi1})
	\begin{equation}\label{eq:(,)}
	    (x,y):=\delta (\langle \Pi x, y\rangle+\pi\langle  x, y\rangle),
	\end{equation}
	where $\delta$ is a fixed element in $\Oo_{\breve F_0}^\times$ such that $\sigma(\delta)=-\delta$. We can use the relation
	\begin{equation}\label{eq:hermitian and symplectic form}
	    \langle x, y\rangle=\frac{1}{2\delta} \tr_{\breve{F}/\breve{F}_0} (\pi^{-1}(x,y)).
	\end{equation}
	to recover $\langle\, ,\, \rangle$. Let $\tau\coloneqq \Pi \bfV^{-1}$ and $C\coloneqq N^\tau$. Then $C$ is an $F$-vector space of dimension $n$ and $N=C\otimes_{F_0} \breve{F}_0$.
    The restriction of $(\, ,\, )$ to $C$ is a $F/F_0$-hermitian form which we still denote by $(\, ,\, )$. There are two choices of $(\bX,\iota_\bX,\lambda_\bX)$ up to quasi-isogenies preserving the polarization on the nose,  according to the sign $\epsilon=\chi(C)$ of $C$. Here $\chi:F_0^\times\rightarrow \{\pm 1\}$ is the character associated to the quadratic extension $F/F_0$ and we define  the sign of $C$ as
	\[\chi(C):=\chi((-1)^{n(n-1)/2}\det(C)).\]
	 When $n$ is odd, two different choices of $\epsilon$ give us isomorphic moduli spaces. When $n$ is even, two different choices of $\epsilon$ give us two  non-isomorphic moduli spaces. See \cite[Remark 2.16]{Shi1} and \cite[Remark 4.2]{RTW}.

	Fix a formal hermitian $O_F$-module $(\bY,\iota_\bY,\lambda_\bY)$ of dimension $1$ over $\Spec \bar\kappa$.  Define
	\begin{equation}
		\bV_n=\Hom_{O_F}(\bY,\bX)\otimes \Q.
	\end{equation}
	We drop the subscript $n$ of $\bV_n$ unless we need to specify the dimension.
	The vector space $\bV$ is equipped with a hermitian form $(\, ,\, )_\bV$ such that for any $x,y \in \bV$
	\begin{equation}\label{eq:h(x,y)}
		(x,y)_\bV=\lambda_\bY^{-1}\circ y^\vee \circ \lambda_\bX \circ x\in \End(\bY)\otimes_{\Z}\Q\overset{\sim}{\rightarrow}F
	\end{equation}
	where $y^\vee$ is the dual quasi-homomorphism of $y$.
	The hermitian spaces $(\bV,(\, ,\, )_\bV)$ and $(C,(, ))$ are related by the $F$-linear isomorphism
	\begin{equation}\label{eq: C V isomorphism}
		b: \bV\rightarrow C, \quad x\mapsto x(e)
	\end{equation}
	where $e$ is a generator of $\tau$-fixed points of the $O_{F_0}$-relative Dieudonn\'e module $M(\bY)$. The relative Dieudonn\'e module $M(\bY)$ is equipped with a hermitian form $(\, ,\, )_\bY$ such that $(e,e)_\bY\in O_{F_0}^\times$. By \cite[Lemma 3.6]{Shi1}, we have
	\begin{equation}\label{eq:h and (,)_X}
		(x,x)_\bV\cdot(e,e)_\bY=(b(x),b(x)).
	\end{equation}
    Here the bilinear form $(,)_{\bY}$ is the analogue of the form $(,)$ in \eqref{eq:(,)} defined on the rational relative Dieudonn\'e module of $\bY$.
	By scaling the polarization $\lambda_\bY$ by a factor in $O_{F_0}^\times$ we can assume that
	\[(e,e)_\bY=1,\]
	so $\bV$ and $C$ are isomorphic as hermitian spaces. When the context is clear we often drop the subscript $\bV$ in $(\, ,\, )_\bV$.

	\subsection{Special cycles}\label{subsec:specialcycles}
	We fix a canonical lift $(\cG,\iota_\cG,\lambda_\cG)$ of $(\bY,\iota_\bY,\lambda_\bY)$ to $O_{\breve F}$ in the sense of \cite{G} such that the action of $O_F$ on $\Lie \cG$ is via the inclusion $\bar i$. Such a lift is unique up to isomorphism by \cite[Proposition 2.1]{Ho2}.
	\begin{definition}
		For an $O_F$-lattice $L$ of $\bV$, define $\cZ(L)$ to be the subfunctor of $\cN$ such that $\cN(S)$ is the set of isomorphism classes of tuples $(X,\iota,\lambda,\rho,\cF)\in \cN(S)$ such that for any $x\in L$ the quasi-homomorphism
		\[\rho^{-1}\circ x\circ \rho_\cG: \bY\times_{\Spec \bar\kappa} \bar{S} \rightarrow X\times_{S} \bar{S}\]
		entends to a homomorphism $\cG\times_{\SpfOF} S\rightarrow X$.
		By Grothendieck-Messing theory $\cZ(L)$ is a closed formal subscheme of $\cN$.
		For $x\in \bV$, we let $\cZ(x)\coloneqq \cZ (L)$ when $L=\langle x\rangle$.
	\end{definition}

	\subsection{Bruhat-Tits stratification of $\cN_{\mathrm{red}}$}
	We say a lattice $\Lambda\subset C$ (resp. $\Lambda\subset \bV$) is a vertex lattice if $\pi \Lambda^\sharp\subset \Lambda \subset \Lambda^\sharp$ where $\Lambda^\sharp$ is dual lattice of $\Lambda$ with respect to $(\, ,\, )$ (resp. $(\, ,\, )_\bV$)\footnote{Notice that the vertex lattice $\Lambda$ in the sense of \cite{RTW}, \cite{HSY} or\cite{HSY3} corresponds to $\Lambda^\sharp$ in our convention.}.
	Using the isomorphism of hermitian spaces \eqref{eq: C V isomorphism}, we often identify $\Lambda$ with $b^{-1}(\Lambda)$ and use the same notation to denote both.
	We call $t=\mathrm{dim}_{\F_\q}(\Lambda^\sharp/\Lambda)$ the type of $\Lambda$.
	Recall from \cite[Lemma 3.2]{RTW} that $t$ has to be an even integer.
	To each vertex lattice $\Lambda$ of type $2m$, we can associate a subscheme $\cN_{\Lambda}$ which is a subscheme of the minuscule special cycle $\cZ(\Lambda)$, see Definition \ref{def:cN Lambda} below. Let $V=\Lambda^\sharp/\Lambda$, we can define a (modified) Deligne-Lusztig variety $Y_V$ over $\bar\kappa$, see \eqref{eq:definition of tY} below. We prove that $Y_V$ is projective and smooth, see Proposition \ref{prop:Y is smooth}.
	When $m\neq 0$ the scheme $\cN_\Lambda$ is isomorphic to $Y_{V,\bar\kappa}$, see Theorem \ref{thm:moduli interpretation of DL variety}.
	
	For vertex lattices of type $0$, we define  $\Exc_\Lambda$ following the idea of \cite[Appendix]{Ho2}.
	\begin{definition}\label{def:Exc}
	The exceptional divisor $\Exc$ of $\cN$ is the set of all points $z=(X,\iota,\lambda,\rho,\cF)\in \cN(\bar\kappa)$ such that the action
	\[\iota: O_F\rightarrow \End(\Lie X)\]
	factor through $O_F\xrightarrow[]{i}O_{\breve F}\rightarrow \bar\kappa$ where $O_{\breve F}\rightarrow \bar\kappa$ is the quotient map. For a vertex lattice $\Lambda$ in $C$ of type $0$, define $\Exc_\Lambda$ to be  the set of all points $z=(X,\iota,\lambda,\rho,\cF)\in \Exc$ such that $\rho(M(X))=\Lambda\otimes_{O_F} O_{\breve F}$. Both $\Exc$ and $\Exc_\Lambda$ are closed subset of $\cN$ and we endow them the structure of reduced schemes over $\bar\kappa$.
	\end{definition}
	The following is a refinement of \cite[Proposition A.2]{Ho2}.
	\begin{lemma}\label{lem:Exc}
	Each $\Exc_\Lambda$ is a Cartier divisor of $\cN$ isomorphic to $\bP^{n-1}_{\bar\kappa}$. The scheme $\Exc$ is a disjoint union of $\Exc_\Lambda$ over all type $0$ lattices $\Lambda$ in $C$.
	\end{lemma}
	\begin{proof}
	Let $z=(X,\iota,\lambda,\rho,\cF)\in \Exc(\bar\kappa)$ and $M=\rho(M(X))\subset N$. Then the action of $\iota(\pi)$ on $\Lie X$ is trivial. Hence $\Pi M\subset \bfV M$ as $\Lie X=M/\bfV M$. Since $\dim_{\bar\kappa} (M/\bfV M)=\dim_{\bar\kappa}(M/\Pi M)=n$, we know $\bfV M=\Pi M$ which is equivalent to $\tau(M)=M$. By \cite[Proposition 4.1]{RTW}, $M=\Lambda\otimes_{O_F} O_{\breve F}$ for some vertex lattice $\Lambda$. Since $M$ is unimodular, $\Lambda$ is of type $0$. Hence $z\in \Exc_\Lambda(\bar\kappa)$.
	Moreover for any $\bar\kappa$-algebra $R$, every rank $n-1$ locally direct summand  of $\Lie X_R$ satisfies  Kr\"amer's signature condition as in Definition \ref{def:NKra} and determines a point of $\Exc_\Lambda(R)$ uniquely. So we get an isomorphism $\bP^{n-1}_{\bar\kappa}\rightarrow \Exc_\Lambda$. Since $\cN$ is regular and $\Exc_\Lambda$ has codimension $1$, $\Exc_\Lambda$ is a Cartier divisor in $\cN$. By looking at $\rho(M(X))$, it is clear that $\Exc_\Lambda\cap \Exc_{\Lambda'}(k)=\emptyset$ if $\Lambda\neq\Lambda'$. Hence $\Exc$ is a disjoint union of  over all type $0$ lattices $\Lambda$.
	\end{proof}
	\begin{remark}
	The proof of Lemma \ref{lem:Exc} shows that the definition of $\Exc_\Lambda$ above agrees with that of \cite[\S 2]{HSY3}.
	\end{remark}
	By Proposition \ref{prop:cN Lambda 0} below, $\cN_\Lambda=\Exc_\Lambda$ for type 0 lattices $\Lambda$.
		The reduced locus $\cN_{\mathrm{red}}$ has a decomposition (see Theorem \ref{thm:Bruhat Tits stratification})
	\[\cN_{\mathrm{red}}=\bigcup_{\Lambda} \ \cN_\Lambda,\]
	where the union is over all vertex lattices.
    The reduced subscheme $\cZ(L)_{\mathrm{red}}$ is a union of Bruhat-Tits strata (see Proposition \ref{prop:reducedlocus})
	\begin{equation}\label{eq:cZ stratification}
	\cZ(L)_{\mathrm{red}}=\bigcup_{\Lambda\supset L}\cN_\Lambda.
	\end{equation}

	\subsection{Horizontal and vertical part}
	A formal scheme $X$ over $\Spf O_{\breve F}$ is called horizontal (resp. vertical) if it is flat over $\SpfOF$ (resp. $\pi$ is locally nilpotent on $\Oo_X$). For a formal scheme $X$ over $\Spf O_{\breve F}$, its horizontal part $X_{\mathscr{H}}$ is canonically defined by the ideal sheaf $\Oo_{X,\mathrm{tor}}$ of torsion sections on $\Oo_X$. If $X$ is noetherian, there exists a $m\in \Z_{>0}$ such that $\pi^m \Oo_{X,\mathrm{tor}}=0$. We define the vertical part $X_{\mathscr{V}}\subset X$ to be the closed formal subscheme defined by the ideal sheaf $\pi^m \Oo_X$. Since $ \Oo_{X,\mathrm{tor}}\cap \pi^m \Oo_X=\{0\}$,
	we have the following decomposition by primary decomposition
	\begin{equation}\label{eq:horizontal vertical decomposition of scheme}
		X=X_\mathscr{H}\cup X_\mathscr{V}
	\end{equation}
	as a union of horizontal and vertical formal subschemes. Notice that the horizontal part  $X_{\mathscr{H}}$ is canonically defined while the vertical part  $X_{\mathscr{V}}$ depends on the choice of $m$.
	
	\begin{lemma}\label{lem:cZnoetherian}
		For a lattice $L^\flat\subset \bV$ of rank greater than or equal to $n-1$ with non-degenerate hermitian form, $\cZ(L^\flat)$ is noetherian.
	\end{lemma}
	\begin{proof}
	First we know that $\cZ(L)$ is locally noetherian since it is a closed formal subscheme of $\cN$ which is locally noetherian. Since the rank of $L$ is greater than or equal to $n-1$, the number of vertex lattices $\Lambda$ containing $L$ is finite. By \eqref{eq:cZ stratification}, we know that $\cZ(L)_\red$ is a closed subset in finitely many irreducible components of $\cN_\red$. Since each irreducible component of $\cN_\red$ is quasi-compact, we know that $\cZ(L)$ is quasi-compact, hence noetherian.
	\end{proof}
	
	\begin{lemma}\label{lem:cZ_vsupportedoncN_red}
		For a rank $n-1$ lattice $L^\flat\subset \bV$ with non-degenerate hermitian form, $\cZ(L)_{\mathscr{V}}$ is supported on the reduced locus $\cN_{\red}$ of $\cN$, i.e., $\Oo_{\cZ(L)_{\mathscr{V}}}$ is annihilated by a power of the ideal sheaf of $\cN_{\red}$.
	\end{lemma}
	\begin{proof}
	We remark here that $\cN_\red$ is exactly the supersingular locus of $\cN$. Hence the proof of the lemma is the same as that of \cite[Lemma 5.1.1]{LZ}.
	\end{proof}
	
	\subsection{Derived special cycles}
	For a locally noetherian formal scheme $X$ together with a formal subscheme $Y$, denote by $K_0^Y(X)$ the Grothendieck group of finite complexes of coherent locally free $\Oo_X$-modules acyclic outside $Y$. For such a complex $A^\bullet$, denote by $[A^\bullet]$ the element in $K_0^Y(X)$ represented by it. We use $K_0(X)$
	to denote $K_0^X(X)$.
	Let $K'_0(Y)$  be the Grothendieck group of coherent sheaves of $\Oo_Y$-modules on $Y$.
	We have a group homomorphism $K_0^Y(X)\rightarrow K'_0(Y)$ which is an isomorphism when $X$ is regular.

	Denote by $\rF^i K_0^Y(X)$ the codimension $i$ filtration on $K_0^Y(X)$ and $\Gr^i K_0^Y(X)$ its $i$-th graded piece. When $X$ is regular, we have a cup product $\cdot$ on $K^Y_0(X)_\Q$ defined by tensor product of complexes. Under the identification $K_0^Y(X)\xrightarrow{\sim} K'_0(Y)$, the cup product is nothing but derived tensor product:
	\[[A]\cdot [B]=[A \otimes^{\bL}_{\Oo_{X}} B].\]
	When $X$ is a scheme, the cup product satisfies (\cite[Section I.3, Theorem 1.3]{soule1994lectures})
	\begin{equation}\label{eq:cupproductfiltration}
		\rF^i K_0^Y(X)_\Q \cdot \rF^j K_0^Y(X)_\Q\subset \rF^{i+j} K_0^Y(X)_\Q.
	\end{equation}
	It is expected that \eqref{eq:cupproductfiltration} is also true when $X$ is a formal scheme, see \cite[(B.3)]{zhang2021AFL}, however we do not need this fact in this paper. Throughout the paper, we assume $X=\cN$ unless stated otherwise.

	Recall that for $x\in \bV$, $\cZ(x)$ is a Cartier divisor (\cite[Proposition 4.3]{Ho2}).
	\begin{definition}
	    Let $L\subset \bV$ be a rank $r$ lattice with a basis $\{x_1,\ldots,x_r\}$.
		Define ${}^\bL \cZ(L)$ to be
		\begin{equation}\label{eq:derivedcZ}
			[\Oo_{\cZ(x_1)}\otimes_{\Oo_\cN}^\bL\cdots \otimes_{\Oo_\cN}^\bL \Oo_{\cZ(x_r)}]\in K_0^{\cZ(L)} (\cN)
		\end{equation}
		where $\otimes^\bL$ is the derived tensor product of complexes of  coherent locally free sheaves on $\cN$. By \cite[Corollary C]{Ho2}, ${}^\bL\cZ(L)$ is independent of the choice of the basis $\{x_1,\ldots,x_r\}$.
	\end{definition}
	
	\begin{definition}\label{def:Int L}
		When  $L$ has rank $n$, we  define the intersection number
		\begin{equation}\label{eq:intL}
			\mathrm{Int}(L)=\chi(\cN,{}^\bL\cZ(L)),
		\end{equation}
		where $\chi$ is the Euler characteristic.
	\end{definition}
	
	\begin{lemma}
	When $L$ is a rank $n$ lattice in $\bV$,
		$\cZ(L)$ is a proper scheme over $\SpfOF$. In particular, $\mathrm{Int}(L)$ is finite.
	\end{lemma}
	\begin{proof}
	By Lemma \ref{lem:cZ_vsupportedoncN_red} $\cZ(L)_{\mathscr{V}}$ is a scheme. We show that $\cZ(L)_{\mathscr{H}}$ is empty. If not, there exists $z\in \cZ(L)(\Oo_K)$ for some finite extension $K$ of $\breve F$. Let $X$ be the corresponding formal hermitian $O_F$-module of signature $(1,n-1)$ over $\Oo_K$.
	Since $L$ has rank $n$ and $\cG$ has signature $(0,1)$, this would imply that $X$ has signature $(0,n)$, which is a contradiction. Hence $\cZ(L)$ is a scheme. Since $\cZ(L)_\red$ is contained in finitely many irreducible components of $ \cN_\red$ and each irreducible component is proper over $\Spec \bar\kappa$, it follows that $\cZ(L)$ is proper over $\SpfOF$. The finiteness of $\Int(L)$ then follows from the same argument before \cite[(B.4)]{zhang2021AFL}.
	\end{proof}

\subsection{A geometric cancellation law}
Recall that for two lattices $L,L'\subset \bV$ of rank $n$, we define
\[n(L',L)=\#\{L''\subset L_F\mid  L\subset L'', L''\cong L'\}.\]
Also recall that $\delta_{\mathrm{odd}}(n) =1$ or $0$ depending on whether $n$ is odd or not.

\begin{proposition}\label{prop:geometric cancel}
Let  $L=I_{\ell}\obot L_2\subset \bV$ where $L_2$ is of rank $r$, $I_\ell$ is unimodular of rank $\ell$ and $n=\ell+r$. Let $I_{r}$ be a unimodular lattice that contains $L_2$.
Then
\begin{align}\label{eq: geom cancel}
 \Int(I_{\ell}\obot L_2)-\Int(L_2)=n(I_{r},L_2)\cdot (\delta_{\mathrm{odd}}(n)-\delta_{\mathrm{odd}}(r)).
 \end{align}
 Moreover, \begin{align}\label{eq: int unimodular}
     \Int(I_{n})=\delta_{\mathrm{odd}}(n).
 \end{align}
\end{proposition}
\begin{proof}
If $L_2$ is unimodular and $r=2$, then $\Int(L_2)=0$ by \cite[Theorem 1.3]{Shi2} and \cite[Theorem 1.3]{HSY}. Combining this with \eqref{eq: geom cancel}, we obtain \eqref{eq: int unimodular}. In order to prove \eqref{eq: geom cancel},
we prove the following equation,
\begin{equation}\label{eq:geometric cancellation}
    \Int(I_1\obot L_2)-\Int(L_2)=(-1)^{r}n(I_r, L_2).
\end{equation}
which is the special case of \eqref{eq: geom cancel} when $\ell=1$. The general case then follows from an easy induction on $n$ using \eqref{eq:geometric cancellation} and the fact
\begin{equation}\label{eq:cancelation for m}
   n(I_n,I_{\ell}\obot L_2) =n(I_{r}, L_2).
\end{equation}

By Proposition 3.2 of \cite{HSY3}, we have the following decomposition of Cartier divisors on $\cN_n$
	\[\cZ(I_1)=\tZ(I_1)+\sum_{\Lambda_0\supset I_1} \Exc_{\Lambda_0},\]
	where the summation is over vertex lattices of type $0$ in $\bV_n$ and $\tZ(I_1)\cong\cN_{n-1}$ by \cite[Corollary 2.7]{HSY3}. By the same corollary, we know that
	\[\chi(\cN_n, [\Oo_{\tZ(I_1)}]\cdot {}^\bL \cZ(L_2))=\chi(\cN_{r}, {}^\bL \cZ(L_2))=\Int(L_2).\]
	Hence we have
	\[\Int(L)-\Int(L_2)=\sum_{\Lambda_0\supset I_1} \chi(\cN_n, [\Oo_{\Exc_{\Lambda_0}}]\cdot {}^\bL \cZ(L_2)).\]
	If $L_2 \not\subset \Lambda_0$, then $\Exc_{\Lambda_0}\cap \cZ(L_2)$ is empty by Proposition \ref{prop:reducedlocus} below.
	If $L_2 \subset \Lambda_0$, then by \cite[Corollary 3.6]{HSY3},  we have \[\chi(\cN_n, [\Oo_{\Exc_{\Lambda_0}}]\cdot {}^\bL \cZ(L_2))=(-1)^r.\]
	Hence
	\[\Int(L)-\Int(L_2)=\sum_{\Lambda_0\supset I_1\obot L_2} (-1)^r.\]
	Combining this with \eqref{eq:cancelation for m} finishes the proof of \eqref{eq:geometric cancellation} and the proposition.
\end{proof}

\section{Bruhat--Tits Stratification of Kr\"amer models}\label{sec:Bruhat-Tits}
We prove a  Bruhat--Tits stratification for the Kr\"amer model (Theorem \ref{thm:Bruhat Tits stratification}), analogous to the case of the Pappas model (proposed in \cite{P}) treated in \cite{RTW}. More precisely,  we define closed subschemes $\cN_\Lambda$ (Definition \ref{def:cN Lambda}) and show that  the reduced locus of $\cN$ is stratified by $\cN_\Lambda$ (Theorem \ref{thm:Bruhat Tits stratification}). From this stratification we obtain a stratification of the reduced locus of $\cZ(L)$ (Proposition \ref{prop:reducedlocus}). We also show that $\cN_\Lambda$ is isomorphic to the (modified) Deligne-Lusztig variety $Y_{V,\bar\kappa}$ defined in \S \ref{subsec:DL variety} (Theorem \ref{thm:moduli interpretation of DL variety}), and is in particular a smooth projective variety over $\bar\kappa$. We remark here that for the purpose of our main result (Theorem \ref{thm: main thm}), only a weaker version of Proposition \ref{prop:reducedlocus} is needed (namely we do not need the reducedness of $\cN_\Lambda$). However we believe the rest of this section contributes to the theory of Rapoport-Zink space and is of independent interest.
\subsection{Deligne-Lusztig varieties}\label{subsec:DL variety}
Throughout this subsection we assume $m\geq 1$.
Let $V$ be a $2m$-dimensional symplectic space over $\kappa=\F_\q$ equipped with the symplectic form $\langle\, ,\, \rangle$. Let $V_{\bar\kappa}=V\otimes_{\kappa} \bar\kappa$ and denote the bilinear extension of $\langle\, ,\, \rangle$ to $V_{\bar\kappa}$ still by $\langle\, ,\, \rangle$.
Let $\Gr(i,V)$ be the Grassmannian variety parametrizing rank $i$ locally direct summands of $V_R$ for any $\kappa$-algebra $R$. Let $\SGr(i,V)$ be the subvariety of $\Gr(i,V)$ whose $\bar\kappa$-points are specified by
\[\SGr(i,V)(\bar\kappa)=\{z\in \Gr(i,V)(\bar\kappa)\mid z \text{ is  isotropic with respect to }  \langle\, ,\, \rangle\}. \]
Let $S_V$ be the subvariety of $\SGr(m,V)$ as in \cite[Equation (5.3)]{RTW} whose $\bar\kappa$-points are specified by
\begin{equation}
    S_V(\bar\kappa)=\{U\in \SGr(m,V)(\bar\kappa) \mid  \dim (U\cap\Phi(U))\geq m-1\},
\end{equation}
where $\Phi$ is the Frobenius endormophism.
By \cite[Proposition 5.3]{RTW} and its remark, $S_V$ has isolated singularities  which are exactly the points where $U=\Phi(U)$. We denote by $\cU$ the nonsingular locus of $S_V$. By Proposition 5.5 of loc.cit., $S_{V,\bar\kappa}$ is irreducible of dimension $m$.
To resolve the singularities of $S_V$, define $Y_V$ to be the subvariety of $\SGr(m,V)\times \SGr(m-1,V)$ whose $\bar\kappa$-points are specified by
\begin{equation}\label{eq:definition of tY}
    Y_V(\bar\kappa)=\{(U,U')\in (\SGr(m,V)\times \SGr(m-1,V))(\bar\kappa)\mid U'\subset U\cap \Phi(U) \}.
\end{equation}
Then the variety $Y_V$ is a projective subvariety of $\Gr(m,V)\times \Gr(m-1,V)$.
The forgetful map $(U,U')\rightarrow U$ defines a morphism $\pi_m:Y_V\rightarrow S_V$.

\begin{lemma}\label{lem:pi m isomorphism}
The morphism $\pi_m$ is a projective morphism. It is an isomorphism outside the singular locus of $S_V$. For a singular point $z$ of $S_V$, $\pi^{-1}_m(z)\cong \bP^{m-1}_{\bar\kappa}$.
\end{lemma}
\begin{proof}
First we know $\pi_m$ is projective as it is a morphism between projective schemes.
Consider a $\bar\kappa$-point $z=U$ in $\cU(\bar\kappa)$. Then $U\cap \Phi(U)$ has dimension $m-1$, this entails $U'=U\cap \Phi(U)$. This shows that the morphism has an inverse when restricted on $\pi_m^{-1}(\cU)$, hence $\pi_m|_{\pi_m^{-1}(\cU)}$ is an isomorphism of varieties.

If $z=U$ is a singular $\bar\kappa$-point, then $U=\Phi(U)$ and $U'$ can be any element in $\Gr(m-1,U)\cong \bP^{m-1}_{\bar\kappa}$. This finishes the proof of the lemma.
\end{proof}

\begin{proposition}\label{prop:Y is smooth}
The projective variety $Y_{V,\bar\kappa}$ is irreducible and  smooth of dimension $m$.
\end{proposition}
\begin{proof}
Define $\SGr(m,m-1,V)$ to be the sub flag variety of $\SGr(m,V)\times \SGr(m-1,V)$ whose $\bar\kappa$-points are specified by
\[ \SGr(m,m-1,V)(\bar\kappa)=\{(U,U')\in (\SGr(m,V)\times \SGr(m-1,V))(\bar\kappa)\mid U'\subset U\}.\]
Then $Y_{V,\bar\kappa}$ is the intersection of the image of the closed immersion
\begin{align*}
    (\SGr(m,m-1,V)_{\bar\kappa})^2 & \rightarrow (\SGr(m,V)_{\bar\kappa}\times \SGr(m-1,V)_{\bar\kappa})^2:\\
    (U_1,U'_1, U_2, U'_2) & \mapsto (U_1,U'_1, U_2, U'_2) ,
\end{align*}
with the image of the closed immersion
\begin{align*}
    \SGr(m,V)_{\bar\kappa}\times \SGr(m-1,V)_{\bar\kappa} & \rightarrow (\SGr(m,V)_{\bar\kappa}\times \SGr(m-1,V)_{\bar\kappa})^2:\\
    (U_3, U_4) & \mapsto (U_3,U_4, \Phi(U_3), U_4).
\end{align*}
Since $(\SGr(m,m-1,V)_{\bar\kappa})^2$ and $\SGr(m,V)_{\bar\kappa}\times \SGr(m-1,V)_{\bar\kappa}$ are smooth (as they are homogeneous varieties), and $\Phi$ induces the zero map on the tangent space, one can see immediately that the intersection is transversal. Hence $Y_{V,\bar\kappa}$ is smooth. Since $S_{V,\bar\kappa}$ is irreducible of dimension $m$, by Lemma \ref{lem:pi m isomorphism}, we know $Y_{V,\bar\kappa}$ is connected and has an open subvariety of dimension $m$. Taking into consideration the smoothness, we know $Y_{V,\bar\kappa}$ must be irreducible of dimension $m$. This finishes the proof of the proposition.
\end{proof}

\begin{remark}
One can show that $Y_{V,\bar\kappa}$ is in fact the blow-up of $S_{V,\bar\kappa}$ along its singular locus.
\end{remark}

\subsection{Minuscule cycle $\cN_\Lambda$ and its tangent space}
In this section, we often identify  a vertex lattice $\Lambda\subset \bV$ with $b^{-1}(\Lambda)$ using the isomorphism of hermitian spaces \eqref{eq: C V isomorphism} unless otherwise stated.
\begin{definition}\label{def:cN Lambda}
For a vertex lattice $\Lambda\subset \bV$ of type $t(\Lambda)=2m$, define the subfunctor $\cN_\Lambda$ to be the subfunctor of $\cN$ such that for a $O_{\breve F}$-scheme $S$, $\cN_\Lambda(S)$ is the set of isomorphism classes of tuples $(X,\iota,\lambda,\rho,\cF)$ satisfying the following conditions.
\begin{enumerate}
    \item $(X,\iota,\lambda,\rho,\cF)\in \cZ(\Lambda)(S)$.
    \item If $m\geq 1$, we require in addition that $x_*(\Lie (\cG\times_{\SpfOF} S))\subset \cF$ for any $x\in \Lambda$.
\end{enumerate}
\end{definition}

We first describe the $\bar\kappa$-points of $\cN$ and $\cN_\Lambda$.
\begin{proposition}\label{prop:k points of cNKra}
There is a bijection between $\cN_{\mathrm{red}}(\bar\kappa)$ and the set of pairs of $O_{\breve F}$-lattices $(M,M')$ in $N$ satisfying
\[M^\sharp=M,\quad \Pi M \subset \tau^{-1}(M) \subset \Pi^{-1} M, \quad \bfV M\subset M'\subset \tau^{-1}(M)\cap M, \quad  \hbox {and }\quad \mathrm{length}(M/M')=1.\]
\end{proposition}
\begin{proof}
Let $(X,\iota,\lambda,\rho,\cF)$ be a $\bar\kappa$-point of $\cN$ and $M(X)$ be the $O_{F_0}$-relative Dieudonn\'e module of $X$. Define $M=\rho(M(X))\subset N$ and $M'=\rho(\Pr^{-1}(\cF))\subset N$ where $\Pr:M(X)\rightarrow \Lie X=M(X)/\bfV M(X)$ is the natural quotient map.
The condition $M^\sharp=M$ is equivalent to the fact that $\lambda$ is a principal polarization. The condition $\Pi M \subset \tau^{-1}(M) \subset \Pi^{-1} M$ is equivalent to $\pi_0 M\subset \bfV M \subset M$.
The condition $\bfV M \subset M'\subset\tau^{-1}(M)\cap M$ and $\mathrm{length}(M/M')=1$ is equivalent to the condition
\[\bfV M\subset M'\subset M, \ \Pi M'\subset \bfV M,\ \dim_{\bar\kappa}(M/M')=1,\]
which is in turn equivalent to
\[\cF\subset \Lie X, \quad \dim_{\bar\kappa}(\cF)=n-1, \quad \Pi \cdot \cF=\{0\},\quad \Pi \cdot \Lie X \subset \cF. \]
Notice that the condition $\Pi \cdot \Lie X \subset \cF$ is automatic once we know $ \dim_{\bar\kappa}(\cF)=n-1$ and $\cF$ is stable under the action of $\Pi$. Hence the filtration $\cF\subset \Lie X$ satisfies Kr\"amer's signature condition and we have translated all conditions in the definition of $\cN$ in term of relative Dieudonn\'e modules. The proposition now follows from Dieudonn\'e theory.
\end{proof}

For a vertex lattice $\Lambda$ in $\bV$ or $C$, define
\begin{equation}\label{eq:breve Lambda}
    \breve{\Lambda}:=\Lambda\otimes_{O_F} O_{\breve F},\ \breve{\Lambda}^\sharp:=\Lambda^\sharp\otimes_{O_F} O_{\breve F}.
\end{equation}
\begin{corollary}\label{cor:reduced point of cN Lambda}
Let $\Lambda$ be a vertex lattice in $C$.
There is a bijection between $\cN_{\Lambda}(\bar\kappa)$ and the set of pairs of $O_{\breve F}$-lattices $(M,M')$ in $N$ satisfying the conditions in Proposition \ref{prop:k points of cNKra} and the following condition.
\begin{enumerate}
    \item If $t(\Lambda)=0$, then $M=\breve{\Lambda}$.
    \item If $t(\Lambda)\geq 2$, then $\breve{\Lambda}\subset M'\subset M$.
\end{enumerate}
\end{corollary}
\begin{proof}
By Dieudonn\'e theory, a point $(X,\iota,\lambda,\rho,\cF)\in \cN(\bar\kappa)$ is in $\cZ(\Lambda)(\bar\kappa)$ if and only if $\rho^{-1}\circ x (M(\bY))\subset M(X)$ for any $x\in b^{-1}(\Lambda)$. Since $M(\bY)$ is generated by $e$, this is the case if and only if $x(e) \in M=\rho(M(X))$ for any $x\in b^{-1}(\Lambda)$, if and only if $\Lambda \subset M$ (by the definition of $b$ \eqref{eq: C V isomorphism}), if and only if $\breve{\Lambda}\subset M$.
When  $t(\Lambda)=0$,  both $M$ and $\breve{\Lambda}$ are unimodular, thus $M=\breve{\Lambda}$. Similarly (2) is equivalent to Condition (2) in Definition \ref{def:cN Lambda} as the Lie algebra of $\bY$ is generated by the image of $e$ under the quotient map $M(\bY)\rightarrow \Lie \bY=M(\bY)/\bfV M(\bY)$.
\end{proof}

To study the tangent space of $\cN_\Lambda$, we recall the Grothendieck-Messing deformation theory of $\cN$ from \cite[\S 3]{Ho2}.
We remark here that although \cite{Ho2} deals with the case $F_0=\Q_p$, the argument in fact applies to general $F_0$ using the relative display theory of
\cite{ACZ}.
Let $R\in \Nilp O_{\breve F}$. For a strict $O_{F_0}$-module $X$ over $\Spec R$, we denote by $D(X)$ the $O_{F_0}$-relative Dieudonn\'e crystal in the sense of \cite[\S 3]{ACZ}.
A point $z\in \cN(R)$ corresponds to a strict $O_{F_0}$-module $(X,\iota,\lambda)$ over $R$ together with filtration $\cF\subset \Lie X$ satisfying Definition \ref{def:NKra}.
We have the following exact sequence of locally free $R$-modules
\begin{equation}\label{eq:Hodge filtration}
    0\rightarrow \Fil(X)\rightarrow  D(X)\rightarrow \Lie X\rightarrow 0,
\end{equation}
where $\Fil(X)$ and $\Lie X$ are of rank $n$ and $ D(X)$ is of rank $2n$.
The principal polarization $\lambda$ induces a symplectic form $\langle\, ,\, \rangle$ on $ D(X)$ such that
\[\langle \iota(a) x,y\rangle=\langle x,\iota(\bar{a}) y\rangle\]
for all $a\in O_F$ and $x,y\in  D(X)$.
With respect to $\langle\, ,\, \rangle$ the Hodge filtration $\Fil(X)$ is maximal isotropic. Hence $\langle\, ,\, \rangle$ induces a perfect pairing (still denoted by $\langle\, ,\, \rangle$):
\begin{equation}\label{eq:perfect pairing between Lie and Fil}
    \langle\, ,\, \rangle: \Fil(X)\times \Lie X\rightarrow R.
\end{equation}
The submodule $\cF\subset \Lie X$ and its perpendicular complement $\cF^\bot$ (which is locally a direct summand of $\Fil(X)$ of rank one) with respect to \eqref{eq:perfect pairing between Lie and Fil} determine each other. The condition on $\cF$ in Definition \ref{def:NKra} is
\begin{equation}\label{eq:cF condition}
    O_F \text{ acts on } \cF \text{ by } O_F\xrightarrow[]{\bar i}O_{\breve F}\rightarrow \Oo_S \text{ and on } \Lie X/\cF \text{ by } O_F\xrightarrow[]{i}O_{\breve F}\rightarrow \Oo_S.
\end{equation}
This is equivalent to the condition that $O_F$ acts on $\cF^\bot$ by $O_F\xrightarrow[]{\bar i}O_{\breve F}\rightarrow \Oo_S$ and on $\Fil(X)/\cF^\bot$ by $O_F\xrightarrow[]{i}O_{\breve F}\rightarrow \Oo_S$. Since $O_{F_0}$ acts on $D(X)$ by $i_0$ and $O_F=O_{F_0}[\pi]$, \eqref{eq:cF condition} is further equivalent to
\begin{equation}
    (\Pi+\pi)\cdot\cF^\bot=0,\
(\Pi-\pi)\cdot\Fil(X)\subset \cF^\bot,
\end{equation}
where we use $\Pi$ to denote the action $\iota(\pi)$ on $D(X)$.

\begin{definition}\label{def:category C}
Let $\mathscr{C}$ be the following category. Objects in $\mathscr{C}$ are triples $(\Oo, \Oo\rightarrow \bar\kappa, d)$ where $\Oo$ is an Artinian $O_{\breve F}$-algebra, $\Oo\rightarrow \bar\kappa$ is an $O_{\breve F}$-algebra homomorphism, and $d$ is a nilpotent $O_{F_0}$-pd-structure (see \cite[Definition 1.2.2]{ACZ}) on $\mathrm{Ker}(\Oo\rightarrow \bar\kappa)$. Morphisms in $\mathscr{C}$ are $O_{\breve F}$-algebra homomorphisms compatible with structure maps to $\bar\kappa$ and $O_{F_0}$-pd-structure structures.
\end{definition}

Let $z=(X,\iota,\lambda,\rho,\cF)\in \cZ(\Lambda)(\bar\kappa)$ and $M=\rho(M(X))\subset N$. Then $\Lambda\subset M$ by Corollary \ref{cor:reduced point of cN Lambda}. We can identify \eqref{eq:Hodge filtration} with
\[0\rightarrow \bfV M/\pi_0 M \rightarrow  M/\pi_0 M\rightarrow M/\bfV M\rightarrow 0.\]
Let $\cF^\bot\subset \bfV M/\pi_0 M$ be the perpendicular complement of $\cF$ as described above.
Denote by $\cZ(\Lambda)_z$ (resp. $\cN_{\Lambda,z}$) the completion of $\cZ(\Lambda)$ (resp. $\cN_\Lambda$) at $z$. For any $\Oo\in\mathscr{C}$ and $\tilde{z}=(\tilde{X},\cdots)\in \cZ(\Lambda)_z(\Oo)$, we can identify $D(\tilde{X})$ with $M_\Oo:=M\otimes_{O_{\breve{F}_0}} \Oo$ and by Grothendieck-Messing theory $\tilde z$ corresponds to a filtration of free $\Oo$-module direct summands
\[\tilde{\cF}^\bot\subset\widetilde{\Fil}\subset M_\Oo,\]
which lifts the filtration $\cF^\bot\subset \Fil \subset M_{\bar\kappa}=M/\pi_0 M$.
Let $f_\Oo$ be the map
\begin{equation}
    f_\Oo: \tilde{z}\mapsto (\tilde{\cF}^\bot,\widetilde{\Fil}).
\end{equation}

\begin{lemma}\label{lem:G M theory of cN Lambda}
Let  the notations  be as above.  Denote by $\Lambda_{M,\Oo}$ the image of the composition of maps $\breve{\Lambda}\rightarrow M\rightarrow M_\Oo$, and let $\Lambda_{M,\Oo}^\bot$ be its perpendicular complement in $M_\Oo$ under the alternating form $\langle\, ,\, \rangle$.
\begin{enumerate}[leftmargin=*]
    \item
The map $f_\Oo$ defines a bijection from $\cZ(\Lambda)_z(\Oo)$ to the set consisting of pairs $(\tilde{\cF}^\bot,\widetilde{\Fil})$ lifting $(\cF^\bot, \Fil)$
satisfying the following conditions:
\begin{enumerate}
    \item $\tilde{\cF}^\bot$ and $\widetilde{\Fil}$ are free $\Oo$-module direct summands of $M_\Oo$ of rank $1$ and $n$ respectively and $\tilde{\cF}^\bot\subset\widetilde{\Fil}$;
    \item $\widetilde{\Fil}$ is isotropic with respect to $\langle\, ,\, \rangle$;
    \item $(\Pi+\pi)\cdot\tilde{\cF}^\bot=0$ and $(\Pi-\pi)\cdot\widetilde{\Fil}\subset \tilde{\cF}^\bot$;
    \item  $\widetilde{\Fil}$ contains $\widetilde{\Fil}^-:=(\Pi+\pi)\cdot\Lambda_{M,\Oo}$.
\end{enumerate}
\item The restriction of $f_\Oo$ to $\cN_{\Lambda,z}(\Oo)$ defines a bijection from $\cN_{\Lambda,z}(\Oo)$ to the set consisting of pairs $(\tilde{\cF}^\bot,\widetilde{\Fil})$ satisfying the above conditions together with the extra condition:
\begin{enumerate}\addtocounter{enumii}{4}
    \item $\tilde{\cF}^\bot \subset\Lambda_{M,\Oo}^\bot$.
\end{enumerate}
\end{enumerate}
\end{lemma}
\begin{proof}
Proof of (1): By the previous discussion, $(\tilde{\cF}^\bot,\widetilde{\Fil})$ satisfies conditions (a), (b) and (c) for any $\tilde z \subset\cN_z(\Oo)$ ($\cN_z$ is the completion of $\cN$ at $z$). Conversely by Grothendieck-Messing theory any pair $(\tilde{\cF}^\bot,\widetilde{\Fil})$ lifting $(\cF^\bot, \Fil)$ satisfying (a), (b) and (c) gives rise to a unique point $\tilde{z}\in \cN_z(\Oo)$.
Since the action of $O_F$ on $\Lie \cG$ is via the inclusion $\bar i$, the Hodge filtration of $\cG_\Oo$ is $\spa_{\Oo}\{ (\Pi+\pi)\cdot e\otimes 1\}$ where $e$ is a generator of $M(\bY)$ as in \S \ref{subsec:associated hermitian space}. The image of the $\spa_{\Oo}\{ (\Pi+\pi)\cdot e\otimes 1\}$ under elements of $\Lambda\subset \bV$ in $M_\Oo$ is exactly $\widetilde{\Fil}^-$. By Grothendick-Messing theory again, $\tilde{z}\in \cZ(\Lambda)_z(\Oo)$ if and only if condition (d) holds.

(2) is a corollary of (1). For any $\tilde{z}=(\tilde{X},\ldots,\tilde{\cF})\in \cZ(\Lambda)_z(\Oo)$, let $\tilde{\cF}'$ be the preimage of $\tilde{\cF}$ under the quotient map $ M_\Oo\rightarrow M_\Oo/\widetilde{\Fil}$.
Condition (2) in Definition \ref{def:cN Lambda} is equivalent to $\Lambda_{M,\Oo}\subset \tilde{\cF}'$ by the same reasoning as Corollary \ref{cor:reduced point of cN Lambda}. The perpendicular complement of $\tilde{\cF}'$ with respect to $\langle\, ,\, \rangle$ is $\tilde{\cF}^\bot$. Hence condition (2) in Definition \ref{def:cN Lambda} is equivalent to condition (e). Hence $\tilde{z}\in \cN_{\Lambda,z}(\Oo)$ if and only if (e) is satisfied. This finishes the proof of the lemma.
\end{proof}

\begin{lemma}\label{lem:Lambda M basis}
Let $\Lambda$ be a vertex lattice of type $2m$ in $C$ and $M\subset N$ be an $O_{\breve F}$-lattice such that $\Lambda\subset M$ and $M=M^\sharp$. Then there is an $O_{\breve F}$-basis $\{e_1,\ldots, e_n\}$ of $M$ such that
\[(e_\alpha, e_{\alpha+m})=1, \ (e_\mu,e_\mu)\in O_{\breve F}^\times \]
for $1\leq \alpha \leq m$, $2m+1\leq \mu \leq n$, the inner product $(\, ,\, )$ between any other basis vectors is zero, and
\[\breve{\Lambda}=\spa_{O_{\breve F}}\{ \Pi e_1,\ldots, \Pi e_m,e_{m+1},\ldots, e_n\}. \]
\end{lemma}
\begin{proof}
By assumption we have $\Pi M\subset \Pi\breve{\Lambda}^\sharp \subset \breve{\Lambda} \subset M$ and $\dim_{\bar\kappa}(M/\breve{\Lambda})=m$.
With respect to the $\bar\kappa$-valued quadratic form $(\, ,\, )\pmod{\pi}$ on $M/\Pi M$, $\breve{\Lambda}/\Pi M$ has a decomposition
\[\breve{\Lambda}/\Pi M=R\obot W,\]
where $R$ is totally isotropic and $W$ is non-degenerate. Then by the nondegeneracy of $(\, ,\, )\pmod{\pi}$ on $M/\Pi M$ we know that there is a totally isotropic subspace $R'$ such that
\[M/\Pi M=(R'\oplus R)\obot W,\]
and $(\, ,\, )\pmod{\pi}$ induces a perfect pairing between $R$ and $R'$.
Hence we can find a basis $\{\bar{e}_1,\bar{e}_n\}$ of $M/\Pi M$ such that $R'=\langle \bar{e}_1,\ldots, \bar{e}_{m}\rangle$, $R=\langle \bar{e}_{m+1},\ldots, \bar{e}_{2m}\rangle$, $W=\langle \bar{e}_{2m+1},\ldots, \bar{e}_n\rangle$, and
\[(\bar{e}_\alpha, \bar{e}_{\alpha+m})=1 \pmod{\pi}, \ (\bar{e}_\mu,\bar{e}_\mu) \pmod{\pi} \in \bar\kappa^\times \]
for $1\leq \alpha \leq m$ and $2m+1\leq \mu \leq n$ and the pairing between all other basis vectors are zero. We can lift the above basis to a basis $\{e_1,\ldots,e_n\}$ of $M$ which will satisfy the assumptions of the lemma.
\end{proof}

\begin{proposition}\label{prop:cZ Lambda defined over k}
The scheme $\cZ(\Lambda)$ has no $O_{\breve F}/(\pi^2)$-point.
\end{proposition}
\begin{proof}
Let $\Oo=O_{\breve F}/(\pi^2)$ with the reduction map $\Oo\rightarrow \bar\kappa$ and the natural $O_{F_0}$-pd structure on $\pi \Oo$. Then $\Oo\in \mathscr{C}$. Let $z=(X,\iota,\lambda,\rho,\cF)\in \cZ(\Lambda)(\bar\kappa)$ and $M=\rho(M(X))\subset N$ as in Proposition \ref{prop:k points of cNKra}. Then by Corollary \ref{cor:reduced point of cN Lambda} $\breve\Lambda\subset M$, and we can assume there is an $O_{\breve F}$-basis $\{e_1,\ldots,e_n\}$ of $M$ as in Lemma \ref{lem:Lambda M basis}. Denote the image of $e_i$ in $M_\Oo$ still by $e_i$.
Then $\{e_1,\ldots, e_n,\Pi e_1,\ldots, \Pi e_n\}$ is an $\Oo$-basis of $M_\Oo$. With respect to the alternating form $\langle\, ,\, \rangle$, we have by \eqref{eq:hermitian and symplectic form}
\begin{equation}\label{eq:symplectic basis}
    \langle e_\alpha, \Pi e_{m+\alpha}\rangle=-1/\delta,  \ \langle e_{m+\alpha}, \Pi e_\alpha\rangle=-1/\delta,\  \langle e_\mu, \Pi e_\mu\rangle\in O_{\breve{F}_0}^\times,
\end{equation}
for $ 1\leq \alpha \leq m$, $2m+1\leq \mu \leq n$, and all other pairings between basis vectors are zero.

Assume that $z$ can be lifted to a point $\tilde{z}\in\cZ(\Lambda)_z(\Oo)$, which corresponds to a pair $(\tilde{\cF}^\bot,\widetilde{\Fil})$ as in Lemma \ref{lem:G M theory of cN Lambda}.
First notice that
\begin{equation}\label{eq:Fil -}
    \widetilde{\Fil}^-=(\Pi+\pi)\cdot\Lambda_{M,\Oo}=\spa_{O_{\breve F}\otimes_{O_{\breve{F}_0}}\Oo}  \{\pi \Pi e_1,\ldots, \pi\Pi e_m,(\Pi+\pi)e_{m+1},\ldots, (\Pi+\pi)e_n \}.
\end{equation}
With respect to the alternating form $\langle\, ,\, \rangle$, its perpendicular complement $(\widetilde{\Fil}^-)^\bot$  in $\Lambda_{M,\Oo}$ is generated by
\begin{equation}\label{eq:Fil - perb}
  \{(\Pi+\pi) e_1,\ldots (\Pi+\pi) e_m,\Pi e_{m+1},\pi e_{m+1},\ldots, \Pi e_{2m},\pi e_{2m}, (\Pi+\pi) e_{2m+1},\ldots, (\Pi+\pi) e_n \}.
\end{equation}
By Lemma \ref{lem:G M theory of cN Lambda} (c), $\tilde{\cF}^\bot$ is annihilated by $\Pi+\pi$, hence it is spanned by a vector
\[v=\sum_{i=1}^n a_i(\Pi-\pi) e_i,\]
where $a_i\in \Oo^\times$ for some $i$ as $\tilde{\cF}^\bot$ is a direct summand of $M_\Oo$. By Lemma \ref{lem:G M theory of cN Lambda}, we must have $ \widetilde{\Fil}^-\subset \widetilde{\Fil}$, $\tilde{\cF}^\bot \subset \widetilde{\Fil}$ and $\widetilde{\Fil}$ is isotropic. Hence $\widetilde{\Fil}\subset(\widetilde{\Fil}^-)^\bot$. Moreover $\langle \widetilde{\Fil}^-, \tilde{\cF}^\bot\rangle=0$, which implies $a_i\in \pi \Oo$ for $1\leq i\leq m$ and $2m+1\leq i \leq n$. Hence without loss of generality we can assume that $a_{m+1}=1$.

Since $\widetilde{\Fil}$ is a direct summand of $M_\Oo$ we have $M_\Oo=\widetilde{\Fil}\oplus S$ where $S$ is an $\Oo$-module. We can write $\Pi e_1=w+w'$ where $w\in\widetilde{\Fil}$ and $w'\in S$. Since $\pi\Pi e_1\in \widetilde{\Fil}^-\subset\widetilde{\Fil}$, we must have $\pi w'=0$. This implies that $w'\in  \pi M_\Oo$ and $w$ is of the form
\[w=(\Pi+b \pi) e_1+x\]
where $b\in \Oo$ and $x\in \pi\cdot\spa_\Oo\{e_2,\Pi e_2,\ldots,e_n,\Pi e_n\}$.
Since $w\in\widetilde{\Fil}\subset(\widetilde{\Fil}^-)^\bot$, by \eqref{eq:Fil - perb}, we must have $b=1$ and $x$ is of the form
\[x=\sum_{i=2}^m d_i(\Pi+\pi) e_i+\sum_{i=m+1}^{2m} (c_i +d_i\Pi) e_i+\sum_{i=2m+1}^n d_i(\Pi+\pi)e_i,\]
where $c_i\in \pi \Oo$ for $m+1\leq i \leq 2m$ and $d_i\in \pi \Oo$ for $2\leq i \leq n$.
Since $(\Pi+\pi)e_i\in \widetilde{\Fil}$ for $2m+1\leq i \leq n$, by changing $w$ and $x$ at the same time if necessary we can assume that $d_i=0$ for $2m+1\leq i \leq n$.
By \eqref{eq:symplectic basis}, we have
\[\langle (\Pi+\pi)e_1,(\Pi-\pi)e_{m+1} \rangle=2\pi \langle e_1,\Pi e_{m+1}\rangle\neq 0.\]
Moreover
\begin{align*}
    \langle  x,v\rangle=& \sum_{i=2}^m \langle d_i (\Pi+\pi)e_i, a_{m+i}(\Pi-\pi)e_{m+i}\rangle+\sum_{i=m+1}^{2m} \langle (c_i  +d_i\Pi)e_i, a_{i-m}(\Pi-\pi)e_{i-m}\rangle=0
\end{align*}
Here we have used the fact that $a_i\in \pi \Oo$ for $1\leq i \leq m$, $c_i\in \pi \Oo$ for $m+1\leq i \leq 2m$ and $d_i\in \pi \Oo$ for $2\leq i \leq 2m$. Then $\langle w, v\rangle\neq 0$ which
contradicts the fact that $\widetilde{\Fil}$ is isotropic. Hence there is no lift of $z$ into $\cZ(\Lambda)(\Oo)$. This proves the lemma.
\end{proof}
As $\cN_\Lambda$ is a formal subscheme of $\cZ(\Lambda)$, the following corollary is immediate.
\begin{corollary}\label{cor:cN Lambda defined over k}
$\cN_\Lambda$ has no $O_{\breve F}/(\pi^2)$-point.
\end{corollary}

\begin{proposition}\label{prop:reducedness of cN Lambda}
Let $\Lambda$ be a vertex lattice of type $2m$ ($m\geq 1$) in $\bV$ and $z\in \cN_\Lambda(\bar\kappa)$. Then the tangent space $T_z(\cN_{\Lambda,\bar\kappa})$ has dimension less than or equal to $m$.
\end{proposition}
\begin{proof}
Let $z=(X,\iota,\lambda,\rho,\cF)\in \cN_\Lambda(\bar\kappa)$ and $M=\rho(M(X))\subset N$ as in Proposition \ref{prop:k points of cNKra}. Let $\Oo=\bar\kappa[\epsilon]/(\epsilon^2)$, then $\Oo$ is an $O_{\breve F}$-algebra through the map $O_{\breve F}\rightarrow \bar\kappa \rightarrow \Oo$ and the ideal $(\epsilon)\subset \Oo$ is equipped with a natural $O_{F_0}$-pd structure. Then $\Oo\in \mathscr{C}$. Any point $\tilde{z}\in T_z(\cN_{\Lambda,\bar\kappa})=\cN_{\Lambda,z}(\Oo)$ corresponds to a unique pair $(\tilde{\cF}^\bot,\widetilde{\Fil})$ lifting $(\cF^\bot, \Fil)$ as in Lemma \ref{lem:G M theory of cN Lambda}. We prove the lemma in two cases.

\noindent
Case (a): $\Fil\neq \Pi \cdot M_{\bar\kappa}$.
Since $M_\Oo$ is a free $O_F\otimes_{O_{F_0}}\Oo$-modules of rank $n$,
we have the following exact sequence
\[0\rightarrow \Pi\cdot M_\Oo\rightarrow M_\Oo\xrightarrow{\Pi}  \Pi\cdot M_\Oo\rightarrow 0,\]
where $\Pi\cdot M_\Oo$ is a free $\Oo$-module of rank $n$ and the first arrow is the natural injection. This implies the following sequence is exact.
\begin{equation}\label{eq:Fil pi exact sequence}
   0\rightarrow (\Pi\cdot M_\Oo)\cap \widetilde{\Fil}\rightarrow \widetilde{\Fil}\xrightarrow{\Pi}  \Pi\cdot \widetilde{\Fil}\rightarrow 0.
\end{equation}
Since $\widetilde{\Fil}\neq \Pi \cdot M_\Oo$, by \eqref{eq:Fil pi exact sequence} we know that $\Pi\cdot\widetilde{\Fil}\neq \{0\}$. By Lemma \ref{lem:G M theory of cN Lambda}, $\Pi \cdot \widetilde{\Fil} \subset \tilde{\cF}^\bot$ and $\tilde{\cF}^\bot$ has rank $1$, we know that $\Pi \cdot \widetilde{\Fil} = \tilde{\cF}^\bot$ by Nakayama's lemma. In particular $\tilde{\cF}^\bot$ is determined by $\widetilde{\Fil}$. Moreover $\widetilde{\Fil}$ is determined by its image in the $\Oo$-module $(\widetilde{\Fil}^-)^\bot/ \widetilde{\Fil}^-$ where $\widetilde{\Fil}^-=\Pi\cdot \Lambda_{M,\Oo}$ as in Lemma \ref{lem:G M theory of cN Lambda}. Equation \eqref{eq:Fil -} is still true and implies that $\widetilde{\Fil}^-$ is an isotropic free $\Oo$-module direct summand of $M_\Oo$ of rank $n-m$ (notice that $\pi=0$ in $\Oo$).
Notice that by \eqref{eq:Fil pi exact sequence} and the fact that $\Pi \cdot \widetilde{\Fil} = \tilde{\cF}^\bot$ is free (in particular projective), $(\Pi\cdot M_\Oo)\cap \widetilde{\Fil}$ is a free direct summand of $\widetilde{\Fil}$ of corank $1$.
This implies that $(\Pi\cdot M_\Oo)\cap \widetilde{\Fil}$ is a free direct summand of $\Pi\cdot M_\Oo$ of corank $1$ as well.
So $(\Pi\cdot M_\Oo)\cap \widetilde{\Fil}/\widetilde{\Fil}^-$ is a hyperplane in the $\Oo$-module $\Pi\cdot M_\Oo/\widetilde{\Fil}^-$ of rank $m$, and is determined by $m-1$ parameters over $\bar\kappa$ as the tangent space of $\bP^{m-1}_{\bar\kappa}$ has dimension $m-1$. Hence $(\Pi\cdot M_\Oo)\cap \widetilde{\Fil}$ is determined by $m-1$ parameters over $\bar\kappa$ as well.
Since $\widetilde{\Fil}$ is maximal isotropic, it corresponds to a hyperplane in the rank two $\Oo$-module
\[((\Pi\cdot M_\Oo)\cap \widetilde{\Fil})^\bot /(\Pi\cdot M_\Oo)\cap \widetilde{\Fil}, \]
hence is further determined by one parameter over $\bar\kappa$ as the tangent space of $\bP^{1}_{\bar\kappa}$ has dimension $1$. This proves case (a).

\noindent
Case (b): $\Fil=\Pi \cdot M_{\bar\kappa}$.  By Lemma \ref{lem:Lambda M basis}, we know $\Pi\cdot M_\Oo\subset \Lambda_{M,\Oo}$ and $\Lambda_{M,\Oo}$ is a free $\Oo$-module direct summand of $M_\Oo$ of corank $m$.
Hence $(\Lambda_{M,\Oo})^\bot$ is a free $\Oo$-module of rank $m$ and is in $(\Pi\cdot M_\Oo)^\bot=\Pi\cdot M_\Oo$.
As in \cite{Kr}, we assume that we have a $\Oo\otimes_{O_{F_0}}O_F$-basis $\{e_1,\ldots,e_n\}$ of $M_\Oo$ such that $\langle e_i,\Pi \cdot e_j\rangle =\delta_{ij}$ for $1\leq i,j\leq n$ and all other pairings between these basis vectors are zero. The lift $\widetilde{\Fil}$ is spanned by $x_1,\ldots,x_n$ where
\[(x_1,\ldots,x_n)=(e_1,\ldots,e_n,\Pi e_1,\ldots, \Pi e_n)\begin{pmatrix}
A\epsilon \\
I_n
\end{pmatrix}\]
where $A\in M_n(\bar\kappa)$ and $A={}^t A$ since $\widetilde\Fil$ is isotropic.
Assume $\cF^\bot\subset\Fil=\Pi \cdot M_{\bar\kappa}$ is spanned by
\[\sum_{i=1}^n b_{n+i} \Pi\cdot e_i.\]
Then $b_{n+i}\neq 0$ for some $i$ and we can assume without loss of generality $b_{n+1}=1$. The lift $\tilde{\cF}^\bot$ is spanned by
\[\sum_{i=1}^{n} \tilde{b}_i e_i+\sum_{i=n+1}^{2n} \tilde{b}_i \Pi\cdot e_i,\]
where $\tilde{b}_{n+1}=1$ and $\tilde{b}_{n+i}=b_{n+i}+\epsilon c_{i}$ for $2\leq i \leq n$ and some $c_i\in \bar\kappa$. Let
\[\lambda={}^t (\tilde{b}_{n+1},\ldots, \tilde{b}_{2n}).\]
Equations (4.7), (4.8) and (4.10) of \cite{Kr} tell us that
\[A=\gamma_1 \lambda \cdot {}^t \lambda \]
for some $\gamma_1\in \bar\kappa$. Equation (4.5) of loc.cit. tells us
\[{}^t (\tilde{b}_{1},\ldots, \tilde{b}_{n})=A \lambda,\]
which is equal to $\gamma_1 \lambda \cdot {}^t \lambda \cdot \lambda=0$ as ${}^t \lambda \cdot \lambda=0$ by  (4.9) of loc.cit..
In particular $\tilde{\cF}^\bot \subset \Pi M_\Oo$ and a point in $T_z(\cN_{\Lambda,\bar\kappa})$ is determined by the $n-1$ parameters $c_i$ for $2\leq i\leq n$ together with the additional parameter $\gamma_1$. Now the condition $\tilde{\cF}^\bot \subset (\Lambda_{M,\Oo})^\bot$ (condition (e) of Lemma \ref{lem:G M theory of cN Lambda}) imposes further $n-m$ independent linear equations on the parameters $c_i$ for $2\leq i\leq n$. This shows that the tangent space $T_z(\cN_{\Lambda,\bar\kappa})$ has dimension less than or equal to $m$. This finishes the proof of the proposition.
\end{proof}

\subsection{Isomorphism between $\cN_\Lambda$ and $Y_{V,\bar\kappa}$}
By \cite[Lemma 6.1]{RTW}, the lattices $\breve{\Lambda}$ and $\breve{\Lambda}^\sharp$ (see \eqref{eq:breve Lambda}) are closed under $\Pi$, $\bfV$ and $\bF$, hence determine supersingular $p$-divisible strict $O_{F_0}$-modules with $O_F$-action $X_-$ and $X_+$ (denoted by $X_{\Lambda^-}$ and $X_{\Lambda^+}$ resp. in \cite[\S 6]{RTW}) of dimension $n$ over $\bar\kappa$ together with quasi-isogenies $\rho_-:X_-\rightarrow \bX$ of height $m$ and $\rho_+:X_+\rightarrow \bX$ of height $-m$. The inclusion $\breve{\Lambda}\subset\breve{\Lambda}^\sharp$ also defines an isogeny $\rho_{\Lambda}:X_-\rightarrow X_+$ of height $2m$.
Since $X_- \cong \bY^n$ as an $O_F$-module for any $\bar\kappa$-scheme $S$, on the special fiber condition (1) in Definition \ref{def:cN Lambda} is equivalent to the condition
\begin{equation}\label{eq: condition (1)'}
    (1)': \text{ The quasi-isogeny } \rho_{X,-}:=\rho^{-1}\circ (\rho_-)_S: (X_-)_S\rightarrow X \text{ is an isogeny}.
\end{equation}
This is further equivalent by loc.cit. to the condition
\begin{equation}\label{eq: condition (1)''}
    (1)'': \text{ The quasi-isogeny } \rho_{X,+}:=(\rho_+)^{-1}_S \circ\rho: X\rightarrow  (X_+)_S\text{ is an isogeny}.
\end{equation}

\begin{lemma}
The functor $\cN_{\Lambda,\bar\kappa}$ is representable by a projective scheme over $\bar\kappa$. The functor morphism $\cN_{\Lambda,\bar\kappa}\rightarrow \cN$ is a closed immersion.
\end{lemma}
\begin{proof}
 $\cZ(\Lambda)_{\bar\kappa}$ is a closed formal subscheme of $\cN$.
Since for any $\bar\kappa$-scheme $S$, Condition (1) in Definition \ref{def:cN Lambda} is equivalent to \eqref{eq: condition (1)'}, the functor $\cZ(\Lambda)_{\bar\kappa}$ can be represented by a projective scheme over $\bar\kappa$ by exact the same argument as that of \cite[Lemma 3.2]{vollaard2011supersingular}. Condition (2) of Definition \ref{def:cN Lambda} defines $\cN_{\Lambda,\bar\kappa}$ as a closed subscheme of $\cZ(\Lambda)_{\bar\kappa}$, hence is itself projective over $\bar\kappa$ and a closed formal subscheme of $\cN$. This finishes the proof of the lemma.
\end{proof}

In the following discussion we assume that $\Lambda$ has type greater than or equal to $2$.
Let $V=\Lambda^\sharp/\Lambda$ and define a symplectic form $\langle\, ,\, \rangle_V$ on $V$ as follows. For $\bar{x},\bar{y}\in V$ with lifts $x,y\in \Lambda^\sharp$, define $\langle \bar{x},\bar{y}\rangle_V$ by the image of $\pi_0 \delta\langle x,y\rangle$ in $\F_\q$ (see \S \ref{subsec:associated hermitian space}). Extend this form bilinearly to $V_{\bar\kappa}$. Note that $\tau$ induces identity on $V$ and the Frobenius $\Phi$ on $V_{\bar\kappa}$.
Let $R$ be a $\bar\kappa$-algebra and $(X,\iota,\lambda,\rho,\cF)\in \cN_\Lambda(R)$. As in the proof of \cite[Corollary 3.9]{vollaard2011supersingular}, $\mathrm{Image}( D((\rho_{\Lambda})_R))$ is a locally free direct summands of $ D((X_+)_R)$ of corank  $2m$ and
\[D((X_+)_R)/\mathrm{Image}( D((\rho_{\Lambda})_R))\cong\breve{\Lambda}^\sharp/\breve{\Lambda} \otimes_{\bar\kappa} R=V_R.\]
As $(\rho_\Lambda)_R=\rho_{X,+}\circ\rho_{X,-}$, we know $\ker(\rho_{X,+})=\ker ((\rho_\Lambda)_R)/\ker(\rho_{X,-})$ as a quotient of finite group schemes over $\Spec R$. Since $\Pi\breve{\Lambda}^\sharp\subset \breve{\Lambda}$, by relative Dieudonn\'e theory, we know $\ker(\rho_\Lambda)\subset X_-[\iota_{X_-}(\pi)]$ or equivalently $\iota_{X_-}(\pi) \cdot \ker(\rho_\Lambda)=\{0\}$. Hence $\iota(\pi) \cdot \ker(\rho_{X,+})=\{0\}$ or equivalently $\ker(\rho_{X,+})\subset X[\iota(\pi)]$. Thus there exists an isogeny $\tilde{\rho}_{X,+}:X_+\rightarrow X$ such that $\tilde{\rho}_{X,+}\circ\rho_{X,+}$ is the isogeny $\iota(\pi):X\rightarrow X$.
Recall $\Fil$ in the exact sequence \eqref{eq:Hodge filtration}.
\begin{lemma}\label{lem:Definition of U(X)}
$D(\tilde{\rho}_{X,+})^{-1}(\Fil)$ is a locally free direct summand of $D(X_+)$ that contains $D((\rho_{\Lambda})_R)$. Moreover the quotient
\begin{equation}\label{eq:U(X)}
    U(X):=D(\tilde{\rho}_{X,+})^{-1}(\Fil)/\mathrm{Image}( D((\rho_{\Lambda})_R))
\end{equation}
is a locally free direct summand of $V_R$ of rank $m$.
\end{lemma}
\begin{proof}
By universality, it suffices to check the case when $\Spec R$ is an affine sub formal scheme of $\cN_\Lambda$. In this case, by Nakayama's lemma, it suffices to check the condition on the $\bar\kappa$-points of $\cN_\Lambda$. A point $z\in \cN_\Lambda(\bar\kappa)$ corresponds to a pair $(M,M')$ as in Corollary \ref{cor:reduced point of cN Lambda}. Then the isogeny $\tilde{\rho}_{X,+}$ is induced by the map of relative Dieudonn\'e modules $\breve{\Lambda}^\sharp \rightarrow M:x\mapsto \Pi \cdot x$. Recall $\Fil=\bfV M$. So
\[D(\tilde{\rho}_{X,+})^{-1}(\Fil)=\Pi^{-1}\bfV M/\pi_0 \breve{\Lambda}^\sharp=\tau^{-1}(M)/\pi_0\breve{\Lambda}^\sharp.\]
Since $\breve{\Lambda}\subset M$, we have $\breve{\Lambda}=\tau^{-1}(\breve{\Lambda})\subset \tau^{-1}(M)$. So
\[\mathrm{Image}( D((\rho_{\Lambda})_R))=\breve{\Lambda}/\pi_0\breve{\Lambda}^\sharp\subset D(\tilde{\rho}_{X,+})^{-1}(\Fil).\]
The condition $M=M^\sharp$ is equivalent to the fact that $\Phi(U(X))$ is Lagrangian in $V$, which in turn is equivalent to the fact that $U(X)$ is Lagrangian which implies $\dim_{\bar\kappa}U(X)=m$.
\end{proof}

By Condition (2) of Definition \ref{def:cN Lambda}, we know $\mathrm{Image}( D((\rho_{\Lambda})_R))\subset D(\rho_{X,+})(q_X^{-1}(\cF))$ where $q_X: D(X)\rightarrow \Lie X$ is the natural quotient homomorphism of $R$-modules (see the proof of Corollary \ref{cor:reduced point of cN Lambda}). Define
\[\cF(X):=D(\rho_{X,+})(q_X^{-1}(\cF))/\mathrm{Image}( D((\rho_{\Lambda})_R)).\]
Then $\cF(X)$ is a locally free direct summand of $U(X)$ of rank $m-1$.
We define a map $\phi:\cN_{\Lambda,\bar\kappa}\rightarrow \Gr(m,V_{\bar\kappa})\times \Gr(m-1,V_{\bar\kappa})$ by
\[\phi:(X,\iota,\lambda,\rho,\cF)\mapsto (U(X),\cF(X))\in (\Gr(m,V_{\bar\kappa})\times \Gr(m-1,V_{\bar\kappa})) (R).\]
\begin{lemma}\label{lem:YV N bijection of k points}
$\phi$ defines a bijection between $\cN_\Lambda(\bar\kappa)$ and $Y_V(\bar\kappa)$.
\end{lemma}
\begin{proof}
A point $z\in \cN_\Lambda(\bar\kappa)$ corresponds to a pair $(M,M')$ as in Corollary \ref{cor:reduced point of cN Lambda}.
By the definition of $\phi$ we have $\phi(z)=(U,U')$ where
\[(U,U')=(\Pi^{-1} \bfV M/\breve{\Lambda}, M'/\breve{\Lambda})=(\tau^{-1}(M)/\breve{\Lambda}, M'/\breve{\Lambda})=(\Phi^{-1}(M/\breve{\Lambda}), M'/\breve{\Lambda}).\]
As in the proof of Lemma \ref{lem:Definition of U(X)}, the condition $M=M^\sharp$ is equivalent to the condition that $U$ is Lagrangian.
The condition $M'\subset M$ is equivalent to $U'\subset \Phi(U)$. The condition $M'\subset \tau^{-1}(M)$ is equivalent to $U'\subset U$. This shows that $\phi(z)\in Y_V(\bar\kappa)$.

Conversely assume $(U,U')\in Y_V(\bar\kappa)$ and let $M=\Pr^{-1}(\Phi(U))$ and $M'=\Pr^{-1}(U')$ where $\Pr:\breve{\Lambda}^\sharp\rightarrow \breve{\Lambda}^\sharp/\breve{\Lambda}$ is the natural quotient map. Then
by definition $\breve{\Lambda} \subset M'\subset M\subset \breve{\Lambda}^\sharp$, and $M=M^\sharp$ as $U$ is Lagrangian.
Since $\bfV M\subset \bfV \breve{\Lambda}^\sharp =\Pi \breve{\Lambda}^\sharp\subset \breve{\Lambda} \subset M'$, we have $\bfV M\subset M'$.
We also have
\[\Pi M \subset \Pi \breve{\Lambda}^\sharp\subset \breve{\Lambda}=\tau^{-1} \breve{\Lambda}\subset \tau^{-1}(M),\]
and
\begin{equation}\label{eq:pi M in Lambda -}
 \tau^{-1}(M)\subset \tau^{-1}(\breve{\Lambda}^\sharp)= \breve{\Lambda}^\sharp\subset \Pi^{-1} \breve{\Lambda} \subset \Pi^{-1} M.
\end{equation}
Hence $\Pi M\subset \tau^{-1}(M)\subset \Pi^{-1} M$.
This shows that $(M,M')$ satisfies the conditions in Proposition \ref{prop:k points of cNKra} and Corollary \ref{cor:reduced point of cN Lambda}. This defines the inverse of $\phi$ on the level of $\bar\kappa$-points.
Hence $\phi$ defines a bijection between $\cN_\Lambda(\bar\kappa)$ and $Y_V(\bar\kappa)$.
\end{proof}

\begin{theorem}\label{thm:moduli interpretation of DL variety}
Let $\Lambda$ be a vertex lattice of type $2m$ ($m\geq 1$) in $\bV$. Then $ \cN_\Lambda$ is reduced and the morphism
$\phi$ defines an isomorphism $\cN_\Lambda\rightarrow Y_{V,\bar\kappa}$. In particular $\cN_\Lambda$ is smooth of dimension $m$ over $\bar\kappa$.
\end{theorem}
\begin{proof}
Let $\cN^\red_\Lambda$ be the underlying reduced $\bar\kappa$-scheme of $\cN_\Lambda$. Lemma \ref{lem:YV N bijection of k points} shows that $\phi$ induces a morphism $\phi^\red:\cN^\red_\Lambda\rightarrow Y_{V,\bar\kappa}$ which is a bijection on $\bar\kappa$-points, in particular quasi-finite. Since $\phi^\red$ is a morphism between projective varieties, it is projective.
Moreover using the theory of relative displays and windows, working with Cohen rings instead of the Witt ring, we can show that $\cN_\Lambda^\red(R)=Y_V(R)$ for any field $R$ containing $\bar\kappa$ by the same proof as that of Lemma \ref{lem:YV N bijection of k points}. In particular $\phi^\red$ is birational. Being quasi-finite and proper at the same time, it is an isomorphism by Zariski's main theorem since $Y_{V,\bar\kappa}$ is normal. Now Proposition \ref{prop:reducedness of cN Lambda} implies that $\cN^\red_\Lambda=\cN_{\Lambda,\bar\kappa}$. By \cite[Lemma 10.3]{RTZ} and Corollary \ref{cor:cN Lambda defined over k}, we have $\cN_\Lambda=\cN_{\Lambda,\bar\kappa}$.
This finishes the proof of the theorem.
\end{proof}

\begin{proposition}\label{prop:cN Lambda 0}
Let $\Lambda$ be a vertex lattice of type $0$ in $\bV$. Then $\cN_\Lambda$ is the exceptional divisor $\Exc_\Lambda$ and is isomorphic to $\bP^{n-1}_{\bar\kappa}$.
\end{proposition}
\begin{proof}
Let $R$ be any $\bar\kappa$-algebra and $z$ be any point in $\cN_\Lambda(R)$ and $(X,\iota,\lambda,\rho,\cF)$ be the pullback of the universal object of $\cN$ to $z$.
As $\Lambda$ is a unimodular lattice, the quasi-isogeny $\rho_-$ has height $0$. Thus the isogeny
\[\rho_{X,-}=\rho^{-1}\circ (\rho_-)_R: (X_-)_R\rightarrow X\]
has height $0$ and is an isomorphism, hence we can identify $(X,\ldots,\rho)$ with $((X_-)_R,\ldots, (\rho_-)_R)$.
As $\Pi|_\Lambda=\bfV |_\Lambda$ for any vertex lattice $\Lambda$, and $\Lie X_-=M(X_-)/\bfV M(X_-)$, the action of $\iota(\pi)$ on $\Lie X_-$ is trivial. The point $z$ is uniquely determined by the filtration $\cF\subset \Lie X$.
Hence $\cF$ can be any rank $n-1$ locally free $R$-module on $\Lie X$. This shows that $\cN_\Lambda$ is isomorphic to $\bP^{n-1}_{\bar\kappa}$ and is in particular reduced.
Moreover if $R=\bar\kappa$, then $\rho(M(X))=\breve{\Lambda}$. This shows that $\cN_\Lambda$ is a subscheme of $\Exc_\Lambda$ according Definition \ref{def:Exc}. By the proof of Lemma \ref{lem:Exc}, we know that $\cN_\Lambda$ and $\Exc_\Lambda$ have the same $\bar\kappa$-points. As they are both reduced subscheme of $\cN$, they must be the same. This proves the proposition.
\end{proof}

\subsection{Bruhat-Tits stratification}\label{subsec:bruhat tits in general}
\begin{lemma}\label{lem:Lambda of M}
For any pair $(M,\cF)$ satisfying the condition in Proposition \ref{prop:k points of cNKra}, there is a unique vertex lattice $\Lambda(M)$ such that $\Lambda(M)\subset M$ and $\Lambda(M)$ is maximal among all such vertex lattices.
\end{lemma}
\begin{proof}
This is essentially \cite[Proposition 4.1]{RTW} as such $M$ satisfies the conditions in Proposition 2.4 of loc.cit..
\end{proof}

\begin{theorem}\label{thm:Bruhat Tits stratification}
There is a stratification of $\cN_{\mathrm{red}}$ by closed strata $\cN_\Lambda$ given by
\begin{equation}\label{eq:Bruhat Tits stratification}
  \cN_{\mathrm{red}}=\bigcup_{\Lambda} \ \cN_\Lambda.
\end{equation}
where the union is over all vertex lattices in $\bV$. We call this the Bruhat-Tits stratification of $\cN_{\mathrm{red}}$. In the following, assume that $\Lambda$ and $\Lambda'$ are vertex lattices of type greater than or equal to $2$, and $\Lambda_0$ and $\Lambda'_0$ are vertex lattices of type $0$.
\begin{enumerate}
    \item If $\Lambda\subset \Lambda'$, then $\cN_{\Lambda'}$ is a subscheme of $\cN_{\Lambda}$.
    \item The intersection of $\cN_{\Lambda'}\cap\cN_{\Lambda}$ is nonempty if and only if $\Lambda''=\Lambda+\Lambda'$ is a vertex lattice, in which case we have $\cN_{\Lambda'}\cap\cN_{\Lambda}=\cN_{\Lambda''}$.
    \item The intersection of $\cN_{\Lambda'_0}\cap\cN_{\Lambda_0}$ is always empty if $\Lambda_0\neq \Lambda'_0$.
    \item The intersection  $\cN_{\Lambda}\cap\cN_{\Lambda_0}$ is nonempty if and only if $\Lambda\subset \Lambda_0$ in which case $\cN_{\Lambda}\cap\cN_{\Lambda_0}$ is isomorphic to $\bP_{\bar\kappa}^{m-1}$ where $2m$ is the type of $\Lambda$.
\end{enumerate}
\end{theorem}
\begin{proof}
To prove \eqref{eq:Bruhat Tits stratification},
it suffices to check this on $\bar\kappa$-points. A point $z\in \cN_{\mathrm{red}}(\bar\kappa)$ corresponds to a pair $(M,M')$ as in Proposition \ref{prop:k points of cNKra}. Take $\Lambda=\Lambda(M)$ as in Lemma \ref{lem:Lambda of M}.
If $\Lambda$ has type $0$, then both $\breve{\Lambda}$ and $M$ are unimodular and $\breve{\Lambda}\subset M$, so they have to be equal.
Hence $z\in \cN_\Lambda$ by Corollary \ref{cor:reduced point of cN Lambda}.
If $\Lambda$ is not of type $0$, then $M$ is not $\tau$-invariant, hence $M'=M\cap \tau^{-1}(M)$ is uniquely determined. Since $\Lambda$ is $\tau$-invariant, $\breve{\Lambda}\subset M'$. Hence $z\in \cN_\Lambda(\bar\kappa)$ by Corollary \ref{cor:reduced point of cN Lambda}. This proves \eqref{eq:Bruhat Tits stratification}.

(1) follows immediately from Definition \ref{def:cN Lambda}.

(2). If $\Lambda''$ is a vertex lattice, then $\cN_\Lambda\cap \cN_{\Lambda'}=\cN_{\Lambda''}$ by Definition \ref{def:cN Lambda}.
Conversely if $\cN_{\Lambda'}\cap\cN_{\Lambda}(\bar\kappa)$ is nonempty, let $(M,M')\in \cN_{\Lambda'}\cap\cN_{\Lambda}(\bar\kappa)$. Then $\Lambda(M)\supset \Lambda+\Lambda'$ by the maximality of $\Lambda(M)$. Then $\Lambda+\Lambda'\subset \Lambda(M)\subset \Lambda(M)^\sharp \subset \Lambda^\sharp\cap (\Lambda')^\sharp=(\Lambda+\Lambda')^\sharp$. Hence $\Lambda+\Lambda'$ is a vertex lattice.

(3) follows directly from Corollary \ref{cor:reduced point of cN Lambda}.

(4). By Corollary \ref{cor:reduced point of cN Lambda}, a point $(M,M')\in \cN(\bar\kappa)$ is in $\cN_{\Lambda}\cap\cN_{\Lambda_0}$ if and only if $M=\Lambda_0\otimes_{O_F}O_{\breve F}$ and $\Lambda\subset M'\subset M$. This show that $\Lambda\subset\Lambda_0$ and $M'$ corresponds to a point in $\bP(\Lambda_0/\Lambda)(\bar\kappa)$.
Hence $\cN_{\Lambda}\cap\cN_{\Lambda_0}(\bar\kappa)=\bP(\Lambda_0/\Lambda)(\bar\kappa)$. Similarly one can show that
\[\cN_{\Lambda}\cap\cN_{\Lambda_0}(R)=\bP(\Lambda_0/\Lambda)(R)\]
for any $\bar\kappa$-algebra $R$. This finishes the proof of (4).
\end{proof}	
	
	\begin{proposition}\label{prop:reducedlocus}
	For a rank $r$ lattice $L\subset \bV$, the reduced subscheme $\cZ(L)_{\mathrm{red}}$ of $\cZ(L)$ is a union of Bruhat-Tits strata:
	\begin{equation}\label{eq:reduced locus cZ}
	\cZ(L)_{\mathrm{red}}=\bigcup_{L\subset \Lambda}\cN_\Lambda,
	\end{equation}
	where  the union is taken over all vertex lattices $\Lambda$ such that $L\subset \Lambda$.
	Moreover, the intersection of $\cZ(L)$ with $\cN_\Lambda$ is nonempty if and only if $L\subset \Lambda^\sharp$.
	\end{proposition}
	\begin{proof}
	The proof of \eqref{eq:reduced locus cZ} is the same as that of \cite[Proposition 3.8]{Shi1}.
	
	If $L\subset \Lambda^\sharp$ and $L$ is integral, define $\Lambda':=L+\Lambda$. Then $\Lambda'$ is a vertex lattice and $\Lambda \subset \Lambda'$. By Theorem \ref{thm:Bruhat Tits stratification} (1) and the definition of $\cZ(L)$, $\cN_{\Lambda'}$ is in the intersection of $\cZ(L)$ and $\cN_\Lambda$.
	
	Conversely if the intersection of $\cZ(L)$ and $\cN_\Lambda$ is not empty, then by \eqref{eq:reduced locus cZ} and Theorem \ref{thm:Bruhat Tits stratification}, there exists a vertex lattice $\Lambda'$ such that $\Lambda\subset \Lambda'$ and $L\subset\Lambda'$. Since $\Lambda'\subset (\Lambda')^\sharp\subset \Lambda^\sharp$, we know that $L\subset \Lambda^\sharp$. This finishes the proof of the lemma.
	\end{proof}

\section{Fourier transform: the geometric side} \label{sect:FourierGeo}
\subsection{Horizontal and vertical part of ${}^\bL\cZ(L^\flat)$}
\begin{definition}\label{def:horizontal lattice}
Let $L^\flat$ be a rank $n-1$ integral lattice in $\bV$. We say that  $L^\flat$ is horizontal if one of the following conditions is satisfied.
\begin{enumerate}
    \item $L^\flat$ is unimodular.
    \item $L^\flat$ is of the form $L^\flat=M\obot L'$ where $M$ is a unimodular sublattice of rank $n-2$ such that $(M_F)^\bot$ (the perpendicular complement of $M_F$ in $\bV$) is nonsplit.
\end{enumerate}
We denote the set of horizontal lattices by $\Hor$.
\end{definition}
\begin{lemma}\label{lem:alternative def of horizontal lattice}
Let $L^\flat$ be a rank $n-1$ lattice in $\bV$. Then
$L^\flat$ is horizontal if and only if there is a unique vertex lattice $\Lambda$ which contains $L^\flat$. If this is the case, $\Lambda$ is of type $0$.
\end{lemma}
\begin{proof}
We first prove the ``only if'' direction. Let $\Lambda$ be any vertex lattice containing $L^\flat$.
If $L^\flat$ is unimodular, then $\Lambda$ has to be of the form $L^\flat\obot L'$ where $L'$ is the unique unimodular lattice in $(L^\flat_F)^\bot$. If $L^\flat$ is of the form $M\obot L'$ such that $M$ is of rank $n-2$ and  $(M_F)^\bot$ is nonsplit, then the proof of \cite[Theorem 3.10]{Shi1} implies that there is a unique vertex lattice $\Lambda'$ in $(M_F)^\bot$ which is of unimodular (this corresponds to the fact that the Bruhat-Tits building of $(M_F)^\bot$ has only one point). Then $\Lambda$ must be of the form $M\obot\Lambda'$. In both cases, $\Lambda$ is unique and is of type $0$.

We now prove the ``if'' direction.  If $t(L^\flat)\geq 2$, then there exist a type $2$ vertex lattice $\Lambda_2$ containing $L^\flat$ and any type $0$ vertex lattice containing $\Lambda_2$ (there are $q+1$ of them) also contains $L^\flat$. Hence $t(L^\flat)\leq 1$ and $L^\flat$ is of the form  $M\obot L'$ such that $M$ is of rank $n-2$. If $(M_F)^\bot$ is split and $\val(L')>0$, then \cite[Corollary 3.11]{HSY} implies that there are more than one type $0$ vertex lattices $\Lambda'$ in $(M_F)^\bot$ containing $L'$. For any such $\Lambda'$, $M\obot \Lambda'$ is a vertex lattice of type $0$ containing $L^\flat$. This shows that in order for such $\Lambda$ to be unique, $L^\flat$ must satisfies the conditions in Definition \ref{def:horizontal lattice}. The lemma is proved.
\end{proof}

For a rank $n-1$ lattice $L^\flat$ in $\bV$, define	\begin{equation}\label{eq:horizontalmodules}
		\Hor(L^\flat)\coloneqq \{M^\flat \in \Hor \mid L^\flat\subset M^\flat\}.
\end{equation}

When $\mathrm{dim}(\bV)=2$ and $\chi(\bV)=-1$, for $y\in \bV$,
define
\[\tZ(y)^\circ\coloneqq \left\{\begin{array}{cc}
		\cZ_{\val(y)}^+ \sqcup\cZ_{\val(y)}^-   & \text{ if } \val(y)>0,  \\
		\cZ_0 & \text{ if } \val(y)=0.
\end{array}\right.\]
Here $\cZ_0\cong \SpfOF$ and $\cZ_{s}^+\cong \cZ_{s}^-\cong \Spf W_s$ are quasi-canonical lifting cycles defined in \cite[\S3]{Shi2} where $W_s$ is a totally ramified abelian extension of $O_{\breve F}$ of degree $q^s$.
When $\mathrm{dim}(\bV)=2$ and $\chi(\bV)=1$, for $y\in \bV^{=0}$, define
$\tZ(y)^\circ$ to be $\cZ^h(y)$, where $\cZ^h(y)\cong \Spf O_{\breve F}$ is as in \cite[Theorem 4.1]{HSY}. In all cases, $\tZ(y)^\circ$ is a closed subscheme of $\cN_2$.

For a $M^\flat \in \Hor$, we can decompose $M^\flat$ as $M\obot \spa\{y\}$ where $M$ is unimodular and $\val(y)$ has to be zero if $(M_F)^\bot$ is split. By \cite[Proposition 2.6]{HSY3}, the unimodular lattice $M$ induces a closed embedding $\cN_2\hookrightarrow \cN_n$.
We define $\tZ(M^\flat)^\circ$ to be the image of the composed embedding  $\tZ(y)^\circ\hookrightarrow \cN_2 \hookrightarrow \cN_n$ where $\tZ(y)^\circ$ is the closed formal subscheme of $\cN_2$ defined above. Moreover by loc. cit., the definition of $\tZ(M^\flat)^\circ$ is independent of the choice of $M$.
The following is \cite[Theorem 4.2]{HSY3}.
	\begin{theorem}\label{thm:horizontalpart}
		Let $L^\flat$ be a rank $n-1$ non-degenerate integral lattice in $\bV$, then
		\begin{equation}\label{eq:horizontalpart}
			\cZ(L^\flat)_\mathscr{H}=\bigcup_{M^\flat \in \Hor(L^\flat)} \tZ(M^\flat)^\circ.
		\end{equation}
		In particular, $\cZ(L^\flat)_{\mathscr{H}}$ is  of pure dimension $1$. We have the following identity in $\Gr^{n-1}K_0^{\cZ(L^\flat)}(\cN)$:
		\begin{equation}\label{eq:derived horizontal part}
		    [\Oo_{\cZ(L^\flat)_\mathscr{H}}]=\sum_{M^\flat \in \Hor(L^\flat)} [\Oo_{\tZ(M^\flat)^\circ}].
		\end{equation}
	\end{theorem}
	
\begin{lemma}\label{lem:filtration is ideal}
For any formal subscheme $Z$ of $\cN$ and $0\leq i \leq n$, $F^i K_0^Z(\cN)$ is an ideal in $K_0(\cN)$.
\end{lemma}
\begin{proof}
By definition $F^i K_0^Z(\cN)$ is generated by elements of the form $[\cF^\bullet]$ where $\cF^\bullet$ is a finite complex of locally free coherent $\Oo_{\cN}$-modules acyclic outside a sub formal scheme $Y$ of $Z$ such that the codimension of $Y$ in $\cN$ is greater than or equal to $i$.
By Kunneth formula for chain complexes, the product complex $\cF^\bullet \otimes_{\Oo_\cN} \mathcal{K}^\bullet$ is acyclic outside $Y$ as well for any finite complexes of locally free coherent $\Oo_{\cN}$-modules $\mathcal{K}^\bullet$. This proves the lemma.
\end{proof}

By Lemma \ref{lem:filtration is ideal}, for any formal subscheme $Z$ of $\cN$, we can define a quotient ring (not necessary with identity)
\begin{equation}
    \Gr' K_0^Z(\cN):=K_0^Z(\cN)/ F^n K_0^Z(\cN).
\end{equation}
In particular $\Gr^{n-1} K_0^Z(\cN)=F^{n-1} K_0^Z(\cN)/F^{n} K_0^Z(\cN)$ is a subgroup of $\Gr' K_0^Z(\cN)$.

Let $L^\flat$ be a rank $n-1$ non-degenerate integral lattice.
Since $\cZ(L^\flat)_\mathscr{H}$ is one-dimensional,  the intersection $\cZ(L^\flat)_\mathscr{H}\cap \cZ(L^\flat)_\mathscr{V}$ must be $0$-dimensional if nonempty. It follows that there is a decomposition
\begin{equation}\label{eq:Gr' decomposition}
    \Gr'K_0^{\cZ(L^\flat)}(\cN)=\Gr'K_0^{\cZ(L^\flat)_{\mathscr{H}}}(\cN)\oplus \Gr'K_0^{\cZ(L^\flat)_{\mathscr{V}}}(\cN).
\end{equation}
Under this  decomposition, we have
\begin{equation}\label{eq:horizontal vertical decomposition}
    {}^\bL\cZ(L^\flat)={}^\bL\cZ(L^\flat)_{\mathscr{H}}+{}^\bL\cZ(L^\flat)_{\mathscr{V}}\in \Gr' K_0^{\cZ(L^\flat)}(\cN),
\end{equation}
where we denote by the same notation the image of ${}^\bL\cZ(L^\flat)$ under the natural quotient map $K_0^{\cZ(L^\flat)}(\cN)\rightarrow \Gr'K_0^{\cZ(L^\flat)}(\cN)$.
It follows that the element ${}^\bL\cZ(L^\flat)_{\mathscr{V}}\in \Gr' K_0^{\cZ(L^\flat)}(\cN)$ is canonically defined although $\cZ(L^\flat)_{\mathscr{V}}$ depends on the choice of a large integer $m\gg 0$.

Since $\cZ(L^\flat)_{\mathscr{H}}$ has expected dimension, ${}^\bL \cZ(L^\flat)_{\mathscr{H}}$ is in fact in $\Gr^{n-1}K_0^{\cZ(L^\flat)}(\cN)$ and is represented by the structure sheaf of $\cZ(L^\flat)_{\mathscr{H}}$. In order to match the analytic side of our conjecture, we need to slightly modify ${}^\bL\cZ(L^\flat)_{\mathscr{H}}$.

\begin{definition}\label{def:horizontal difference cycle}
Let $L^\flat$ be a horizontal lattice in $\bV$. Define ${}^\bL\cZ(L^\flat)^\circ\in \Gr'K_0^{\cZ(L^\flat)}(\cN)$ by
\[{}^\bL\cZ(L^\flat)^\circ=
\begin{cases}
[\Oo_{\tZ(L^\flat)^\circ}]+\frac{1-(-1)^{n-1}}{2}[\Oo_{\bP_\Lambda}] & \text{ if } L^\flat \text{ is unimodular},\\
[\Oo_{\tZ(L^\flat)^\circ}]+[\Oo_{\bP_\Lambda}] & \text{ otherwise},
\end{cases}\]
where $\Lambda$ is the unique type $0$ vertex lattice containing $L^\flat$ as in Lemma \ref{lem:alternative def of horizontal lattice} and $\bP_\Lambda$ is a projective line over $\bar\kappa$ in $\Exc_\Lambda$.
\end{definition}
\begin{remark}
${}^\bL\cZ(L^\flat)^\circ$ is the difference cycle $\cD(L^\flat)$ defined in \cite[Definition 2.15]{HSY3}.
\end{remark}

\begin{definition}\label{def:horizontal and vertical intersection number}
Let $L^\flat$ be a rank $n-1$ non-degenerate integral lattice. Define ${}^\bL\cZ(L^\flat)_{\mathscr{H}}^*\in \Gr'K_0^{\cZ(L^\flat)}(\cN)$ by
\[{}^\bL\cZ(L^\flat)_{\mathscr{H}}^*:=\sum_{M^\flat\in \Hor(L^\flat)} {}^\bL\cZ(M^\flat)^\circ ,\]
where $\cZ(M^\flat)_{\mathscr{H}}^\circ$ is as in Definition \ref{def:horizontal difference cycle}.
Define the modified vertical part of the derived special cycle ${}^\bL\cZ(L^\flat)$ by
\[{}^\bL\cZ(L^\flat)_{\mathscr{V}}^*:={}^\bL\cZ(L^\flat)-{}^\bL\cZ(L^\flat)_{\mathscr{H}}^*\in \Gr'K_0^{\cZ(L^\flat)}(\cN).\]
For any $x\in \bV\setminus L^\flat_F$, define
\begin{align}\label{eq: def of Int_sH}
    \Int_{L^{\flat},\mathscr{H}}(x):=\chi(\cN,{}^\bL\cZ(L^\flat)_{\mathscr{H}}^*\cdot [\Oo_{\cZ(x)}]), \hbox{ and } \Int_{L^{\flat},\mathscr{V}}(x):=\chi(\cN,{}^\bL\cZ(L^\flat)_{\mathscr{V}}^*\cdot [\Oo_{\cZ(x)}]).
\end{align}
\end{definition}

\begin{lemma}\label{lem:derived vertical part}
For a rank $n-1$ non-degenerate integral lattice $L^\flat$, we have
\[{}^\bL\cZ(L^\flat)_{\mathscr{V}}^*\in \Gr' K_0^{\cZ(L^\flat)_\mathscr{V}}(\cN).\]
\end{lemma}
\begin{proof}
By the definition of ${}^\bL\cZ(L^\flat)_{\mathscr{V}}^*$, the decomposition \eqref{eq:horizontal vertical decomposition} and Theorem \ref{thm:horizontalpart}, we have
\begin{align*}
    {}^\bL\cZ(L^\flat)_{\mathscr{V}}^*=&{}^\bL\cZ(L^\flat)_{\mathscr{V}}+{}^\bL\cZ(L^\flat)_{\mathscr{H}}-\sum_{M^\flat\in \Hor(L^\flat)} {}^\bL\cZ(M^\flat)^\circ \\
    =&{}^\bL\cZ(L^\flat)_{\mathscr{V}}+\sum_{M^\flat\in \Hor(L^\flat)} ([\Oo_{\cZ(M^\flat)^\circ}]-{}^\bL\cZ(M^\flat)^\circ).
\end{align*}
We know all terms in the last expression are in $\Gr' K_0^{\cZ(L^\flat)_\mathscr{V}}(\cN)$ by the definition of ${}^\bL\cZ(L^\flat)_{\mathscr{V}}$ and Definition \ref{def:horizontal difference cycle}.
\end{proof}

\begin{lemma}\label{lem:modified horizontal part}
If $L^\flat$ is a horizontal lattice of rank $n-1$ in $\bV$, then
\begin{equation}\label{eq:modified horizontal part}
    {}^\bL\cZ(L^\flat)_{\mathscr{H}}^*={}^\bL\cZ(L^\flat).
\end{equation}
In particular for any $x\in \bV \setminus L_F^\flat$ we have
\[\Int_{L^{\flat},\mathscr{H}}(x)=\Int_{L^{\flat}}(x).\]
\end{lemma}
\begin{proof}
Let  $\Lambda$ be the unique type $0$ vertex lattice containing $L^\flat$ as indicated by Lemma \ref{lem:alternative def of horizontal lattice}. Then $\Lambda\cap L^\flat_F$ is the unique unimodular lattice in $\Hor(L^\flat)$.
By Theorem \ref{thm:horizontalpart}, we have
\[{}^\bL\cZ(L^\flat)_{\mathscr{H}}^*-{}^\bL\cZ(L^\flat)=(m-1+
\frac{1-(-1)^{n-1}}{2}) [\Oo_{\bP_\Lambda}]-{}^\bL\cZ(L^\flat)_{\mathscr{V}},\]
where $m:=|\Hor(L^\flat)|$.
By Proposition \ref{prop:reducedlocus} and Lemma \ref{lem:alternative def of horizontal lattice}, we know that ${}^\bL Z(L^\flat)_{\mathscr{V}}\in \Gr' K_0^{\cN_\Lambda} (\cN)$.  \cite[Corollary 3.5]{HSY3} implies that in fact ${}^\bL Z(L^\flat)_{\mathscr{V}}\in \Gr^{n-1} K_0^{\cN_\Lambda} (\cN)$, hence
\[{}^\bL\cZ(L^\flat)_{\mathscr{V}}=m'[\Oo_{\bP_\Lambda}]\]
for some integer $m'$. In order to prove \eqref{eq:modified horizontal part}, it suffices to show
\begin{equation}\label{eq:m'(L flat)}
    m'=m-
\frac{1+(-1)^{n-1}}{2}.
\end{equation}
Now assume $L^\flat=M\obot L'$ where $M$ is unimodular and of rank $n-2$ and $\val(L')=a$.
Then $m=a+1$. By \cite[Lemma 4.4]{HSY3}, we know that
\[\chi(\cN,{}^\bL\cZ(L^\flat)_{\mathscr{H}}\cdot [\Oo_{\cN_\Lambda}])=2a+1=2m-1.\]
By \cite[Corollary 3.7]{HSY3}, we know
\[\chi(\cN,{}^\bL\cZ(L^\flat)_{\mathscr{V}}\cdot [\Oo_{\cN_\Lambda}])=m'\cdot\chi(\cN,[\Oo_{\bP_\Lambda}]\cdot [\Oo_{\cN_\Lambda}])=-2m'.\]
On the other hand, by \cite[Corollary 3.6]{HSY3},
\[\chi(\cN,{}^\bL\cZ(L^\flat)_{\mathscr{H}}\cdot [\Oo_{\cN_\Lambda}])+\chi(\cN,{}^\bL\cZ(L^\flat)_{\mathscr{V}}\cdot [\Oo_{\cN_\Lambda}]) =\chi(\cN,{}^\bL\cZ(L^\flat)\cdot [\Oo_{\cN_\Lambda}])=(-1)^{n-1}.\]
Combine the above equations, we get \eqref{eq:m'(L flat)}.
\end{proof}

\subsection{Hermitian lattices and Fourier transform}\label{subsec:hermfourier}
We fix an additive character $\psi:F_0\rightarrow \C^\times$ whose conductor is $O_{F_0}$.
Recall that the Fourier transform with respect to $\psi$ is defined by
	\begin{equation}\label{eq:fourier transform}
	    \widehat{\varphi}(x)=\int_{\bV} \varphi(y) \cdot \psi(\tr_{F/F_0} (x,y)) d\mu (y),
	\end{equation}
where $d\mu$ is the unique self-dual Haar measure on $\bV$ with respect to this transform.
For a lattice $L$ in $\bV$ we use $L^\vee$ to denote its dual under the quadratic form
$\tr_{F/F_0}((\, ,\, ))$. The following lemma is well-known and easy to check.
\begin{lemma}\label{lem:Fourier transform}
Let $L\subset \bV$ be a lattice of rank $n$ and $1_L\in \sS(\bV)$ be its characteristic function. Then
\[\widehat{1_L}=\vol(L,d\mu) \cdot 1_{L^\vee}.\]
\end{lemma}

\begin{lemma}\label{lem:translation invariant and support}
Let $L$ be a rank $n$ lattice in $\bV$.
A function $\varphi\in\sS(\bV)$ is $L$-invariant (invariant under the translation of $L$) if and only if its Fourier transform $\widehat{\varphi}$ is supported on $L^\vee$.
\end{lemma}
\begin{proof}
We first prove the ``only if" direction.
For any $\mu\in \bV$, let $\bar\mu$ be its image in the quotient $\bV/L$. Define
\[L(\bar\mu):=\mu+L.\]
 Any $L$-invariant function  $\phi \in  \sS(\bV)$ is a linear combination of the characteristic functions $1_{L(\bar\mu)}$. So it suffices to assume $\phi= 1_{L(\bar\mu)}$. In this case,
\[\widehat{\phi}(x)=\psi(\tr_{F/F_0} (x,\mu)) \cdot\widehat{1_L}(x).\]
So $\widehat{\phi}(x)$ is supported on $L^\vee$ by Lemma \ref{lem:Fourier transform}. This proves the ``only if" direction.

For the ``if" direction.  It suffices to show that if $\varphi$ is supported on $L^\vee$, then $\widehat{\varphi}$ is $L$-invariant.
For any $z\in L$, we have
\[ \widehat{\varphi}(x+z)=\int_{\bV} \varphi(y) \cdot \psi(\tr_{F/F_0} (x,y)) \cdot \psi(\tr_{F/F_0} (z,y)) d\mu (y).\]
Since $\psi(\tr_{F/F_0}(z,y))=1$ for any $z\in L$ and $y\in L^\vee$ and $\varphi$ is supported on $L^\vee$, the above is equal to $ \widehat{\varphi}(x)$.
This finishes the proof of the lemma.
\end{proof}

For an integer $m$, recall that
\[
    \bV^{\geq m}=\{x\in \bV\mid \val(x)\geq m\}.
\]
\begin{definition}\label{def:S(V) geq -1}
Define $\sS(\bV)^{\geq m}$ to be the subspace of $\sS(\bV)$ consisting of functions $\varphi$ such that $\widehat{\varphi}$ is supported on $\bV^{\geq m}$.
\end{definition}

\begin{lemma}\label{lem:Lambda invariant function is geq -1}
Let $\Lambda$ be a vertex lattice in $\bV$.
Any $\Lambda$-invariant function in $ \sS(\bV)$ is in $\sS(\bV)^{\geq -1}$.
\end{lemma}
\begin{proof}
By Lemma \ref{lem:translation invariant and support}, it suffices to show that $\Lambda^\vee\subset \bV^{\geq -1}$. Since $\Lambda$ is a vertex lattice, we have
\[\Lambda^\sharp= H^t\obot I_{n-2t},\]
for some $t$. Simple calculation gives then
\[\Lambda^\vee=\frac{1}\pi \Lambda^\sharp =\frac{1}{\pi}H^t\obot \frac{1}{\pi} I_{n-2t} \subset \bV^{\geq -1}.\]
\end{proof}

\subsection{Fourier transform of $\Int_{L^\flat,\mathscr{V}}$}
\begin{theorem}\label{thm:higher modularity}
Let $\Lambda$ be a vertex lattice and $\mathcal{K}\in K^{\cN_\Lambda}_0(\cN)$. For any $x\in \bV\setminus\{0\}$, the function that takes $x$ to $ \mathcal{K}\cdot [\Oo_{\cZ(x)}] \in K^{\cN_\Lambda}_0(\cN)$ is $\Lambda$-invariant. More precisely,  for any  $y\in \Lambda$ such that $x+y\neq 0$, we have
\begin{equation}\label{eq:K Lambda invariant}
  \mathcal{K}\cdot [\Oo_{\cZ(x)}]=\mathcal{K}\cdot [\Oo_{\cZ(x+y)}] .
\end{equation}
Moreover, the function
\[\Int_\mathcal{K}(x):=\chi(\cN,\mathcal{K}\cdot[\Oo_{\cZ(x)}])\]
extends to a $\Lambda$-invariant function in $\sS(\bV)^{\geq -1}$.
\end{theorem}
\begin{proof}
Any element $\mathcal{K}\in K^{\cN_\Lambda}_0(\cN)\cong K'_0(\cN_\Lambda)$ is a sum of elements of the form $[\cF]$ where $[\cF]$ is a coherent sheaf of $\Oo_{\cN_\Lambda}$-module. Hence it suffices to prove the theorem for $\mathcal{K}=[\cF]$.
By \cite[Corollary C]{Ho2}, we know
\[[\Oo_{\cZ(y)}\otimes_{\Oo_\cN}^\bL \Oo_{\cZ(x)}]=[\Oo_{\cZ(y)}\otimes_{\Oo_\cN}^\bL \Oo_{\cZ(x+y)}].\]
For any $y \in  \Lambda$ with $x+y\ne 0$, $\cN_\Lambda$ is a subscheme of $\cZ(y)$ by Proposition \ref{prop:reducedlocus}. Hence we have
\begin{align*}
        \mathcal{K}\cdot[\Oo_{\cZ(x)}]
        =&[\cF\otimes_{\Oo_\cN}^\bL \Oo_{\cZ(x)}]\\
    =&[\cF \otimes_{\Oo_{\cN_\Lambda}}^\bL \Oo_{\cN_\Lambda}\otimes_{\Oo_{\cZ(y)}}\Oo_{\cZ(y)}\otimes_{\Oo_\cN}^\bL\Oo_{\cZ(x)}]\\
    =&[\cF \otimes_{\Oo_{\cN_\Lambda}}^\bL \Oo_{\cN_\Lambda}\otimes_{\Oo_{\cZ(y)}}\Oo_{\cZ(y)}\otimes_{\Oo_\cN}^\bL\Oo_{\cZ(x+y)}]\\
    =&\mathcal{K}\cdot[\Oo_{\cZ(x+y)}]
\end{align*}
We have proved the $\Lambda$-invariance of $\mathcal{K}\cdot [\Oo_{\cZ(x)}]$. It follows that $\Int_\mathcal{K}(x)$ is also $\Lambda$-invariant. Hence we can define $\Int_\mathcal{K}(0)$ to be $\Int_\mathcal{K}(x)$ for any $0\ne x\in \Lambda$ and obtain a (unique) $\Lambda$-invariant function (still denoted by $\Int_\mathcal{K}(x)$) for all $x\in\bV$. In particular $\Int_\mathcal{K}(x)$ is locally constant.
If $x\notin \Lambda^\sharp$, by Proposition \ref{prop:reducedlocus}, the intersection of $\cZ(x)$ with $\cN_\Lambda$ is empty, which implies $\Int_\mathcal{K}(x)=0$. This shows that the function $\Int_\mathcal{K}(x)$ is compactly supported. Hence it is in $\sS(\bV)$ and is in fact in $\sS(\bV)^{\geq -1}$ by Lemma \ref{lem:Lambda invariant function is geq -1}.
This finishes the proof of the theorem.
\end{proof}

\begin{theorem}\label{thm:Int Gamma_0 invariant}
For every non-degenerate lattice $L^\flat$ of $\bV$ of rank $n-1$, the function  $\mathrm{Int}_{L^\flat,\mathscr{V}}$ on $\bV\setminus L^\flat_F$ can be extended to an element in $\sS(\bV)^{\geq -1}$ which we denote by the same notation.
\end{theorem}
\begin{proof}
Lemmas \ref{lem:derived vertical part} and \ref{lem:cZ_vsupportedoncN_red} imply that ${}^\bL\cZ(L^\flat)_{\mathscr{V}}^*\in \Gr'K_0^{\cN_\red}(\cN)\cap \Gr'K_0^{\cZ(L^\flat)_\mathscr{V}}(\cN)$.
Lemma \ref{lem:cZnoetherian} implies that  there exist finitely many classes $\mathcal{K}_i\in \Gr' K_0^{\cN_\mathrm{red}}(\cN)_\Q$ together with $C_i\in \Q$ such that
\[\Int_{L^\flat,\mathscr{V}}(x)\\
    =\sum_{i} C_i \cdot \chi(\cN,\mathcal{K}_i \cdot[\Oo_{\cZ(x)}]).\]
By Theorem \ref{thm:Bruhat Tits stratification} we may assume that $\mathcal{K}_i$ is supported on some $\cN_\Lambda$. Now we can apply Theorem \ref{thm:higher modularity} to conclude the proof.
\end{proof}

\subsection{Partial Fourier transform}
Let $L^\flat$ be a rank $n-1$ non-degenerate  lattice in $\bV$. Let $\bW=(L^\flat_F)^\bot$.
For any function $\varphi$ defined on $\bV\setminus L^\flat_F$, we define its partial Fourier transform $\varphi^\bot$ as a function on $\bW\setminus \{0\}$ by
\begin{equation}\label{eq:partial fourier transform}
    \varphi^\bot(x):=\int_{L^\flat_F} \varphi(x+y)dy, \quad  \forall \,   x\in \bW.
\end{equation}

\begin{theorem}\label{lem:Int vertical is schwartz}
The partial Fourier transform $\Int_{L^\flat,\mathscr{V}}^\bot\in \sS(\bW)^{\geq -1}$ and is $\bW^{\ge 0}$-invariant.  In particular it is constant on $\bW^{\ge 0}$.
\end{theorem}
\begin{proof}
It is easy to see that partial Fourier transform maps $\mathscr{S}(\bV)$ to $\mathscr{S}(\bW)$.
It remains to show that the Fourier transform of $\Int_{L^\flat,\mathscr{V}}^\bot\in \sS(\bW)$ is supported on $\bW^{\geq -1}$. For $x\in \bW$, we have
\[\widehat{\Int_{L^\flat,\mathscr{V}}^\bot}(x)=\widehat{\Int_{L^\flat,\mathscr{V}}}(x),\]
where $\widehat{\Int_{L^\flat,\mathscr{V}}}(x)$ is the Fourier transform of $ \Int_{L^\flat,\mathscr{V}}\in \sS(\bV)$.
Since $\widehat{\Int_{L^\flat,\mathscr{V}}}$ is supported on $\bV^{\geq -1}$ by Theorem \ref{thm:Int Gamma_0 invariant}, we know that $\widehat{\Int_{L^\flat,\mathscr{V}}^\bot}(x)$ is supported on $\bW^{\geq -1}$.

Since $\bW$ is one-dimensional, $\bW^{\geq m}$ is a full rank lattice in $\bW$ for any $m\in\Z$. By Lemma \ref{lem:translation invariant and support} and what we just proved, $\Int_{L^\flat,\mathscr{V}}^\bot$ is invariant under the translation of $(\bW^{\geq -1})^\vee=\bW^{\geq 0}$.
\end{proof}

\section{Review of local densities and primitive local densities}
In this section, we recall various explicit formulas of local density polynomials following Section 5 of \cite{HSY3}.
\subsection{Basic properties of local density and primitive local density polynomials}
\begin{definition}\label{def: Den}
Let $M$ and $L$ be two hermitian $O_F$-lattices of rank $m$ and $n$ respectively. Let $a$ be an integer such that $(x, y) \in  \pi_0^{-a} \partial_{F/F_0}^{-1}$ for $x, y \in M$ or $x, y \in L$. Define the local density of $M$ representing $L$ as
$$
\mathrm{Den}(M, L):=\lim _{d \rightarrow \infty} \frac{\left|\mathrm{Herm}_{L, M}(O_{F_0} /(\pi_0^{ d+a}))\right|}{q^{2(d+a)nm-dn^2}},
$$
which is independent of the choice of $a$. Here
$\mathrm{Herm}_{L, M}(O_{F_0} /(\pi_0^{ d+a}))$ is given by the set
	\begin{align*}
	\{ \phi \in  \mathrm{Hom}&_{O_F}(L/\pi_0^{d+a} L, M/\pi_0^{d+a} M) \mid     (\phi(x),\phi(y))\equiv (x,y) \mod (\pi_0^{d}\partial_{F/F_0}^{-1}),\,  x,y \in L\}.
	\end{align*}
\end{definition}

In this paper, we only deal with the case when we can and will choose $a=0$.
It is well-known that there is a local density polynomial $\den(M,L,X) \in \Q[X]$  such that 	\begin{equation}\label{eq:local density M,L}
	\den(M, L, \q^{-2k})=\den(M\obot H^k,L).
	\end{equation}
Moreover, we denote $\den(M, L)=\den(M, L, 1)$ and
	\begin{equation}
		\den'(M, L) \coloneqq-2\cdot \frac{\partial}{\partial X} \den(M,L, X)|_{X=1}.
	\end{equation}
Similarly, the primitive local density polynomial $\Pden(M, L, X)$ is defined  to be the polynomial in $\Q[X]$ such that
	\begin{equation}\label{eq:prim local density M,L}
	\Pden(M, L, \q^{-2k})=\lim_{d\to \infty} \q^{-d(2 n(m+2k)-n^2)} |\mathrm{Pherm}_{L,M\obot H^k}(O_{F_0}/(\pi_0^d))|,
	\end{equation}
	where
	\begin{align*}
	    \mathrm{Pherm}_{L,M\obot H^k}(O_{F_0}/(\pi_0^d))&\coloneqq\{ \phi \in  	\mathrm{Herm}_{L,M\obot H^k}(O_{F_0}/(\pi_0^d)) \mid  \text{ $\phi$ is primitive}\}.
	\end{align*}
Here we recall that $\phi\in	\mathrm{Herm}_{L,M\obot H^k}(O_{F_0}/(\pi_0^d))$ is primitive if $\dim_{\mathbb{F}_\q}((\phi (L)+\pi (M\obot \cH^k))/\pi (M\obot \cH^k)=n$. In particular, we have $\Pden(M,M)=\den(M,M)$ since any $\phi\in \Herm_{M,M}(O_{F_0}/(\pi_0^d))$ is primitive for large enough $d$.

Recall that without explicit mentioning, we assume $\epsilon=\chi(L)$. As an analogue of \eqref{eq: def of Den'} and \eqref{eq: def of den(L)}, we define
\begin{align}  \label{eq:PdenDerivative}
  \Pden'(L) \coloneqq-2 \cdot \frac{\frac{\rd}{\rd X}\big|_{X=1}\Pden(I_{n},L, X)}{\den(I_n,I_n)}  \text{ and }	 \Pden_{t}(L)\coloneqq\frac{\Pden(\Lambda_t^{\sharp},L) }{\den(\Lambda_t^{\sharp},\Lambda_t^{\sharp})}.
  \end{align}
To save notation, we simply denote $\Pden_0(L)$ by $\Pden(L)$.
We define
			\begin{equation}\label{eq: def of ppden}
	\ppden(L) \coloneqq \Pden'(L)+\displaystyle\sum_{j=1}^{\frac{t_{\max}}{2}}c_{2j} \cdot\Pden_{2j}(L),
	\end{equation}
 where $c_{2j}$ is as in \eqref{eq: def of pdenL}.
	
\begin{lemma}\label{lem: vanish val le -1}
Let $L$ be a lattice. If there exists $x\in L$ such $\val(x)\le -1$, then $\pden(L)=\den'(L)=\ppden(L)=\Pden'(L)=0$.
\end{lemma}	
\begin{proof}
Assume $M\cong I_{n}$ or $M\cong \Lambda_{2t}$ for some $t$. Then $\den(M\obot \cH^k,L)=0$ and $\Pden(M\obot \cH^k,L)=0$ since there is no vector in $M$ with valuation less than or equal to $-1$.
\end{proof}

	Now we record several results that describe the relation between local density and primitive local density polynomials.

	\begin{lemma}\cite[Lemma 5.1]{HSY3}\label{lem: ind formula to prim ld}
Let $M$ and $L$ be lattices of rank $m$ and $n$. Then we have
		$$\den(M,L,X)=\sum_{L\subset L'\subset L_{F}}(\q^{n-m}X)^{\ell(L'/L)}\Pden(M,L',X),$$
		where $\ell(L'/L)=\mathrm{length}_{O_F}L'/L$. Here $\Pden(M,L',X)=0$ for $L'$ with fundamental invariant less than the smallest fundamental invariant of $M$. In particular, the summation is finite.
	\end{lemma}
		\begin{corollary}\label{cor: decomp of pden(L)}
Let $L$ be a lattice. We have the following identity:
		$$
		\pden(L)=\sum_{L \subset L' \subset L_{F}}\ppden(L').
		$$
	\end{corollary}
	\begin{proof}
 Since  $\Pden(I_n,L',1)=\Pden(I_n,L')=0$, we have by Lemma \ref{lem: ind formula to prim ld}
    	\begin{align*}
    	&-2\frac{\rd}{\rd X}\bigg|_{X=1}\den(I_n,L,X) = -2\sum_{L\subset L'\subset L_{F}} \frac{\rd}{\rd X}\bigg|_{X=1}\Pden(I_n,L',X)=\sum_{L\subset L'\subset L_{F}}  \Pden'(I_n,L').
\end{align*}
Similarly, according to Lemma \ref{lem: ind formula to prim ld}, we have
\begin{align*}
    \den(\Lambda_{2j},L)=\sum_{L\subset L'\subset L_{F}}\Pden(\Lambda_{2j},L')
\end{align*}
for $0< j\le t_{\max}/2$. Now the corollary follows from \eqref{eq: def of pdenL} and \eqref{eq: def of ppden}.
	\end{proof}

Conversely, the primitive local density polynomial is a linear combination of local density polynomials.
	\begin{theorem}\cite[Theorem 5.2]{HSY3}\label{thm:ind formula reducing valuation}
	Let $M$ and $L$ be lattices of rank $m$ and $n$. We have
		\begin{align*}
		\Pden(M, L,X)
			= \sum_{i=0}^{n} (-1)^{i} \q^{i(i-1)/2+i(n-m)}X^{i}
	 \sum_{\substack{L \subset L' \subset \pi^{-1}L \\ \ell(L'/L)=i}} \den(M, L',X).
		\end{align*}
	\end{theorem}

\begin{corollary}\label{cor: ind structure of partial Den(T)} Let $L$ be a lattice of rank $n$.		Then
		\begin{align*}	
		\ppden(L) = \sum_{i=0}^{n} (-1)^{i} \q^{i(i-1)/2}
			\sum_{\substack{L \subset L' \subset \pi^{-1} L \\ \ell(L'/L)=i}} \pden(L').
		\end{align*}
	\end{corollary}

Recall that for two lattices $L,L'\subset \bV$ of rank $n$,  $n(L',L)=|\{L''\subset L_F\mid  L\subset L'', L''\cong L'\}|. $
 	\begin{lemma}\label{lem: 0 vanish of error term}
For two lattices $L$ and $M$ of the same rank $n$, we have
		\begin{align}\label{eq: Pden(M,L)}
		    \Pden(M,L)=\begin{cases}
		    \den(M,L) & \text{ if $M\cong L$},\\
		    0 & \text{ if $M\not\cong L$}.
		    \end{cases}
		\end{align}
Moreover,
    	\begin{align*}
			\den(M,L)=n(M,L)\cdot \den(M,M).
			\end{align*}
In particular, if $\chi(M)\not =\chi(L)$, then $\den(M,L)=0$.			
	\end{lemma}
	\begin{proof}
First of all, for $M\cong L$, $\Pden(M,L)=\den(M,L)$ by the definition of primitive local density.
Now we show that if $\Pden(M,L)\not =0$, then $M\cong L$. If $\Pden(M,L)\not =0$, then for any large enough $d$ we have
\begin{align*}
    \mathrm{Pherm}_{L,M}(O_{F_0}/(\pi_0^d))\not =0.
\end{align*}
Let $\phi\in  \mathrm{Pherm}_{L,M}(O_{F_0}/(\pi_0^d))$ be a primitive embedding and $L_{(d)}=L\otimes_{O_{F_0}}O_{F_0}/(\pi_0^d)$. Let $\overline{\phi(L_{(d)})}$ be the image of  $\phi(L_{(d)})$ in $\overline{M_{(d)}}$. Since $\phi$ is primitive, we have $\overline{\phi(L_{(d)})}=\overline{M_{(d)}}$. Then by Nakayama's lemma, we know $\phi(L_{(d)})=M_{(d)}$.   Hence $\phi$ is an isometry between $L_{(d)}\cong M_{(d)}$. Since this holds for any large enough $d$, we have $L\cong M$.

Now the formula of $\den(M,L)$ follows from \eqref{eq: Pden(M,L)} and Lemma \ref{lem: ind formula to prim ld}.
	\end{proof}
	\begin{corollary}\label{cor: int of pden_{2t}}
	Let $L$ be a lattice. Then for any even integer $t$ such that $0\le t\le t_{\max}$, we have
	\begin{align*}
	\Den_{t}(L):=\frac{\Den(\Lambda_t^\sharp, L)}{\Den(\Lambda_t^\sharp,\Lambda_t^\sharp)}\in \mathbb{Z}.
	\end{align*}
	\end{corollary}
	
	\begin{corollary}\label{cor: vanishing of error term}
	Assume $L\not \cong \Lambda_{t}^{\sharp}$ for any vertex lattice $\Lambda_{t}$ with  $t >0$.  Then
	\begin{align*}
	    \ppden(L)=\Pden'(L).
	\end{align*}
	\end{corollary}
	
	\begin{corollary} \label{cor: coefficient}
	Let $c_{t} $ be the coefficients in \eqref{eq:coeff} with  even $t$ and  $0<t\le t_{\max}$.  Then
		\begin{align*}
		c_{t} =-\Pden'(\Lambda_t^{\sharp}).
		\end{align*}
	\end{corollary}
	\begin{proof}
		 On the one hand, combining Corollary \ref{cor: ind structure of partial Den(T)} with \eqref{eq:coeff}, we obtain
		\begin{align*}
			\ppden(\Lambda_t^{\sharp})=0.
		\end{align*}
On the other hand, by  Lemma \ref{lem: 0 vanish of error term} and \eqref{eq: def of pdenL},
\begin{align*}
    \ppden(\Lambda_{t}^{\sharp})=\Pden'(\Lambda_{t}^{\sharp})+c_{t} .
\end{align*}
	\end{proof}
Write $\Lambda_t^{\sharp}=H^{t/2}\obot L_1$ where $L_1$ is unimodular of rank $n_1$. Then by Lemma \ref{lem: L=H^i obot L_2}, Corollaries \ref{cor: Pden} and   \ref{cor: coefficient},  we have (see the following subsections for the relevant notations)
\begin{align*}
    c_{t}
    =-2\frac{\prod_{\ell =1}^{t/2-1}(1-q^{2\ell})}{\Den(I_n,I_n)}\cdot\sum_{i=0}^{n_1}   \prod_{\ell=0}^{n_1-i-1}(1-\q^{2(\ell+t/2)}) \cdot  \sum_{V_1\in \mathrm{Gr}(i,\overline{L_1})(\mathbb{F}_\q)}  |\mathrm{O}(V_1,\overline{I_n})|.
\end{align*}
Combining this formula with Lemma \ref{lem: Pden(I_m,L)} and Lemma \ref{lem: counting subspace}, we can compute $c_t$ explicitly. We give some examples here.
\begin{example}\label{ex: coeff}
If $n$ is odd,  we have
\begin{align}\label{eq: Den' Lambda tmax}
    c_{t_{\mathrm{max}}}=-\frac{\Pden'(I_n,\Lambda_{t_{\max}}^\sharp)}{\Den(I_n,I_n)}= \frac{(-1)^{\frac{n+1}{2}}}{q^{(\frac{n-1}2)^2}(q^{\frac{n-1}2}+1)}.
\end{align}
If $n$ is even and $\epsilon=1$, we have
\begin{align*}
     c_{t_{\mathrm{max}}}=-\frac{\Pden'(I_n,\Lambda_{t_{\max}}^\sharp)}{\Den(I_n,I_n)}= \frac{(-1)^{\frac{n}{2}}}{q^{\frac{n}2(\frac{n}2 -1)}(q^{\frac{n}2}+1)}.
\end{align*}

We also give a list of   $c_t$ for small $n$, $t$ and $\epsilon=1$ in the following table:\newline

\begin{table*}[h]
\begin{center}
\begin{tabular}{ c c c c c c  }
\toprule
\diagbox{$t$}{$n$} & \makecell{2}& \makecell{3} & \makecell{4} & \makecell{5} &\makecell{6}
  \\
\midrule
2 &  $\frac{-1}{q+1}$ & $\frac{1}{q(q+1)}$ & $\frac{-1}{q^2(q+1)}$& $\frac{1}{q^3(q+1)}$ & $\frac{-1}{q^4(q+1)}$   \\
4&  0 & 0 & $\frac{1}{q^2(q^2+1)}$ &  $\frac{-1}{q^4(q^2+1)}$ & $\frac{1}{q^6(q^2+1)}$   \\
6&  0 & 0 & 0 & 0 & $-\frac{1}{q^6(q^3+1)}$   \\
\bottomrule
\end{tabular}
\end{center}
\end{table*}
In fact, computation suggests that $c_t$ has the following simple formula:
\begin{align}
 c_t=\begin{cases}
 \frac{ (-1)^{n+t/2}(q^{n/2}+1)}{q^{t/2(n-t-1)}(q^{n/2}-1)(q^{t/2}+1)}& \text{ if $n$ is even and $\epsilon=-1$},\\
  \frac{ (-1)^{n+t/2}}{q^{t/2(n-t/2-1)}(q^{t/2}+1)} & \text{ otherwise}.
      \end{cases}
\end{align}
We believe this formula can be proved by similar method as in \S  \ref{sec:pPden}. Since this formula is not needed in our proof, we omit the details.
\end{example}

 \subsection{Explicit formulas for some simple primitive local density polynomials}
	\begin{lemma} (\cite[Lemma 2.15]{LL2}\label{lem: Pden(H^k,L)})
  Assume $L$ is an integral lattice of rank $n$. Then
  \begin{align*}
      \Pden(\cH^k,L)=\prod_{\ell=0}^{n-1}(1-\q^{-2k+2\ell}).
  \end{align*}
\end{lemma}
	\begin{lemma}\cite[Corollary 5.8]{HSY3}\label{lem: L=H^i obot L_2}
		Assume $L=\cH^j\obot L_1$ where $j>0$ and  $L_1$ is an integral lattice of rank $n_1$. Then 	
		\begin{align*}
			\den(I_m,L,X)&=\Big(\prod_{\ell=0}^{j-1}(1-\q^{2\ell} X)\Big)\den(I_m,L_1,\q^{2j}X),\\
			\Pden(I_m,L,X)&=\Big(\prod_{\ell=0}^{j-1}(1-\q^{2\ell} X)\Big)\Pden(I_m,L_1,\q^{2j}X).
		\end{align*}
		In particular,
			\begin{align}\label{eq: Den'. L=H^j obot L_1}
			\Pden'(I_m,L)=2\Big(\prod_{\ell=1}^{j-1}(1-\q^{2\ell} )\Big)\Pden(I_m,L_1,\q^{2j}).
		\end{align}
	\end{lemma}
	\begin{proof}
	First, by \cite[Corollary 5.8]{HSY3} and Lemma \ref{lem: Pden(H^k,L)},
	\begin{equation}\label{eq:Den L=H^j obot L_1}
	    \den(I_m,L,X)=\Big(\prod_{\ell=0}^{j-1}(1-\q^{2\ell} X)\Big)\den(I_m,L_1,\q^{2j}X).
	\end{equation}
Notice that if $L\subset L'$ and $L'$ is not of the form $H^j\obot L_1'$,  then there exists $v\in L'\setminus L$ such that $\mathrm{Pr}_{H_F^j}(v)\neq 0$ and $\mathrm{Pr}_{H_F^j}(v)\not \in H_j$. Hence some fundamental invariant of $L'$ is less than or equal to $-2$. Hence $\den(I_m,L',X)=0$ by Lemma \ref{lem: vanish val le -1}. Now  Theorem \ref{thm:ind formula reducing valuation} and \eqref{eq:Den L=H^j obot L_1} implies
\begin{align*}
	&	\Pden(I_m, L,X)
	\\
	 	&= \sum_{i=0}^{n} (-1)^{i} \q^{i(i-1)/2+i(n-m)}X^{i}
	 \sum_{\substack{L_1 \subset L_1' \subset \pi^{-1}L_1 \\ \ell(L_1'/L_1)=i}} \den(I_m, H^j\obot L_1',X)\\
	 &=\Big(\prod_{\ell=0}^{j-1}(1-\q^{2\ell} X)\Big) \sum_{i=0}^{n_1} (-1)^{i} \q^{i(i-1)/2+i(n-m-2j)}(q^{2j}X)^{i}
	 \sum_{\substack{L_1 \subset L_1' \subset \pi^{-1}L_1 \\ \ell(L'/L)=i}} \den(I_m,   L_1',q^{2j}X)\\
	 &=\Big(\prod_{\ell=0}^{j-1}(1-\q^{2\ell} X)\Big)\Pden(I_m,L_1,\q^{2j}X)
		\end{align*}
		as expected.
	\end{proof}

\begin{definition}
Assume $U$ and $V$ are quadratic spaces over $\mathbb{F}_q$. We define $\mathrm{O}(U,V)$ to be the set of isometries from $U$ into $V$, and $M(U,V)$ to be the set of subspaces $V_1\subset V$ such that $V_1\cong U$. Moreover, we define $m(U,V)=|M(U,V)|$.
\end{definition}

\begin{definition}
We define $U_i^\epsilon$ to be the $i$-dimensional non-degenerate quadratic space over $\mathbb{F}_q$ with $\chi(U_n^\epsilon)=\epsilon.$ Moreover, we define $0_i$ to be the $i$-dimensional totally isotropic space.
\end{definition}

\begin{lemma}\cite[Lemma A.11]{HSY3}\label{lem: Pden(I_m,L)}
  Assume $L=I_{n-t}^{\epsilon_1}\obot L_2$ where $L_2$ is a lattice of full type $t$ and $n\le m$. Then
  \begin{align*}
      \Pden(I_m^{\epsilon_2},L)=\q^{-mn+n^2}\cdot |\mathrm{O}(0_t\obot U_{n-t}^{\epsilon_1},U_m^{\epsilon_2})|.
  \end{align*}
   Specifically,  we have by \cite[Lemma 3.2.1]{LZ2},
\begin{align*}
|\mathrm{O}(0_j\obot U_k^{\epsilon_1}, U_m^\epsilon)|
&= q^{(k+j)(2m-k-j-1)/2} \prod_{\lfloor\frac{m-k}2\rfloor+1-j \le   l \le \lfloor \frac{m-1}2\rfloor} (1-q^{-2l})
\\
&\quad \cdot
 \begin{cases}
 (1+\epsilon \epsilon_1 q^{-\frac{m-k}2+j})   &\ff\, m \equiv  k \equiv 1 \pmod 2,
 \\
1  &\ff\, m \equiv  k -1 \equiv 1 \pmod 2,
 \\
  (1-\epsilon q^{-\frac{m}2})   &\ff\, m \equiv  k -1 \equiv 0 \pmod 2,
  \\
  (1-\epsilon q^{-\frac{m}2}) (1+\epsilon \epsilon_1 q^{-\frac{m-k}2+j})   &\ff\, m \equiv  k \equiv 0 \pmod 2.
 \end{cases}
\end{align*}
  \end{lemma}
\begin{corollary}\label{lem: alpha(M,M)}
Let $I_n$ be the unimodular lattice of rank $n$ and sign $-\epsilon$. Then
	\begin{align*}
		\den(I_{n},I_{n})=\begin{cases}
			2q^{\frac{n(n-1)}{2}}\prod _{s=1}^{\frac{n-1}{2}}(1-\q^{-2s})& \text{ if $n$ is odd},\\
			2q^{\frac{n(n-1)}{2}}(1+\epsilon \q^{-\frac{n}{2}})\prod _{s=1}^{\frac{n}{2}-1}(1-\q^{-2s})& \text{ if $n$ is even}.
		\end{cases}
	\end{align*}
\end{corollary}

\subsection{Counting formulas for subspaces of a quadratic space over $\mathbb{F}_q$}
The main results of \S \ref{sec: formula of ppden} transforms the calculation of primitive local density polynomial into a sum over the subspaces of a quadratic space over $\mathbb{F}_q$.
 In this subsection, we  count the number of such subspaces  with a fixed  quadratic form.
\begin{lemma}\label{lem: counting subspace}
 Given quadratic spaces $U$ and $V$ over $\mathbb{F}_q$, let $M(U,V)$ be the set of subspaces $V_1\subset V$ such that $V_1\cong U$, and  let $m(U,V)=|M(U,V)|$. Then
	  \begin{align*}
	      m(0_j\obot U_{k}^{\epsilon_2},U_{n}^{\epsilon})&= \frac{|\mathrm{O}(0_j\obot U_{k} ^{\epsilon_2},U_{n}^{\epsilon})|}{\q^{jk}|\mathrm{O}( U_{k} ^{\epsilon_2},U_{k} ^{\epsilon_2})|\cdot |\mathrm{GL}_j(\mathbb{F}_{\q})|}.
	  \end{align*}
In particular,
\begin{equation}\label{eq:m equal m times m}
     m(0_j\obot U_{k}^{\epsilon_2},U_{n}^{\epsilon})=q^{-jk}   m(0_j,U_{n-k}^{\delta(n,k,\epsilon,\epsilon_2)})m(U_{k}^{\epsilon_2},U_{n}^{\epsilon}),
\end{equation}
where
\begin{equation}\label{eq:delta}
     \delta(n,k,\epsilon,\epsilon_2)=
     \begin{cases}
        \epsilon & \text{ if $k=0$},\\
           -\epsilon\epsilon_2 & \text{ if both $k$ and $n-k$ are odd},\\
           \epsilon \epsilon_2 & \text{ otherwise}.
     \end{cases}
\end{equation}
\end{lemma}
\begin{proof}
We prove the first identity first.
Fix a subspace $V_1$ of $U_n^\epsilon$ such that $V_1\cong 0_j\obot U_k^{\epsilon_2}$. Then by Witt's theorem we have a surjection
\begin{align*}
    \mathrm{O}(0_j\obot U_{k} ^{\epsilon_2},U_{n}^{\epsilon}) \to M(0_j\obot U_{k} ^{\epsilon_2},U_{n}^{\epsilon}),\quad  \phi \to \phi(V_1).
\end{align*}
Moreover, each fiber of this surjection is in bijection with $\mathrm{O}(0_j\obot U_k^{\epsilon_2})$. Any $\phi \in \mathrm{O}(0_j\obot U_k^{\epsilon_2})$ is determined uniquely by  $\phi_1=\phi|_{0_j}$ and $\phi_2=\phi|_{U_k^{\epsilon_2}}$. The number of different choices of $\phi_1$ is $|\mathrm{GL}_j(\mathbb{F}_q)|$.  The number of different choices of $\phi_2$ is $\q^{jk}|\mathrm{O}( U_{k} ^{\epsilon_2},U_{k} ^{\epsilon_2})|$.
\end{proof}	
\begin{lemma}\label{lem: m with type}
For any $\epsilon_1, \epsilon_2 \in \{\pm  1\}$, we have
\begin{align*}
 	    m(0_j\obot U_{k}^{\epsilon_2},0_t\obot U_{n-t}^{\epsilon_1})&=\sum_{\ell=0}^{\mathrm{min}\{t,j\}} \binom{t}{\ell}_\q\cdot q^{(t-\ell)(j+k-\ell)} m(0_{j-\ell}\obot U_{k}^{\epsilon_2},U_{n-t}^{\epsilon_1}).
	    \end{align*}
\end{lemma}	
\begin{proof}
Let $V$ and $U$ be quadratic spaces over $\mathbb{F}_q$ such that $V\cong 0_t\obot U_{n-t}^{\epsilon}$, and $U\cong 0_j\obot U_k^{\epsilon_2}.$ Let $R\cong 0_t$ be the radical of $V$.
First, we consider a partition of
$$ M(U,V)= \bigsqcup_{\ell=0}^{\mathrm{min}\{t,j\}}M_\ell(U,V)$$
such that
\begin{align*}
    V_1\in M_\ell(U,V) \text{ if and only if } \mathrm{dim}_{\mathbb{F}_q}(V_1\cap R)=\ell.
\end{align*}
The number of choices of $\ell$-dimensional subspace of $R$ is $\binom{t}{\ell}_q$. Now we fix an $\ell$-dimensional subspace $W$ of $R$. Let $\overline{R}\cong 0_{t-\ell}$ be the radical of the quotient space of $V/W$. Then a choice of $V_1\in M_\ell(U,V)$ such that $V_1\cap R=W$ corresponds to an element of
\begin{align*}
 S=\{ \overline{V_1}\subset V/W   \mid   \overline{V_1}\cap \overline{R}=\{0\}\text{ and }\overline{V_1}\cong 0_{j-\ell}\obot U_{k}^{\epsilon_2}\}.
\end{align*}
Write $V/W=\overline{R}\obot \overline{V_2}\cong 0_{t-\ell}\obot U_{n-t}^{\epsilon_1}$. Let $\mathrm{Pr}:V/W\rightarrow \bar{V}_2$ be the natural quotient map.
For $\overline{V_1}\in S$, the condition $\overline{V_1}\cap \overline{R}=\{0\}$ implies that $\mathrm{Pr}(\overline{V_1})\cong 0_{j-\ell}\obot U_{k}^{\epsilon_2}$ by the rank-nullity theorem. Moreover, the following map
\begin{align*}
    S\stackrel{}{\to} M(0_{j-\ell}\obot U_{k}^{\epsilon_2},\overline{V_2}),\quad \overline{V_1}\mapsto \mathrm{Pr}(\overline{V_1})
\end{align*}
is a surjection with each fiber in a bijection with $\overline{R}^{j+k-\ell}$.
\end{proof}
\begin{corollary}
\begin{align*}
       | \mathrm{O}(0_j\obot U_{k}^{\epsilon_2},0_t\obot U_{n-t}^\epsilon)|&=\q^{jk}|\mathrm{O}(U_{k} ^{\epsilon_2},U_{k} ^{\epsilon_2})|\cdot |\mathrm{GL}_j(\mathbb{F}_{\q})|  \cdot m(0_j\obot U_{k}^{\epsilon_2},0_t\obot U_{n-t}^\epsilon)\\
	    &=|\mathrm{GL}_j(\mathbb{F}_{\q})|\cdot \sum_{\ell=0}^{\mathrm{min}\{t,j\}} \binom{t}{\ell}_\q\cdot q^{(t-\ell)(j+k-\ell)+\ell k} \cdot \frac{|\mathrm{O}(0_{j-\ell}\obot U_{k}^{\epsilon_2},U_{n-t}^\epsilon)|}{|\mathrm{GL}_{j-\ell}(\mathbb{F}_\q)|}.
\end{align*}
\end{corollary}

We will also need the following lemma later (Section 7), which follows from Lemmas \ref{lem: counting subspace} and \ref{lem: Pden(I_m,L)} .

\begin{lemma}\label{lem:quotientm} For integers $0\le r \le n$ and $\epsilon_1, \epsilon =\pm 1$, we have
$$
\frac{m(U_r^{-\epsilon_1},U_n^\epsilon)}{m(U_r^{\epsilon_1}, U_n^\epsilon)}
=\begin{cases}
 1  &\ff \, r\equiv n-1 \equiv 1 \pmod 2,
 \\
 \frac{1-\epsilon \epsilon_1 q^{-\frac{n-r}2}}{1+\epsilon \epsilon_1 q^{-\frac{n-r}2}}  &\ff \, r\equiv n\equiv 1 \pmod 2,
 \\
 \frac{1-\epsilon_1 q^{-\frac{r}2}}{1+\epsilon_1 q^{-\frac{r}2}} &\ff \, r\equiv n-1\equiv 0 \pmod 2,
\\
 \frac{1-\epsilon \epsilon_1 q^{-\frac{n-r}2}}{1+\epsilon \epsilon_1 q^{-\frac{n-r}2}}\cdot \frac{1-\epsilon_1 q^{-\frac{r}2}}{1+\epsilon_1 q^{-\frac{r}2}}
 &\ff \, r\equiv n\equiv 0 \pmod 2,
 \end{cases}
$$
and
$$
\frac{m(U_{r+1}^{\epsilon_1},U_n^\epsilon)}{m(U_r^{\epsilon_1}, U_n^\epsilon)}
=\frac{q^{n-2r-1}(1-(-1)^{n-r} \epsilon \epsilon_1 q^{-\lfloor\frac{n-r}2\rfloor})}{1-(-1)^{r+1} \epsilon_1 q^{-\lfloor \frac{r+1}2\rfloor}}.
$$
\end{lemma}

\begin{lemma}\label{lem:number of orthogonal flag}
Assume $i\leq r\leq n$ and $\epsilon,\sigma,\delta'\in \{\pm 1\}$. Let  $\delta(r,i,\delta',\sigma)$ or $\delta(n,i,\epsilon,\sigma)$ be as in \eqref{eq:delta}. Then
\[  m(U_i^{\sigma}, U_n^\epsilon) m(U_{r-i}^{\delta(r,i,\delta',\sigma)},U_{n-i}^{\delta(n,i,\epsilon,\sigma)})=m(U_r^{\delta'}, U_n^{\epsilon}) m(U_{i}^{\sigma}, U_r^{\delta'}). \]
\end{lemma}
\begin{proof}
Let $V=U_n^{\epsilon}$ and $S$ be the following set of flags in $V$,
\[S=\{0\subset F_1 \subset F_2 \subset V\mid F_1\cong U_{i}^{\sigma}, F_2\cong U_r^{\delta'}\}.\]
We can count the cardinality of $S$ in two ways. One way is to first count the number of $F_1\in M(U_{i}^{\sigma},V)$, then for a fixed $F_1$ count the number of $F_2'\in M(U_{r-i}^{\delta(r,i,\delta',\sigma)},(F_1)^\bot)$ which has a one-to-one correspondence with $F_2=F_2'\obot F_1\in M(U_r^{\delta'},V)$. This way we get
\[
\#|S|=m(U_i^{\sigma}, U_n^\epsilon) m(U_{r-i}^{\delta(r,i,\delta',\sigma)},U_{n-i}^{\delta(n,i,\epsilon,\sigma)}).
\]
On the other hand, we can first count the number of $F_2\in M(U_r^{\delta'},V)$, then count the number of  $F_1\in M(U_i^{\sigma},F_2)$ and get
\[\#|S|=m(U_r^{\delta'}, U_n^{\epsilon}) m(U_{i}^{\sigma}, U_r^{\delta'}).\]
This finishes the proof of the lemma.
\end{proof}

\subsection{$q$-binomial theorem}
In this subsection we discuss the $q$-binomial theorem and related results, which are used repeatedly in \S \ref{sec:pPden} to obtain certain vanishing results and transform complicated linear combinations into simple formulas.
\begin{definition}
The $q$-analogue of $\binom{n}{i}$ is defined to be	$$\binom{n}{i}_\q\coloneqq\frac{(\q^{n}-1)\cdots (\q^{n-i+1}-1)}{(\q^{i}-1)\cdots(\q-1)}.$$
\end{definition}
In fact, $\binom{n}{i}_\q$ is the number of $i$-dimensional subspaces of a $n$-dimensional vector space over $\mathbb{F}_\q$.
Now we recall the $\q$-binomial theorem.

 \begin{lemma}[$q$-binomial theorem]\label{lem: qbinomial}
The following identity between polynomials of $X$ holds:
		\begin{align}\label{eq: qbinomial}
			\prod_{i=0}^{n-1}(1-\q^i X)=\sum_{i=0}^{n} (-1)^i\q^{\frac{i(i-1)}{2}}\binom{n}{i}_\q X^i.
		\end{align}
 \end{lemma}
 \begin{corollary}\label{cor: qbinomial vanish}
 Let $f$ be a polynomial of degree $\le n-1$. Then
 \begin{align*}
     \sum_{i=0}^{n}(-1)^i\q^{\frac{i(i-1)}{2}}\binom{n}{i}_\q \cdot f(q^{-i})=0.
 \end{align*}
 \end{corollary}
 \begin{proof}
 Let $f=a_{n-1}x^{n-1}+\cdots+a_0$.  For $0\le s\le n-1$, by evaluating \eqref{eq: qbinomial} at $X=\q^{-s}$, we have
 \begin{align*}
      \sum_{i=0}^{n}(-1)^i\q^{\frac{i(i-1)}{2}}\binom{n}{i}_\q \cdot a_{s} q^{-si}=0.
 \end{align*}
 Hence
 \begin{align*}
       \sum_{i=0}^{n}(-1)^i\q^{\frac{i(i-1)}{2}}\binom{n}{i}_\q \cdot f(q^{-i})=\sum_{s=0}^{n-1} \sum_{i=0}^{n}(-1)^i\q^{\frac{i(i-1)}{2}}\binom{n}{i}_\q \cdot a_{s} q^{-si}=0.
 \end{align*}
 \end{proof}
The following is in some sense an inverse of $q$-binomial theorem that will be used in \S \ref{sec:pPden}.
\begin{lemma}\label{lem: inverse q bin}
  \begin{equation*}
    \sum_{i=0}^{n}(-1)^iq^{\frac{i(i-1)}{2}}\cdot\binom{n}{i}_q\cdot \prod_{\ell=0}^{i-1}(1+q^{-\ell}X)=(-X)^n.
\end{equation*}
\end{lemma}
\begin{proof}
Let $g_i(X)=\prod_{\ell=0}^{i-1}(q^{-\ell}X+1)$. Since $\{g_i(X)\}$ forms a basis of the vector space of polynomials of degree $\le n$,  there exist $a_{n,i} \in \mathbb{C}$ such that
\begin{align*}
  (-X)^n=\sum_{i=0}^{n}a_{n,i}\cdot g_i(X).
\end{align*}
Notice that $g_{i+1}(X)=(1+q^{-i}X)g_i(X)$, hence $Xg_i(X)=q^i(g_{i+1}(X)-g_i(X))$. Then we have
\begin{align*}
   \sum_{i=0}^{n+1}a_{n+1,i}\cdot g_i(X)= (-X)^{n+1}=(-X)\cdot (-X)^n&=\sum_{i=0}^{n}(-a_{n,i})\cdot Xg_i(X)\\
   &=\sum_{i=0}^{n}(-a_{n,i})q^i\cdot (g_{i+1}(X)-g_i(X)).
\end{align*}
As a result, we have
\begin{align}\label{eq: recursion1}
    a_{n+1,i}=\begin{cases}
          a_{n,0} & \text{ if $i=0$},\\
          q^ia_{n,i}-q^{i-1}a_{n,i-1} & \text{ if $0<i< n+1$},\\
          -a_{n,n}q^{n} & \text{ if $i=n+1$}.
    \end{cases}
\end{align}
It is easy to check that $b_{n, i}=(-1)^i q^{\frac{i(i-1)}{2}}\cdot\binom{n}{i}_q $ satisfies \eqref{eq: recursion1} and that $a_{1, i}= b_{1, i}$. So we have $a_{n, i}=b_{n, i}$.
\end{proof}

\section{Decomposition of primitive local densities}\label{sec: formula of ppden}
This section is devoted to prove the following decomposition of the primitive local density polynomial, which is a vast generalization of \cite[Proposition A.14]{HSY3} and one of the main tools we use to prove Theorem \ref{thm: formula of ppden}.
	\begin{theorem}\label{thm: decom of Pden}
	Assume that  $L$ is an integral lattice of rank $n$. For any $m\ge 0$  we have
		\begin{align*}
			\Pden(I_{m},L,X)=\sum_{i=0}^{n}\Pden^{n-i}(I_{m},L,X),
		\end{align*}
		where $\Pden^{n-i}(I_{m},L,X)$ is a polynomial characterized by
		\begin{align}\label{eq: beta_i, L positive}
			\Pden^{n-i}(I_{m},L,\q^{-2k})&=\q^{-2ki}\Pden(\cH^k,0_{n-i})\sum_{V_1\in \mathrm{Gr}(i,\overline{L})(\mathbb{F}_\q)}\q^{(n-i)i}   \Pden(I_{m},L_{V_1}),
		\end{align}
		where $0_{n-i}$ is a totally isotropic lattice of rank $n-i$ and $L_{V_1}\subset L$ is a sublattice of rank $i$ such that $\overline{L}_{V_1}=V_1$.
\end{theorem}
Here, an important special case is when $m=n$. In this case, $\Pden^{0}(I_n^{-\epsilon},L,X)=0$ since $\chi(L)\neq \chi(I_n^{-\epsilon})$. Hence,
\begin{align*}
    	\Pden(I_{n},L,X)=\sum_{i=0}^{n-1}\Pden^{n-i}(I_{n},L,X).
\end{align*}

Applying the formulas of $\Pden(\cH^k,0_{n-i})$ and $\Pden(I_{m},L_{V_1})$ given in Lemmas \ref{lem: Pden(H^k,L)} and \ref{lem: Pden(I_m,L)}, we obtain the following corollary.
\begin{corollary}\label{cor: Pden}
Let $L$ be  an integral lattice of rank $n$. We have
		\begin{align}
		    \Pden(I_{m},L,X)=\sum_{i=0}^{n} \Big((\q^{n-m}X)^{i}\prod_{\ell=0}^{n-i-1}(1-\q^{2\ell}X) \Big)\cdot  \sum_{V_1\in \mathrm{Gr}(i,\overline{L})(\mathbb{F}_\q)}  |\mathrm{O}(V_1,\overline{I_m})|.
		\end{align}
		In particular,
\begin{align*}
     \Pden'(I_{n},L)=2\sum_{i=0}^{n} &\Big( \prod_{\ell=1}^{n-i-1}(1-\q^{2\ell}) \Big)\cdot   \sum_{V_1\in \mathrm{Gr}(i,\overline{L})(\mathbb{F}_\q)}  |\mathrm{O}(V_1,\overline{I_n})|.
\end{align*}
\end{corollary}

When $L$ is a full type lattice of rank $n$,  $\overline{L}$ is totally isotropic. So we obtain
\begin{corollary}\label{cor: L full type Pden(Im,L)}
Let $L$ be a full type lattice of rank $n$. Then
\begin{align*}
     \Pden(I_{m},L,X)=\sum_{i=0}^{n} &\Big((\q^{n-m}X)^{i}\prod_{\ell=0}^{n-i-1}(1-\q^{2\ell}X) \Big)\cdot  \binom{n}{i}_\q |\mathrm{O}(0_i,\overline{I_m})|.
\end{align*}
In particular,
\begin{align*}
     \Pden'(I_{n},L)=2\sum_{i=0}^{n} &\Big( \prod_{\ell=1}^{n-i-1}(1-\q^{2\ell}) \Big)\cdot  \binom{n}{i}_\q |\mathrm{O}(0_i,\overline{I_n})|.
\end{align*}
Here by Lemma \ref{lem: Pden(I_m,L)}, we have
\begin{align*}
|\mathrm{O}(0_i,\overline{I_m})|
&=\q^{\frac{i(i-1)}{2}}\cdot \begin{cases}
\prod_{1 \leq \ell\le i}\left(\q^{m+1-2\ell}-1\right) &  \text {if $m$ is odd},\\
\left(\q^{m/2}-\chi(I_m)  \right)\left(q^{m/2-i}+\chi(I_m)  \right) \cdot \prod_{1 \leq \ell<i}\left(q^{m-2\ell}-1\right) &  \text {if $m$ is even}.
\end{cases}
\end{align*}
\end{corollary}

\begin{proof}[Proof of Theorem \ref{thm: decom of Pden}]
		To save notation, we use  $M$ to denote $I_{m}$ in this proof.
Recall that by \eqref{eq:prim local density M,L},
	\begin{equation*}
	\Pden(M, L, \q^{-2k})=\lim_{d\to \infty} \q^{-d n (2(m+2k) -n)} |\mathrm{Pherm}_{L,M\obot H^k}(O_{F_0}/(\pi_0^d))|.
	\end{equation*}
First, we define a partition
$$\mathrm{Pherm}_{L,M\obot H^k}(O_{F_0}/(\pi_0^d))=\bigsqcup_{0\le i\le n}\mathrm{Pherm}_{L,M\obot H^k}^i(O_{F_0}/(\pi_0^d)),$$
where
\begin{align}
	\mathrm{Pherm}_{L,M\obot H^k}^i(O_{F_0}/(\pi_0^d))&\coloneqq \{\phi\in \mathrm{Pherm}_{L,M\obot H^k}(O_{F_0}/(\pi_0^d))\mid  \mathrm{dim}_{\mathbb{F}_\q}\overline{\mathrm{Pr}_{\cH^k}(\phi(L))}=i\}.
	\end{align}
Here $\mathrm{Pr}_{H^k}$ denote the projection map to $H^k$, and $\overline{\mathrm{Pr}_{\cH^k}(\phi(L))}$ denote the image of $ \mathrm{Pr}_{\cH^k}(\phi(L))$ in $\overline{H}^k$.
As a result, we have $$\Pden(M,L,X)=\sum_{i=0}^{n}	\Pden^i(M,L,X),$$
where $\Pden^i(M,L,X)$ is the function such that
\begin{align*}
		\Pden^i(M,L,\q^{-2k})&\coloneqq\lim_{d\to \infty}\q^{-(2n(2k+m)-n^2)d}|	\mathrm{Pherm}_{L,M\obot H^k}^i(O_{F_0}/(\pi_0^d))|.
	\end{align*}
We need to count $|	\mathrm{Pherm}_{L,M\obot H^k}^i(O_{F_0}/(\pi_0^d))|$. For $\phi \in \mathrm{Pherm}_{L,M\obot H^k}^i(O_{F_0}/(\pi_0^d)),$ it induces $$\overline{\phi}: V= \overline{L} \longrightarrow M\obot \cH^k/\pi(M\obot \cH^k),\text{ and } \overline{\phi}_{\cH^k}\coloneqq \mathrm{Pr}_{\overline{\cH}^k}\circ \overline{\phi}.$$
	By the definition of  $\mathrm{Pherm}_{L,M\obot H^k}^i(O_{F_0}/(\pi_0^d))$,
	For a $(n-i)$-dimensional subspace $V_1\subset \overline{L}$, let
	\begin{align*}
	    \mathrm{Pherm}_{L,M\obot H^k}^{V_1}(O_{F_0}/(\pi_0^d))=\{ \phi \in \mathrm{Pherm}_{L,M\obot H^k}^i(O_{F_0}/(\pi_0^d))\mid \mathrm{Ker}(\overline{\phi}_{\cH^k})=V_1 \subset \overline{L} \}.
	\end{align*}
	Since $\mathrm{Ker}(\overline{\phi}_{\cH^k})\subset \overline{L}$ has dimension $n-i$ for any $\phi\in \mathrm{Pherm}_{L,M\obot H^k}^i(O_{F_0}/(\pi_0^d))$, we have
	\begin{align}\label{eq: partition of Pherm^i}
	    	\mathrm{Pherm}_{L,M\obot H^k}^i(O_{F_0}/(\pi_0^d))=\bigsqcup_{V_1\in \mathrm{Gr}(n-i,V)(\mathbb{F}_\q)}\mathrm{Pherm}_{L,M\obot H^k}^{V_1}(O_{F_0}/(\pi_0^d)).
	\end{align}

	We need to show
	\begin{align}\label{eq: Pherm V_1}
	  	\q^{-(2(m+2k)n-n^2)d}    |\mathrm{Pherm}_{L,M\obot H^k}^{V_1}(O_{F_0}/(\pi_0^d))|=\q^{(n-i)i}X^{n-i}\Pden(\cH^k,L_{V_2})\cdot \Pden(M,L_{V_1}).
	\end{align}
	Let $V_2$ be a subspace of $V$ such that $V=V_1\oplus V_2$.  Let $L_{V_1}\subset L$ be a sublattice of rank $n-i$ such that the image of $L_{V_1}$ in $V$ is $V_1$. Similarly, let $L_{V_2}\subset L$ be a sublattice of rank $i$ such that the image of $L_{V_2}$ in $V$ is $V_2$. Let  $\phi_i=\phi|_{L_{V_i}}$ for $i\in \{1,2\}$.  According to Lemma \ref{lem: cancel beta_n}, the number of different choices of $\phi_2$ is given by
	\begin{align}\label{eq: L_{V_2}}
	    |	\mathrm{Pherm}_{L_{V_2},M\obot H^k}^i(O_{F_0}/(\pi_0^d))|=\q^{i(2(m+2k)-i)d} \Pden(\cH^k,L_{V_2}).
	\end{align}
Now for a fixed $\phi_2\in 	\mathrm{Pherm}_{L_{V_2},M\obot H^k}^i(O_{F_0}/(\pi_0^d))$, let
\begin{align*}
    	&\mathrm{Pherm}_{L_{V_1},M\obot H^k}^{\phi_2}(O_{F_0}/(\pi_0^d))\\
    	&\coloneqq\{\phi_1\in \mathrm{Pherm}_{L_{V_1},M\obot H^k}(O_{F_0}/(\pi_0^d))\mid (\phi_1,\phi_2)\in \mathrm{Pherm}_{L,M\obot H^k}^{V_1}(O_{F_0}/(\pi_0^d))\}.
\end{align*}
Claim: For any $\phi_2\in 	\mathrm{Pherm}_{L_{V_2},M\obot H^k}^i(O_{F_0}/(\pi_0^d))$,
\begin{align}\label{eq: L_{V_1}}
    	|\mathrm{Pherm}_{L_{V_1},M\obot H^k}^{\phi_2}(O_{F_0}/(\pi_0^d))|&=\q^{(2d-1)(2k-i)(n-i)}	|\mathrm{Pherm}_{L_{V_1},M}(O_{F_0}/(\pi_0^d))|\\ \notag
    	&=\q^{(2d-1)(2k-i)(n-i)+(2m(n-i)-(n-i)^2)d}\Pden(M,L_{V_1}).
\end{align}
Assuming the claim holds,   \eqref{eq: Pherm V_1} follows from \eqref{eq: L_{V_2}} and \eqref{eq: L_{V_1}} since for any fixed $\phi_2$ we have
\begin{align*}
    |\mathrm{Pherm}_{L,M\obot H^k}^{V_1}(O_{F_0}/(\pi_0^d))|= |	\mathrm{Pherm}_{L_{V_2},M\obot H^k}^i(O_{F_0}/(\pi_0^d))|\cdot |\mathrm{Pherm}_{L_{V_1},M\obot H^k}^{\phi_2}(O_{F_0}/(\pi_0^d))|.
\end{align*}
Proof of the claim:
For $\phi_1 \in \mathrm{Pherm}_{L_{V_1},M\obot H^k}^{\phi_2}(O_{F_0}/(\pi_0^d)),$ write $\phi_1=\phi_{1,\cH^k}+\phi_{1,M}$, where $\phi_{1,\cH^k}=\mathrm{Pr}_{\cH^k}\circ \phi_1$ and $\phi_{1,M}=\mathrm{Pr}_{M}\circ \phi_1$. First,
for any $g\in \mathrm{U}(M\obot H^k)$, one can directly check the map
\begin{align}\label{eq: map}
    	\mathrm{Pherm}_{L_{V_1},M\obot H^k}^{\phi_2}(O_{F_0}/(\pi_0^d))&\to	\mathrm{Pherm}_{L_{V_1},g(M\obot H^k)}^{g\circ \phi_2}(O_{F_0}/(\pi_0^d)),\\ \notag
    	\phi_1&\mapsto g\circ \phi_1
\end{align}
is well-defined and is in fact a bijection.
Then according to Lemma  \ref{lem: translation to H^k}, we may assume $\phi_2(L_{V_2})\subset \cH^k$.

Now finding $\phi_1$ such that $(\phi_1,\phi_2) \in \mathrm{Pherm}_{L,M\obot H^k}^{V_1}(O_{F_0}/(\pi_0^d))$ is equivalent to finding $\phi_1$ such that  $\phi_{1,M}$ is primitive, $\phi_{1,H^k}\in \pi H^k$,
\begin{align}\label{eq: phi_{1,H}}
       (\phi_1(v),\phi_2(w))\equiv (v,w) \mod (\pi^{2d-1})\text{ for any $v\in L_{V_1},w\in L_{V_2}$},
\end{align}
and
\begin{align}\label{eq: temp 3}
    (\phi_1(v),\phi_1(w))\equiv (v,w)  \mod (\pi^{2d-1}) \text{ for any $v\in L_{V_1},w\in L_{V_1}$}.
\end{align}

We consider  condition \eqref{eq: phi_{1,H}} first. Since $\phi_2(L_{V_2})\subset \cH^k$, we have
\begin{align*}
     (\phi_{1,\cH^k}(v),\phi_2(w))\equiv (\phi_1(v),\phi_2(w)) \text{ for any $v\in L_{V_1},w\in L_{V_2}$}.
\end{align*}
When $k$ is large enough, we can always find and fix a $\phi_{1,\cH^k}'$ that satisfies \eqref{eq: phi_{1,H}}. Then finding $\phi_{1,\cH^k}$ that satisfies \eqref{eq: phi_{1,H}} is equivalent to find  $\Phi_{1,\cH^k}=\phi_{1,\cH^k}-\phi_{1,\cH^k}'$, which satisfies
\begin{align*}
      (\Phi_{1,\cH^k}(v),\phi_2(w))\equiv 0 \mod (\pi^{2d-1})\text{ for any $v\in L_{V_1},w\in L_{V_2}$},
\end{align*}
Then according to Lemma \ref{lem: Nperp}, the number of different choices for $\phi_{1,\cH^k}$ is $   \q^{(2d-1)(2k-i)(n-i)}.$

Now we consider condition  \eqref{eq: temp 3}. Since $\phi_{1,\cH^k}(v)\in \pi \cH^k$ for any $v\in L_{V_1}$,
\eqref{eq: temp 3} is equivalent to
\begin{align*}
    (\phi_{1,M}(v),\phi_{1,M}(w))+\pi \alpha \equiv (\phi_1(v),\phi_1(w))\equiv (v,w)  \mod (\pi^{2d-1})   \text{ for any $v, w\in L_{V_1}$} ,
\end{align*}
for some $\alpha \in \Oo_{F}$. By Lemma \ref{lem: Pden(I_m,L)}, for a unimodular lattice $M$ and any integral lattice $N$, $\Pden(M,N)$ only depends on $\overline{M}$ and $\overline{N}$.
Hence, for our purpose, we may replace \eqref{eq: temp 3} by
\begin{align*}
    (\phi_{1,M}(v),\phi_{1,M}(w))  \equiv (\phi_1(v),\phi_1(w))\equiv (v,w)  \mod (\pi^{2d-1}) \text{ for any $v, w\in L_{V_1}$}.
\end{align*}
Therefore the number of different choices of primitive $\phi_{1,M}$  is given by
\begin{align*}
    \q^{(2m(n-i)-(n-i)^2)d}\Pden(M,L_{V_1}).
\end{align*}
As a result, we have
\begin{align*}
    |\mathrm{Pherm}_{L_{V_1},M\obot H^k}^{\phi_2}(O_{F_0}/(\pi_0^d))|  =\q^{(2d-1)(2k-i)(n-i)}\cdot q^{(2m(n-i)-(n-i)^2)d}\Pden(M,L_{V_1}).
\end{align*}
This finishes the proof of the claim.
\end{proof}

	\begin{lemma}\label{lem: cancel beta_n} Assume that  $L$ is an integral lattice of rank $n$ and $k\ge n$.
	Then
	\begin{align*}
	     |\mathrm{Pherm}_{L,I_m\obot H^k}^n(O_{F_0}/(\pi_0^d))|=\q^{2mnd}|\mathrm{Pherm}_{L,H^k}(O_{F_0}/(\pi_0^d))|.
	\end{align*}
	\end{lemma}
\begin{proof}
For $\phi \in \mathrm{Pherm}_{L,I_m\obot H^k}^n(O_{F_0}/(\pi_0^d))$, we may identify $\phi$ with $(\phi_{\cH^k},\phi_{I_m})$ where $\phi_{H^k}=\mathrm{Pr}_{H^k}\circ \phi$ and $\phi_{I_m}=\mathrm{Pr}_{I_m}\circ \phi$. As
$$|\Hom_{O_F}(L,I_m)(O_{F_0}/(\pi_0^d))|=q^{2mnd},$$
it suffices to show that for any fixed $\varphi \in \Hom_{O_F}(L,I_m)$, we have
\begin{align}\label{eq: temp}
    |\{\phi \in \mathrm{Pherm}_{L,I_m\obot H^k}^n(O_{F_0}/(\pi_0^d))\mid\phi_{I_m}=\varphi\}|=|\mathrm{Pherm}_{L, H^k}(O_{F_0}/(\pi_0^d))|.
\end{align}
Let $L_{\varphi}$ be the lattice $L$ equipped with the hermitian form $(x,y)_{L_{\varphi}}\coloneqq(\phi_{H^k}(x),\phi_{H^k}(y))$ where $\phi$ is any element in $\mathrm{Pherm}_{L,I_m\obot H^k}^n(O_{F_0}/(\pi_0^d))$ such that $\phi_{I_m}=\varphi$. Since each such $\phi$ is an isometry and $\phi_{I_m}$ is fixed, $(\, ,\,)_{L_\varphi}$ is independent of the choice of $\phi$.
Then we have a bijection
\begin{align}\label{eq: temp2}
    \{\phi \in \mathrm{Pherm}_{L,I_m\obot H^k}^n(O_{F_0}/(\pi_0^d))\mid\phi_{I_m}=\varphi\}&\to
     \mathrm{Pherm}_{L_{\varphi},\cH^k}(O_{F_0}/(\pi_0^d)),\\ \notag
     \phi&\mapsto \phi_{H^k}.
\end{align}
Since $L$ is integral and $\phi$ is an isometric embedding,  $L_{\phi_{\cH^k}}$ is also integral. Then according to \cite[Lemma 2.16]{LL2},
\begin{align*}
    |\mathrm{Pherm}_{L_{\varphi},\cH^k}(O_{F_0}/(\pi_0^d))|
   =|\mathrm{Pherm}_{L, H^k}(O_{F_0}/(\pi_0^d))|.
\end{align*}
Combining with the bijection in \eqref{eq: temp2}, this proves \eqref{eq: temp}, hence  finishes the proof of the lemma.
\end{proof}

\begin{lemma}\label{lem: translation to H^k}
	Assume that $M$ is a unimodular lattice, $L$ is an integral lattice of rank $n$, and $\phi: L\rightarrow M\obot \cH^k$ is a primitive isometric embedding such that $\mathrm{Pr}_{\cH^k}(\phi(L))$ is primitive in $\cH^k$. Then there exists a $g\in \mathrm{U}(M\obot H^k)$ such that  $g(\phi(L)) \subset \cH^k$.
\end{lemma}
\begin{proof}
  Consider the non-degenerate symplectic space over $\mathbb{F}_\q$: $(\overline{\cH}^k,\langle\, , \, \rangle)=(\overline{\cH}^k,\overline{\pi (\,,\,)})$. Let $\overline{v}$ denote the image of $v\in \cH^k$ in $\overline{\cH}^k$. Since $\mathrm{Pr}_{\cH^k}(\phi(L))$ is primitive in $\cH^k$,  $V(L)\coloneqq\overline{\mathrm{Pr}_{\cH^k}(\phi(L))}$ is a $n$-dimensional subspace. Since $M$ and $L$ are integral, $\mathrm{Pr}_{\cH^k}(\phi(L))$ is integral. Hence $V(L)$ is an isotropic subspace. Let $\{\ell_1,\cdots,\ell_n\}$ be a basis of $L$, $\ell_{s,\cH^k}=\mathrm{Pr}_{\cH^k}(\ell_s)$ and $e_s=\overline{\ell}_{s,H^k}$. Since $V(L)$ is an $n$-dimensional isotropic space, we have $k\ge n$ and we can extend $\{e_1,\cdots,e_n\}$ to a standard symplectic basis $\{e_1,f_1,\cdots,e_k,f_k\}$ of $\overline{H}^k$, where $(e_s,f_t)=\delta_{st}$, and $(e_s,e_t)=(f_s,f_t)=0$ for $1\le s,t\le k$.

  Now let $\{\tilde{f}_1,\cdots, \tilde{f}_n\}$ be a lifting of $\{ f_1,\cdots,f_n\}$. In particular, for $1\le s\le n$, we have $(\ell_s,\tilde{f}_s)=\pi^{-1}+x$ for some $x\in O_F$. Therefore, $L\oplus \langle \tilde{f}_{1}\cdots,\tilde{f}_{n}\rangle \cong \cH^n$. Hence, there exists $g\in \mathrm{U}(M\obot \cH^k)$ such that $g(L\oplus \langle \tilde{f}_{1}\cdots,\tilde{f}_{n}\rangle)\subset \cH^k$.
\end{proof}

\begin{lemma}\label{lem: Nperp}
Let $N\subset \cH^k$ be a primitive integral lattice of rank $i$.  Then
	\begin{align*}
\#\{  w \in \pi \cH^k/\pi_0^d  \cH^k\mid(N,w)=0\mod (\pi^{2d-1})
		\}=\q^{(2d-1)(2k-i)}.
	\end{align*}
\end{lemma}
\begin{proof}
Through this proof, we use $\overline{L}$ to denote the image of $L$ in $\overline{H}^k$  for any sublattice $L$ of $H^k$.
Let $N^{\perp}$ be the perpendicular lattice of $N$ in $H^k$. First we show $N^{\perp}$ is primitive of rank $2k-i$.  If $N^{\perp}$ is  not primitive, then there exists $v\in N^{\perp}$ such that $\pi^{-1}v\in H^k\setminus N^{\perp}.$  However, $(\pi^{-1}v, N)=0$, hence $\pi^{-1}v\in N^{\perp}$, which is a contradiction.

We claim that for any $w\in \pi H^k$ and $\pi (N,w)=0 \mod \pi^{a}$ with $a\ge 0$, there exists a $x\in \pi^{a} \cH^k$ such that $w-x\in N^\perp$. We prove the lemma by assuming the claim, and give the proof of the claim in the last paragraph. Taking $a=2d$, the claim implies that
\begin{align*}
    	&\#\{  w \in \pi \cH^k/\pi_0^d  \cH^k\mid(N,w)=0\mod (\pi^{2d-1})\}\\
    	&=	\#\{  w \in \pi (N^\perp+\pi_0^d \cH^k)/\pi_0^d  \cH^k\mid(N,w)=0\mod (\pi^{2d-1})\}.
\end{align*}
Since $N^{\perp}$ is primitive of rank $2k-i$, we have
\begin{align*}
    	\#\{  w \in \pi N^{\perp}/\pi_0^d  \cH^k\mid(N,w)=0\mod (\pi^{2d-1})\}=\q^{(2d-1)(2k-i)}.
\end{align*}
This proves the lemma assuming the claim.

Now we prove the claim.
  Consider the symplectic space over $\mathbb{F}_\q$: $(\overline{\cH}^k,\overline{\pi (\, , \,)})$.  Since $N$ is integral, $\overline{N}$ is isotropic in $\overline{\cH}^k$. Let $\overline{N}$ be spanned by $\{e_1,\cdots,e_i\}$. Then we may extend $\{e_1,\cdots,e_i\}$ to a standard symplectic basis $\{e_1,f_1,\cdots,e_k,f_k\}$ of $\overline{H}^k$, where $(e_s,f_t)=\delta_{st}$, and $(e_s,e_t)=(f_s,f_t)=0$ for $1\le s,t\le k$.  Let $\{\tilde{e}_1,\tilde{f}_1,\cdots,\tilde{e}_k,\tilde{f}_k\}$ be a lifting of $\{e_1,f_1,\cdots,e_k,f_k\}$.  By our choice of $\tilde{e}_s$, we can find a basis $\{w_1,\cdots,w_i\}$ of $N$ such that $w_s-\tilde{e}_s\in \pi \cH^k$ for any $1\le s\le i$. Consider $x=a_1\tilde{f}_1+\cdots+a_i\tilde{f}_i\in \langle \tilde{f}_1,\cdots,\tilde{f}_k\rangle$. In order to have $w-x\in N^\perp$, we need to solve the following system of equations:
 \begin{align}\label{eq: sys of eqs}
    \pi (w_s,x)=\pi (w_s,w) \text{ for $1\le s\le i$}.
 \end{align}
Let $A$ denote the $i\times i$ matrix corresponding to this system of linear equations. Since $w_s-\tilde{e}_s\in \pi \cH^k$, we have $A\equiv \mathrm{Id}_{i}\mod (\pi)$. Therefore,  there exists a unique solution $x$ of \eqref{eq: sys of eqs}. Moreover, since $\pi (N,w)=0 \mod \pi^{a}$, we have $\pi (w_s,w)\in (\pi^{a})$ for $1\le s\le i$. Then \eqref{eq: sys of eqs} implies that $a_s\in (\pi^a)$, hence $x\in \pi^a\cH^k$. The claim is proved.
\end{proof}

\section{Explicit formulas for $\Pden'(L)$}\label{sec:pPden}
\subsection{Explicit formulas and consequences}
The goal of this section is to establish the following formulas for
\begin{align}
\ppden(L) &=\Pden'(L) + \sum_{j=0}^{t_{\mathrm{max}}/2} c_{2j} \Pden_{2j}(L).
\end{align}
Here
$$
\Pden'(L) =\frac{\Pden'(I_n, L)}{\Den(I_n, I_n)}
$$
is normalized as  in (\ref{eq:PdenDerivative}). Recall, from Lemma \ref{cor: vanishing of error term}, that
\begin{equation}
  \ppden(L) =\Pden'(L)
\end{equation}
when  $L$ is not dual to some vertex lattice $\Lambda_t$ of positive type $t>0$.

\begin{theorem}\label{thm: formula of ppden}
Let $L\subset \bV$ be a lattice of rank $n$.\hfill
    \begin{enumerate}
        \item If $L$ is not integral, then
        $\ppden(L)=0.$
        \item If $L$ is unimodular, then $$\ppden(L)=\Pden'(L)=\begin{cases}
        1, & \text{if $n$ is odd},\\
        0, & \text{if $n$ is even}.
        \end{cases}$$
        \item If $L=I_{n-t}\obot L_2$ where  $L_2$ is of full type $t$, then
    \begin{align*}
			\ppden(L)=\Pden'(L)=\begin{cases}
		\prod_{\ell=1}^{\frac{t-1}{2}}(1-\q^{2\ell}), & \text{  if $t$ is odd},\\
				(1-\chi(L_2)\q^{\frac{t}{2}})\prod_{\ell=1}^{\frac{t}{2}-1}(1-q^{2\ell}), & \text{ if  $t$ is even}.
			\end{cases}
		\end{align*}
    \end{enumerate}
\end{theorem}

\begin{corollary}\label{cor: int of pden}
Let $L$ be a lattice. Then $\pden(L)\in \Z.$ Moreover, $\pden(L)=0$ for non-integral $L$.
\end{corollary}
\begin{proof}
According to Corollary \ref{cor: decomp of pden(L)}, we have
$$
	\pden(L)=\sum_{L \subset L' \subset L_{F}}   \ppden(L').
$$
Now Theorem \ref{thm: formula of ppden} implies that $\ppden(L')\in \Z$, hence $\pden(L)\in \Z$. If $L$ is non-integral, then $\ppden(L')=0$ for each $L'$ such that $L\subset L'$ by $(1)$ of Theorem \ref{thm: formula of ppden}.
\end{proof}

As another corollary, we prove the following cancellation law for $\pden(L)$. Recall that for Hermitian lattices $L$ and $L'$ of the same rank, $n(L',L)=\#\{L''\subset L_F\mid L\subset L'', L''\cong L'\}$.
\begin{corollary}\label{cor: cancellation of pDen} Let $L=L_1\obot L_2 \subset \bV$ be a  rank $n$ lattice, with
$L_1$ being unimodular and $L_i$ of rank $n_i$.  Then
\begin{align}\label{eq: pden(L)-pden(L_2)}
    \pden(L)-\pden(L_2)=n(I_{n_2},L_2)\cdot(\delta_{\mathrm{odd}}(n)-\delta_{\mathrm{odd}}(n_2)).
\end{align}
\end{corollary}
\begin{proof}
By Corollary \ref{cor: decomp of pden(L)} and Lemma \ref{lem: red to L_2}, we have
\begin{align*}
  \pden(L)=\sum_{L\subset L'\subset L_F}\ppden(L')
          =\sum_{L_2\subset L_2'\subset L_{2,F}}\ppden(L_1\obot L_2').
\end{align*}
Similarly,
\begin{align*}
  \pden(L_2)=\sum_{L_2\subset L_2'\subset L_{2,F}}\ppden(L_2').
\end{align*}
Hence
\begin{align*}
    \pden(L)-\pden(L_2)&=\sum_{L_2\subset L_2'\subset L_{2,F}}(\ppden(L_1\obot L_2')-\ppden(L_2')).
   \end{align*}
If $L_2'$ is not integral, then both $\ppden(L_1\obot L_2')$ and
$\ppden(L_2')$ vanishes by $(1)$ of Theorem \ref{thm: formula of ppden}.
If $L_2'$ is integral but not unimodular, then $(3)$ of Theorem \ref{thm: formula of ppden} implies $\ppden(L_1\obot L_2')-\ppden(L_2')=0.$
Hence
    \begin{align}\label{eq: ppden(L)-ppden(L_2)}
                       \pden(L)-\pden(L_2)&=\sum_{\substack{L_2\subset L_2'\subset L_{2,F}\\ L_1\obot L_2' \cong \Lambda_{0}}}(\ppden(L_1\obot L_2')-\ppden(L_2')).
        \end{align}
Combining \eqref{eq: ppden(L)-ppden(L_2)} with (2) of Theorem \ref{thm: formula of ppden}, we have
\begin{align*}
      \pden(L)-\pden(L_2)=n(I_{n_2},L_2)\cdot(\delta_{\mathrm{odd}}(n)-\delta_{\mathrm{odd}}(n_2)).
\end{align*}
This proves the corollary.
\end{proof}
\begin{lemma}\label{lem: red to L_2}
    Assume $L=L_1\obot L_2$ is a lattice  where $L_1$ is unimodular. If $L\subsetneq L' \subset \pi^{-1}L$ and $L'$ is not of the form $L_1\obot L_2'$, then $L'$ is not integral and $ \ppden(L')=0$.
\end{lemma}
\begin{proof}
Consider the $\mathbb{F}_\q$-vector space $\pi^{-1}L/L$.  Since we assume $L'$ is not of the form $L_1\obot L_2'$, there exists $v\in L'\setminus L$ such that $\mathrm{Pr}_{\pi^{-1}L_1}(v)\not =0$, which in turn implies that $L'$ is not integral. Hence $\ppden(L')=0$ by $(1)$ of Theorem \ref{thm: formula of ppden}.
\end{proof}

\subsection{Proof strategy}\label{subsec:proofstra} The proof of Theorem \ref{thm: formula of ppden} occupies the rest of this section. Since the proof is rather long and technical, we   summarize the main idea of the proof first. When  there is some $x \in L$ with $\val(x) \le -1$, $\partial\Pden(L)=0$ by Lemma \ref{lem: vanish val le -1}. Otherwise,  write
$$
L = H^j \obot I_{n_1-t} \obot L_2,
$$
where  $L_2$ is of full type $t$. There are four cases.
\begin{enumerate}
    \item[(a)]  The case $n_1-t=0$ (i.e., $L=H^j \oplus L_2$) is significantly simpler than the general case, and we will deal with it in next subsection although it is  part of the general case. For example, when $L$ is of full type, the reduction  $\overline{L}$ of $L$ modulo $\pi$  is a totally isotropic quadratic space over $\mathbb{F}_q$. Hence, the summation in Corollary \ref{cor: Pden} is simply:
\begin{align*}
    \sum_{V_1\in \mathrm{Gr}(i,\overline{L})(\mathbb{F}_\q)}  |\mathrm{O}(V_1,\overline{I_m})|=\binom{n}{i}_q |\mathrm{O}(0_i,\overline{I_m})|.
\end{align*}
 An application of $q$-binomial theorem settles this case.
    \item[(b)] \label{case:integral} The case $j=0$, i.e., $L$ is integral.
    \item[(c)] The case $j> 0$ and $t> 0$.
     \item[(d)] The case that  $t=0$ and $j >0$ is part of the modification assumption.
\end{enumerate}
In general, the problem becomes harder when $n_1-t$ is larger. In fact, when $t>0$, i.e., $n_1-t<n_1$, (b) and (c) can be proved via Corollary \ref{cor: Pden} and an involved application of the induction formulas of $\Den(I_n,L)$ established in \cite{HSY3}. However, when $t=0$, i.e.,  $L$ is unimodular,   this method fails. To overcome this difficulty and give a uniform proof of $(b)$ and $(c)$, we introduce a new method which is different from \cite{HSY3} even in the case $n \le 3$.

To illustrate the idea, we stick to case $(b)$ for now. The first key step is to discover a finer structure of $\ppden(L)$ and prove the following formula (Lemma \ref{lem: Pden' to g}):
\begin{align} \label{eq:ppden}
    \ppden(L)=(*)  \sum_{r=0}^{\min\{n-t, n-1\}}\sum_{i=0}^{n-1}\binom{n-r}{i-r}_q (-1)^{n-1+i+r} q^{\frac{(i-r)(i-r-1)}{2}}  q^{rt} \cdot g(n,n_1,r,q^{-i}),
\end{align}
where $(*)$ is some constant number and $g(n,n_1,r,X)$ is a linear combination of polynomials of degree $n-1$. The second key  observation is that there is a lot of cancellation underlying this linear combination. Indeed,  we show for $r<n$ that $g(n,n_1,r,X)$ is actually of degree $\le n-r-1$  and is essentially a simple multiple of some simple polynomial (denoted by $h(n, r, X)$)  (Lemmas \ref{lem: ind of g} and \ref{lem: g(r,r,X)}).  This  enables us to apply $q$-binomial theorem (Corollary \ref{cor: qbinomial vanish}) to the inner sums in (\ref{eq:ppden}). Consequently, we obtain
\begin{align}
    \ppden(L)=(*)  \sum_{r=0}^{\m(n,t)} (-1)^r q^{\frac{(n-r)(n-r-1)}{2}}  q^{rt} \cdot g(n,n-t,r,q^{-n}).
\end{align}
The last step is to evaluate this sum and the result is given by Lemma \ref{lem: sum of g}. It is in this step that the case $n_1=n$ ($L$ is unimodular) becomes different: the sum above is a sum from  $r=0$ only to $n-1$, not to $n_1=n$.  To make it worse, the `missing' term $g(n, n, n, X)$ is in fact ill-behaved.

One common strategy in proving Lemmas \ref{lem: g(r,r,X)} and \ref{lem: sum of g} is to express both sides of the identity as (uniquely) linear combinations of certain basis of polynomials, and prove that the coefficients satisfy the same recursion formulas and the same initial conditions. Here we use crucially the combinatorical properties of $m(U,V)$ (Lemma \ref{lem: counting subspace}) for $U$ and $V$ quadratic spaces over $\mathbb{F}_q$.

In  Case (c),  the derivative becomes the value of some primitive local density polynomials at some non-central point $q^j$ by Lemma \ref{lem: L=H^i obot L_2}. Strikingly, the formula for this value (see (\ref{eq: non integral pf of main})) is very similar to the formula  for the derivative $\Pden(I_{n_1-t} \obot L_2)$ (see (\ref{eq:PdenD})). Proof of Cases (b) and (c) will be given in Subsection \ref{proof} after long preparation in Subsections \ref{preparation} and \ref{sec:identity}.

\subsection{The case $n_1-t=0$} In this subsection we assume that $n_1-t=0$ and divide it further into two subcases: $j=0$ or $j>0$.
	\begin{proposition}\label{prop: L=H obot L_1, v(L_1)>)}
	Assume that $L=H^j\obot L_2$ where $j>0$ and $L_2$ is of full type and has rank $n_2=n-2j$. Then
		\begin{align*}
		\ppden(L)=0.
		\end{align*}
	\end{proposition}
	\begin{proof}
	By Lemma \ref{lem: L=H^i obot L_2}, we have
\begin{align*}
    \Pden'(I_n,L)=2\Big(\prod_{\ell=1}^{j-1}(1-\q^{2\ell} )\Big)\Pden(I_n,L_2,\q^{2j}).
\end{align*}
	Hence it suffice to show
		\begin{align*}
		\Pden(I_{n},L_2,\q^{n-n_2})=0.
		\end{align*}
We prove the odd $n$ case and leave the even $n$ case to the reader.		 According to Corollary \ref{cor: L full type Pden(Im,L)}, we have
		\begin{align*}
			\Pden(I_{n} ,L_2,\q^{n-n_2})
			&=\sum_{i=0}^{n_2}	(-1)^{n_2-i}\q^{\frac{i(i-1)}{2}}\binom{n_2}{i}_\q\cdot  \prod_{\ell=0}^{n_2-i-1}(\q^{2\ell+n-n_2}-1)\cdot  \prod_{\ell=1}^{i}(\q^{n+1-2\ell}-1)\\
			&=\sum_{i=0}^{n_2}	(-1)^{n_2-i}\q^{\frac{i(i-1)}{2}}\binom{n_2}{i}_\q\cdot  \prod_{\ell=\frac{n-n_2}{2}}^{\frac{n+n_2}{2}-i-1}(\q^{2\ell}-1)\cdot  \prod_{\ell=\frac{n+1}{2}-i}^{\frac{n-1}{2}}(\q^{2\ell}-1).
		\end{align*}
		We can factor out
		$
	\displaystyle	\prod_{\ell=\frac{n-n_2}{2}}^{\frac{n-1}{2}}(\q^{2\ell}-1)
		$
		so that
		\begin{align*}
			\Pden(I_{n} ,L_2,\q^{n-n_2})\prod_{\ell=\frac{n-1}{2}}^{\frac{n-n_2}{2}}(\q^{2\ell}-1)^{-1}
			&=\sum_{i=0}^{n_2}	(-1)^{n_2-i}\q^{\frac{i(i-1)}{2}}\binom{n_2}{i}_\q g( n_2, q^{-i}),
		\end{align*}
where
\begin{equation} \label{eq:gn2}
g(n_2, X) = \prod_{\ell=\frac{n+1}{2}}^{\frac{n_2+n}{2}-1}(\q^{2\ell} X^2-1)
\end{equation}
is a polynomial of $X$ of degree $n_2-1$. Now $	\Pden(I_{n} ,L_2,\q^{n-n_2})=0$	by Corollary \ref{cor: qbinomial vanish}.
	\end{proof}	
	
		\begin{proposition}\label{prop: prim local density full type}
		Assume that $L$ is a full type lattice of rank $n$.  We have
		\begin{align*}
			\Pden'(L)=	\begin{cases}
				\prod_{\ell=1}^{\frac{n-1}{2}}(1-\q^{2\ell}), & \text{ if $n$ is odd},\\
				(1-\epsilon\q^{\frac{n}{2}})\prod_{\ell=1}^{\frac{n}{2}-1}(1-q^{2\ell}), & \text{ if $n$ is even}.
			\end{cases}
		\end{align*}
	\end{proposition}
	\begin{proof}
	First of all, recall that
	\begin{align*}
	 \Pden'(L)=\frac{\Pden'(I_n,L)}{\Den(I_n,I_n)},
	\end{align*}
	where
	\begin{align}\label{eq: norm factor}
	    \Den(I_n,I_n)=\begin{cases}
	    2q^{(\frac{n-1}2)^2}\prod_{\ell=1}^{\frac{n-1}2}(q^{2\ell}-1) & \text{ if $n$ is odd},\\
	      2q^{(\frac{n}2)(\frac{n}2-1)}(q^{\frac{n}2}+\epsilon)\prod_{\ell=1}^{\frac{n}2-1}(q^{2\ell}-1) & \text{ if $n$ is even}.
	    \end{cases}
	\end{align}
We verify the even  $n$ case and leave the odd $n$ case to the reader.
Direct calculation using 		 Corollary \ref{cor: L full type Pden(Im,L)} gives
		\begin{align*}
			\Pden'(I_{n},L)\Big((\q^{\frac{n}{2}}+\epsilon)\prod_{\ell=1}^{\frac{n}{2}-1}(\q^{2\ell}-1)\Big)^{-1}&=2\sum_{i=0}^{\frac{n}{2}-1}(-1)^{n-i-1}\q^{\frac{i(i-1)}{2}}\binom{n}{i}_\q  (\q^{\frac{n}{2}-i}-\epsilon)\prod_{\ell=\frac{n}{2}-i+1}^{n-i-1}(\q^{2\ell}-1)\\
			&=2\sum_{i=0}^{n-1}(-1)^{n-i-1}\q^{\frac{i(i-1)}{2}}\binom{n}{i}_\q   (\q^{\frac{n}{2}-i}-\epsilon)\prod_{\ell=\frac{n}{2}+1}^{n-1}(\q^{2\ell-2i}-1).
		\end{align*}
		According to Corollary \ref{cor: qbinomial vanish},
		\begin{align*}
		\sum_{i=0}^{n-1}(-1)^{n-i-1}\q^{\frac{i(i-1)}{2}}\binom{n}{i}_\q   (\q^{\frac{n}{2}-i}-\epsilon)\prod_{\ell=\frac{n}{2}+1}^{n-1}(\q^{2\ell-2i}-1)&=q^{\frac{n(n-1)}{2}}   (\q^{-\frac{n}{2}}-\epsilon)\prod_{\ell=\frac{n}{2}+1}^{n-1}(\q^{2\ell-2n}-1)\\
		&=q^{(\frac{n}{2})(\frac{n}{2}-1)}   (1-\epsilon\q^{\frac{n}{2}})\prod_{\ell=1}^{\frac{n}2-1}(1-\q^{2\ell}).
		\end{align*}
Now by \eqref{eq: norm factor}, we conclude that
		\begin{align*}
			\Pden'(L)=(1-\epsilon\q^{\frac{n}{2}})\prod_{\ell=1}^{\frac{n}{2}-1}(1-q^{2\ell}),
		\end{align*}
as claimed.
	\end{proof}

\subsection{Preparation} \label{preparation}  In this subsection, we rewrite $\mathrm{Pden}^{\prime}(I_n,L)$ as a linear combination of special values of certain polynomial $g_{\epsilon_1}(n,m,r,X)$ as in Lemma \ref{lem: Pden' to g}. We then express general $g_{\epsilon_1}(n,m,r,X)$ as a simple combination of $g_{\epsilon_3}(n,r,r,X)$, see Lemma \ref{lem: ind of g}.

Let $L=I_{n-t}^{\epsilon_1}\obot L_2$ where $L_2$ is of full type $t$. By Corollary \ref{cor: Pden}, we have
\begin{align}\label{eq: Pden' as sum}
\mathrm{Pden}^{\prime}(I_n,L)=2 \sum_{i=0}^{n-1} \prod_{\ell=1}^{n-i-1}(1-q^{2 \ell})  \sum_{0 \leq j \leq i}\sum_{\epsilon_{2} \in\{\pm 1\}} m\left(0_{j} \oplus U_{i-j}^{\epsilon_{2}}, 0_{t} \oplus U_{n-t}^{\epsilon_1}\right)\left|\mathrm{O}\left(0_{j} \obot U_{i-j}^{\epsilon_{1}}, U_{n}^{-\epsilon}\right)\right|.
\end{align}
Here and in the following, we interpret $\sum_{\epsilon_2\in \{\pm 1\}}f(U_0^{\epsilon_2})$ as $f(U_0^1)$ for a function  $f$ with  $U_i^{\epsilon_2}$ as input.

Let $s$ and $n$ be integers such that $0\le s< n$, and let $\epsilon_2=\pm 1$. For odd $n$, we define
\begin{align}\label{eq: def of f n odd}
    f_{\epsilon_2}(n,s,X)=\begin{cases}
    \prod_{\ell=\frac{n+1}{2}}^{\frac{n+s-2}{2}}(\q^{2\ell} X^2) \cdot   \q^{\frac{n+s}{2}}X( \q^{\frac{n+s}{2}}X-\epsilon \epsilon_2) \cdot \prod_{\ell = \frac{n+s+2}{2} }^{n-1}(\q^{2\ell}X^2-1)& \text{ if $s$ is odd,}\\
        \prod_{\ell=\frac{n+1}{2}}^{\frac{n-1}{2}+\frac{s}{2}}(q^{2\ell} X^2) \cdot   \prod_{\ell =\frac{n+1+s}{2}}^{n-1}(\q^{2\ell}X^2-1) & \text{ if $s$ is even.}
    \end{cases}
\end{align}
Similarly, for an even integer $n>0$ and $0\le s<n$, we define

\begin{align}\label{eq: def of f n even}
    f_{\epsilon_2}(n,s,X)=  \begin{cases}
    \prod_{\ell=\frac{n}{2}}^{\frac{n+s-3}{2}}(\q^{2\ell} X^2)  q^{\frac{n}{2}+s-1}X   \cdot \prod_{\ell = \frac{n+s+1}{2} }^{n-1}(\q^{2\ell}X^2-1)& \text{ if $s$ is odd,}\\
        \prod_{\ell=\frac{n}{2}}^{\frac{n+s-2}{2}}(q^{2\ell} X^2) \cdot  q^{\frac{s}{2}}(q^{\frac{n+s}{2}}X-\epsilon\epsilon_2)\cdot\prod_{\ell =\frac{n+s+2}{2}}^{n-1}(\q^{2\ell}X^2-1) & \text{ if $s$ is even.}
    \end{cases}
\end{align}

Here when $s=0$, we always take $\epsilon_2=1$. Notice that $ f_{\epsilon_2}(n,s,X)$ is a polynomial of degree $n-1$.

\begin{lemma}\label{lem: tran to poly}\hfill
\begin{enumerate}
    \item Assume  $0\le i < n$. We have
\begin{align*}
&\prod_{\ell=1}^{n-i-1}(1-\q^{2\ell}) |\mathrm{O}(0_{i-s}\obot U_{s}^{\epsilon_2},U_n^{-\epsilon})| = (-1)^{n-i-1}q^{\frac{i(i-1)}{2}} I_\epsilon(n) f_{\epsilon_2}(n,s,q^{-i}),
\end{align*}
where
\begin{align*}
&I_\epsilon(n)
\coloneqq  \begin{cases}
     \prod_{\ell=1}^{\frac{n-1}{2}}(q^{2\ell}-1)  & \text{ if $n$ is odd},\\
    (q^{\frac{n}{2}}+\epsilon)   \prod_{\ell=1}^{\frac{n}{2}-1}(q^{2\ell}-1)  & \text{ if $n$ is even}.
\end{cases}
\end{align*}
    \item Assume $0\le i\le n$, $s<n$ and that  $n'-n>0$ is even. We have
    \begin{align*}
&\prod_{\ell=0}^{n-i-1}(1-\q^{2\ell+n'-n}) |\mathrm{O}(0_{i-s}\obot U_{s}^{\epsilon_2},U_{n'}^{-\epsilon})| = (-1)^{n-i}q^{\frac{i(i-1)}{2}} I_\epsilon(n',n) f_{\epsilon_2}(n,s,q^{\frac{n'-n}2-i}),
\end{align*}
where
\begin{align*}
&I_\epsilon(n',n)
\coloneqq  \begin{cases}
     \prod_{\ell=\frac{n'-n}2}^{\frac{n'-1}{2}}(q^{2\ell}-1)  & \text{ if $n$ is odd},\\
    (q^{\frac{n'}{2}}+\epsilon)   \prod_{\ell=\frac{n'-n}{2}}^{\frac{n'}2-1}(q^{2\ell}-1)  & \text{ if $n$ is even}.
\end{cases}
\end{align*}
\end{enumerate}

\end{lemma}
\begin{proof}
This follows from the formula of $|\mathrm{O}(0_{i-s}\obot U_{s}^{\epsilon_2},U_n^{-\epsilon})|$ given in Lemma \ref{lem: Pden(I_m,L)} and a straightforward computation.
\end{proof}

\begin{lemma}
For integers  $0\le i,   t \le  n$, we have
  \begin{align}\label{eq: guess}
    \binom{t}{i}_q= \sum_{a=0}^{n-t}(-1)^a \cdot q^{a(t+1-i)}q^{\frac{a(a-1)}{2}}\cdot \binom{n-t}{a}_q \cdot \binom{n-a}{i-a}_q.
\end{align}
\end{lemma}
\begin{proof} The identity is automatically true for $i >t $ as both sides are zero.
Recall the following analogue of Pascal's identity for $q$-binomial coefficients:
\begin{align}\label{eq: pascal}
    \binom{t}{i}_q=\binom{t+1}{i}_q-q^{t-i+1}\binom{t}{i-1}_q.
\end{align}
By this identity, we obtain $2$ terms, one with the $t$-index raised, another with the $i$-index lowered.
Applying again \eqref{eq: pascal} to $\binom{t+1}{i}_q$ and $\binom{t}{i-1}_q$ respectively, we obtain
\begin{align*}
      \binom{t}{i}_q=\binom{t+2}{i}_q-q^{t-i+2}\binom{t+1}{i-1}_q-q^{t-i+1}\Big(\binom{t+1}{i-1}_q-q^{t-i+2}\binom{t-1}{i-2}_q\Big).
\end{align*}
We may continue this process and after $n-t$ steps, we obtain $2^{n-t}$ many terms. Each term corresponds to a lattice path starting from the origin, going to north and east as follows. If the $\ell$-th step raises the index of $t$ (resp. reduces the index of $i$), we define the lattice path goes towards north for the $\ell$-th step (resp. east). We use $I=(i_1,\cdots,i_{n-t})$ where $i_\ell\in \{0,1\}$ to denote the path whose $\ell$-th step goes towards north (east) if $i_\ell=0$ ($i_\ell$ =1)   and let $|I|=i_1+\cdots+i_{n-t}$. We use $P_I$ to denote the term corresponding to $I$. Now the lemma follows from the following claim.\newline
Claim: $$\displaystyle\sum_{\substack{I,  |I|=a}}P_I=(-1)^a \cdot q^{a(t+1-i)}q^{\frac{a(a-1)}{2}}\cdot \binom{n-t}{a}_q \cdot \binom{n-a}{i-a}_q.$$
Indeed a direct calculation shows that
\begin{align*}
    P_{(1^a,0^{n-t-a})}=(-1)^a \cdot q^{a(t+1-i)}q^{\frac{a(a-1)}{2}}\cdot \binom{n-a}{i-a}_q.
\end{align*}
Let $A_I$ denote the area bounded by the lattice path $I$, the horizontal axis, and the vertical line given by $x=|I|$. Then a direct computation shows that
\begin{align*}
    P_{I}=q^{A_I} \cdot P_{(1^a,0^{n-t-a})}.
\end{align*}
Now  the claim follows from the well-known formula of $q$-binomial coefficient (see \cite[Theorem 6.9]{Carm} for example):
\begin{align*}
    \sum_{I,|I|=a}q^{A_I} =\binom{n-t}{a}_q.
\end{align*}
This proves the claim and the lemma.
\end{proof}

\begin{lemma}\label{lem: m 0_j sum}
For an integer $n\ge 0$ and $\epsilon =\pm 1$,  let
\begin{align*}
    \alpha(n)=q^{\lfloor\frac{n}2\rfloor \lfloor\frac{n-1}2\rfloor} =\begin{cases}
       q^{(\frac{n-1}2)^2} &\hbox{ if  $n$ is odd},
      \\
        q^{\frac{n}{2} (\frac{n}{2}-1)} &\hbox{ if  $n$ is even},
\end{cases}\quad \text{ and } \quad
\beta_\epsilon(n)=\begin{cases}
     (-1)^{\frac{n-1}{2}}     &\hbox{ if  $n$ is odd},
      \\
        \epsilon  (-1)^{\frac{n}{2}} &\hbox{ if  $n$ is even},\\
        1 & \text{ if $n=0$}.
\end{cases}
\end{align*}
Then
$$
\sum_{j=0}^n (-1)^j q^{\frac{j(j-1)}2} m(0_j, U_n^\epsilon) =  \alpha(n)\beta_\epsilon(n).
$$
\end{lemma}
\begin{proof}
If $n=0$, the statement both sides are $1$ by definition.  From now on we assume $n>0$. By \cite[Lemma 3.2.2.]{LZ2}, we have
$$
|m(0_j, U_n^\epsilon)|= \binom{d}{j}_\q  \cdot \prod_{l=1}^{j} (q^{d+e-l}+1),
$$
with
$$
d=\begin{cases}
\frac{n-1}2 &\ff \, n \,  \hbox{ is odd},
\\
\frac{n}2 - \frac{1-\epsilon}2 &\ff \, n \, \hbox{ is even},
\end{cases} \quad \text{ and } \quad
e=\begin{cases}
1 &\ff \,  n \hbox{ is odd},
\\
1-\epsilon &\ff \,  n \hbox{ is even}.
\end{cases}
$$
So we have
$$
\sum_{j=0}^{n}(-1)^jq^{\frac{j(j-1)}{2}}\cdot m(0_j,U_{n}^\epsilon)
=\sum_{j=0}^d (-1)^j q^{\frac{j(j-1)}2} \binom{d}{j}_q \prod_{l=1}^{j} (q^{d+e-l}+1),
$$
which by Lemma \ref{lem: inverse q bin} equals to
$(-q^{d+e-1})^d$. A direct calculation checks that $(-q^{d+e-1})^d=\alpha(n)\beta_\epsilon(n).$
\end{proof}

\begin{definition}
  For $0\le r\le m\le n$, we define
\begin{align*}
   &g_{\epsilon_1}(n,m,r,X)\\ \notag
  &\coloneqq\sum_{k=0}^r (-1)^k q^{\frac{k(k-1)}{2}}\sum_{\epsilon_2\in \{\pm 1\}}m(U_{k}^{\epsilon_2},U_{m}^{\epsilon_1}) \sum_{j=0}^{r-k}(-1)^{j} q^{\frac{j(j-1)}{2}} \cdot   \binom{m-j-k}{r-j-k}_q \cdot
     m(0_{j}, U_{m-k}^{\delta}) f_{\epsilon_2}(n,k,X)
\end{align*}
with $ \delta=\delta(m,k,\epsilon_1,\epsilon_2)$ as defined in  \eqref{eq:delta}.
In the following, when $n$ is clear in the context, we simply write $g_{\epsilon_1}(m,r,X)$ for $g_{\epsilon_1}(n,m,r,X)$.
\end{definition}
In particular, $ g_{\epsilon_1}(n,m,0,X)=f_1(n,0,X)$ and by Lemma \ref{lem: m 0_j sum},
\begin{align}\label{eq: g(m,m,X)}
     g_{\epsilon_1}(n,r,r,X)=&\sum_{k=0}^{r} (-1)^k q^{\frac{k(k-1)}{2}} \cdot
   \sum_{\epsilon_2 \in \{\pm 1\}} m(U_k^{\epsilon_2},U_{r}^{\epsilon_1})\cdot \alpha(r-k)\beta_{\delta}(r-k)\cdot  f_{\epsilon_2}(n,k,X).
\end{align}

In the rest of this section, we let $m=n-t$ without explicit mentioning.
\begin{lemma}\label{lem: Pden' to g}
Assume $L$ is a lattice of rank $n$ and type $t$.
\begin{enumerate}
   \item   Let $\m(n,t)\coloneqq \mathrm{min}\{n-t,n-1\}$. Then for $0\le i\le n-1$, we have
  \begin{equation*}
       (\Pden^{n-i})'(I_n,L)=  2 I_{\epsilon}(n) \cdot \sum_{r=0}^{\m(n,t)}\binom{n-r}{i-r}_q (-1)^{i-r+n-1} q^{\frac{(i-r)(i-r-1)}{2}}  q^{r(n-m)} \cdot g_{\epsilon_1}(n,m,r,q^{-i}).
  \end{equation*}
    \item Assume that  $n'-n$ is a positive even integer and $m<n$. Then for $0\le i\le n$, we have
  \begin{align*}
   &\Pden^{n-i} (I_{n'}^{-\epsilon},L,q^{n'-n}) = I_{\epsilon}(n',n)  \sum_{r=0}^{m}\binom{n-r}{i-r}_q (-1)^{i-r+n} q^{\frac{(i-r)(i-r-1)}{2}}  q^{r(n-m)} \cdot g_{\epsilon_1}(n,m,r,q^{\frac{n'-n}2-i}).
  \end{align*}
\end{enumerate}

\end{lemma}
\begin{proof}
We prove $(1)$ first.
By \eqref{eq: Pden' as sum},  Lemma \ref{lem: tran to poly} and Lemma \ref{lem: m with type}, we have
\begin{align*}
  &(\Pden^{n-i})'(I_n,L)\cdot \Big(2(-1)^{n-i-1}I_{\epsilon}(n)  \Big)^{-1}\\
  &= \sum_{\ell=0}^{t}\binom{t}{\ell}_q \sum_{k=0}^{i-\ell}\sum_{\epsilon_2 \in {\pm 1}} q^{(t-\ell)(i-\ell)}q^{\frac{i(i-1)}{2}}m(0_{i-k-\ell}\obot U_{k}^{\epsilon_2}, U_{n-t}^{\epsilon_1}) f_{\epsilon_2}(n,k,q^{-i})\\
  &= \sum^{i}_{s=\mathrm{max}\{i-t,0\}}\binom{t}{i-s}_q \sum_{k=0}^{s}\sum_{\epsilon_2 \in {\pm 1}} q^{(t-(i-s))s}q^{\frac{i(i-1)}{2}}m(0_{s-k}\obot U_{k}^{\epsilon_2}, U_{n-t}^{\epsilon_1}) f_{\epsilon_2}(n,k,q^{-i}),
\end{align*}
where the last identity is obtained by setting $s=i-\ell$. Notice that if $s>n-t$, then $m(0_{s-k}\obot U_{k}^{\epsilon_2}, U_{n-t}^{\epsilon_1})=0$. Hence we may assume $s\le n-t$, or equivalently $t\le n-s$ in the above summation.
Now applying \eqref{eq: guess} to $i-s\le t\le n-s$ and let $m=n-t$, we may rewrite the above summation as
\begin{align*}
&\sum^{i}_{s=\mathrm{max}\{i-t,0\}}  \sum_{a=0}^{m-s}(-1)^a \binom{n-s-a}{i-s-a}_q\cdot  \binom{m-s}{a}_q \cdot \\
   &\quad\quad\quad\quad  \sum_{k=0}^{s}\sum_{\epsilon_2 \in {\pm 1}} q^{a(n-m+1-(i-s))+\frac{a(a-1)}{2}+(n-m+s-i)s} q^{\frac{i(i-1)}{2}}m(0_{s-k}\obot U_{k}^{\epsilon_2}, U_{m}^{\epsilon_1}) f_{\epsilon_2}(n,k,q^{-i})\\
  &=\sum^{i}_{s=\mathrm{max}\{i-t,0\}} \sum_{a=0}^{m-s}\binom{n-s-a}{i-s-a}_q   q^{\frac{(i-(s+a))(i-(s+a+1))}{2}}q^{(s+a)(n-m)}\\
    & \quad\quad\quad\quad \cdot(-1)^a \binom{m-s}{a}_q\sum_{k=0}^{s}\sum_{\epsilon_2 \in {\pm 1}} q^{(s+a)(s-\frac{s+a+1}2)+\frac{a(a+1)}{2}} m(0_{s-k}\obot U_{k}^{\epsilon_2}, U_{m}^{\epsilon_1}) f_{\epsilon_2}(n,k,q^{-i}).
\end{align*}
Now let $r=s+a$. Notice that $r\le m$ and  $\binom{n-r}{i-r}_q=0$ for $r>i$. Rearranging the summation index, we obtain
\begin{align*}
 &(\Pden^{n-i})'(I_n,L)\cdot \Big((-1)^{n-i-1}I_{\epsilon}(n)  \Big)^{-1}=\sum_{r=0}^{\m\{n,t\}}\binom{n-r}{i-r}_q q^{\frac{(i-r)(i-r-1)}{2}}  q^{r(n-m)} \cdot g_{\epsilon_1}(n,m,r,q^{-i}),
  \end{align*}
   where
   \begin{align*}
     &g_{\epsilon_1}(n,m,r,X)\\
     &=  \sum_{s=0}^{r}(-1)^{r-s} \cdot   \binom{m-s}{r-s}_q \cdot
     \sum_{k=0}^{s}\sum_{\epsilon_2 \in {\pm 1}} q^{\frac{s(s-1)}{2}} m(0_{s-k}\obot U_{k}^{\epsilon_2}, U_{m}^{\epsilon_1}) f_{\epsilon_2}(n,k,X)\\
     &=\sum_{k=0}^r (-1)^k q^{\frac{k(k-1)}{2}}\sum_{\epsilon_2\in \{\pm 1\}}m(U_{k}^{\epsilon_2},U_{m}^{\epsilon_1}) \sum_{j=0}^{r-k}(-1)^{j} q^{\frac{j(j-1)}{2}} \cdot   \binom{m-j-k}{r-j-k}_q \cdot
     m(0_{j}, U_{m-k}^{\delta}) f_{\epsilon_2}(n,k,X).
   \end{align*}
Here, we use Lemma \ref{lem: counting subspace} to obtain the last identity.

Using $(2)$ of Lemma \ref{lem: tran to poly}, the same proof of $(1)$ proves $(2)$.
\end{proof}

We conclude this subsection by establishing a formula to express  $g_{\epsilon_1}(m,r,X)$ in terms of $g_{\epsilon_3}(r,r,X)$, which, as we will see, has a particular simple form (Lemma \ref{lem: g(r,r,X)}). First, we need the following identity which might have independent interest.
\begin{lemma}\label{lem: binom times m sum}
For any integers $0\le r  \le n$, we have
\begin{equation}\label{eq: binom times m sum}
  \sum_{j=0}^r (-1)^j q^{j(j-1)/2} \binom{n-j}{r-j}_q m(0_j,U_n^\epsilon)=\sum_{\tau\in \{\pm 1\}} m(U_r^\tau,U_n^\epsilon) \alpha(r) \beta_{\tau}(r),
\end{equation}
where $\alpha(r)$ and $\beta_\tau(r)$ are defined in Lemma \ref{lem: m 0_j sum}.
\end{lemma}
\begin{proof}
We proceed by induction on $n$. The case  $n=1$ is obvious.
Now recall the identities
\begin{align*}
  m(0_j,U_n^\epsilon) &=\binom{n}{j}_q-\sum_{i=1}^j \sum_{\sigma\in \{\pm 1\}} m(0_{j-i}\obot U_i^{\sigma}, U_n^\epsilon)
  \\
  &=\binom{n}{j}_q -\sum_{i=1}^j \sum_{\sigma\in \{\pm 1\}} q^{-(j-i)i} m(U_i^{\sigma}, U_n^\epsilon) m(0_{j-i}, U_{n-i}^{\delta(n,i,\epsilon,\sigma)}),
\end{align*}
by \eqref{eq:m equal m times m}, and
\begin{equation}
    \sum_{j=0}^r (-1)^j q^{j(j-1)/2} \binom{n-j}{r-j}_q \binom{n}{j}_q=\sum_{j=0}^r (-1)^j q^{j(j-1)/2} \binom{n}{r}_q \binom{r}{j}_q=0.
\end{equation}
These imply that
\begin{align*}
 &\sum_{j=0}^r (-1)^j q^{j(j-1)/2} \binom{n-j}{r-j}_q m(0_j,U_n^\epsilon) \\
 =&- \sum_{j=0}^r (-1)^j q^{j(j-1)/2} \binom{n-j}{r-j}_q \sum_{i=1}^j \sum_{\sigma\in \{\pm 1\}} q^{-(j-i)i} m(U_i^{\sigma}, U_n^\epsilon) m(0_{j-i}, U_{n-i}^{\delta(n,i,\epsilon,\sigma)})\\
    =& \sum_{i=1}^r \sum_{\sigma\in\{\pm 1\}} (-1)^{i+1} q^{i(i-1)/2} m(U_i^{\sigma}, U_n^\epsilon)\sum_{j=0}^{r-i} (-1)^{j} q^{j(j-1)/2}  \binom{n-i-j}{r-i-j}_q m(0_j, U_{n-i}^{\delta(n,i,\epsilon,\sigma)}).
\end{align*}
where in the last step we switch the order of summation and substitute $j$ by $j+i$.

We can now use the induction hypothesis
\[\sum_{j=0}^{r-i} (-1)^{j} q^{j(j-1)/2}  \binom{n-i-j}{r-i-j}_q m(0_j, U_{n-i}^{\delta(n,i,\epsilon,\sigma)})=
 \sum_{\tau\in \{\pm 1\}} m(U_{r-i}^{\tau},U_{n-i}^{\delta(n,i,\epsilon,\sigma)})  \alpha(r-i) \beta_{\tau}(r-i)\]
and
\[  m(U_i^{\sigma}, U_n^\epsilon) m(U_{r-i}^{\tau},U_{n-i}^{\delta(n,i,\epsilon,\sigma)})=m(U_r^{\delta'}, U_n^{\epsilon}) m(U_{i}^{\sigma}, U_r^{\delta'}), \]
where $\delta'\in \{\pm 1\}$ such that $\tau=\delta(r,i,\delta',\sigma)$ (see Lemma \ref{lem:number of orthogonal flag}) to obtain
\begin{align}
     &\sum_{j=0}^r (-1)^j q^{j(j-1)/2} \binom{n-j}{r-j}_q m(0_j,U_n^\epsilon) \notag\\
    =& \sum_{i=1}^r \sum_{\sigma\in\{\pm 1\}} (-1)^{i+1} q^{i(i-1)/2}  \sum_{\tau\in \{\pm 1\}}  m(U_r^{\delta'}, U_n^{\epsilon}) m(U_{i}^{\sigma}, U_r^{\delta'})  \alpha(r-i) \beta_{\tau}(r-i)
\end{align}
Hence \eqref{eq: binom times m sum} is equivalent to
\begin{align}\label{eq:equivalent form binom times m sum}
   \sum_{i=0}^r \sum_{\sigma\in\{\pm 1\}} (-1)^{i} q^{i(i-1)/2} \sum_{\tau\in \{\pm 1\}}  m(U_r^{\delta'}, U_n^{\epsilon}) m(U_{i}^{\sigma}, U_r^{\delta'})  \alpha(r-i) \beta_{\tau}(r-i)=0
\end{align}
Now applying Lemma \ref{lem: m 0_j sum}, the left hand side of \eqref{eq:equivalent form binom times m sum} is equal to
\begin{align*}
    & \sum_{i=0}^r \sum_{\sigma\in\{\pm 1\}} (-1)^{i} q^{i(i-1)/2}  \sum_{\tau\in \{\pm 1\}}  m(U_r^{\delta'}, U_n^{\epsilon}) m(U_{i}^{\sigma}, U_r^{\delta'})  \sum_{j=0}^{r-i} (-1)^j q^{j(j-1)/2} m(0_j,U_{r-i}^\tau)\\
    =&  \sum_{\tau\in \{\pm 1\}}  m(U_r^{\delta'}, U_n^{\epsilon}) \sum_{i=0}^r  \sum_{j=0}^{r-i} \sum_{\sigma\in\{\pm 1\}} (-1)^{i+j} q^{(i+j)(i+j-1)/2}   m(0_j\obot U_{i}^{\sigma}, U_r^{\delta'})   \\
    =&  \sum_{\delta'\in \{\pm 1\}}  m(U_r^{\delta'}, U_n^{\epsilon}) \sum_{k=0}^r (-1)^k q^{k(k-1)/2} \binom{r}{k}_q=0.\\
\end{align*}
The lemma is proved.
\end{proof}

\begin{lemma}\label{lem: ind of g}
For $0 \le r \le m < n$, we have
  \begin{align*}
    g_{\epsilon_1}(m,r,X)&=\sum_{\epsilon_3 \in \{\pm 1\}}m(U_r^{\epsilon_3},U_{m}^{\epsilon_1})g_{\epsilon_3}(r,r,X).
    \end{align*}
\end{lemma}
\begin{proof}
When $r=0$, we have by definition
$$g(m,0,X)=f_{1}(n,0,X)=\sum_{\epsilon_3 \in \{\pm 1\}}m(U_0^{\epsilon_3},U_{m}^{\epsilon_1})g_{\epsilon_3}(0,0,X).$$

Now we assume $r>0$, $\delta=\delta(m, k, \epsilon_1, \epsilon_2)$ and  $\delta'=\delta(r,k,\epsilon_3,\epsilon_2)$.
On the one hand, by definition
\begin{align*}
       &g_{\epsilon_1}(m,r,X)\\
       &=\sum_{k=0}^{r}(-1)^k q^{\frac{k(k-1)}2}\sum_{\epsilon_2\in \{\pm 1\}}m(U_{k}^{\epsilon_2},U_{m}^{\epsilon_1})\sum_{j=0}^{r-k}(-1)^{j} q^{\frac{j(j-1)}{2}} \cdot   \binom{m-j-k}{r-j-k}_q \cdot
     m(0_{j}, U_{m-k}^{\delta})f_{\epsilon_2}(n,k,X).
\end{align*}
On the other hand,  we have by \eqref{eq: g(m,m,X)},
\begin{align*}
 &\sum_{\epsilon_3 \in \{\pm 1\}}m(U_r^{\epsilon_3},U_{m}^{\epsilon_1})g_{\epsilon_3}(r,r,X)\\
    &=\sum_{k=0}^{r} (-1)^k q^{\frac{k(k-1)}{2}} \cdot
 \sum_{\epsilon_2 \in \{\pm 1\}}\sum_{\epsilon_3 \in \{\pm 1\}}  m(U_k^{\epsilon_2},U_{r}^{\epsilon_3})m(U_r^{\epsilon_3},U_{m}^{\epsilon_1}) \cdot \alpha(r-k)\beta_{\delta'}(r-k)\cdot  f_{\epsilon_2}(n,k,X).
\end{align*}
By Lemma  \ref{lem:number of orthogonal flag},
\begin{align*}
    m(U_k^{\epsilon_2},U_{r}^{\epsilon_3})m(U_r^{\epsilon_3},U_{m}^{\epsilon_1})=m(U_{k}^{\epsilon_2},U_{m}^{\epsilon_1})m(U_{r-k}^{\delta'} , U_{m-k}^{\delta}).
\end{align*}
Hence, in order to prove the theorem, it suffices to show for any $k$ and $\epsilon_2$
\begin{align*}
    &\sum_{j=0}^{r-k}(-1)^{j} q^{\frac{j(j-1)}{2}} \cdot   \binom{m-j-k}{r-j-k}_q \cdot
     m(0_{j}, U_{m-k}^{\delta})  =\sum_{\delta' \in \{\pm 1\}}  m(U_{r-k}^{\delta'} , U_{m-k}^{\delta }) \cdot \alpha(r-k)\beta_{\delta'}(r-k),
\end{align*}
which is exactly the content of Lemma \ref{lem: binom times m sum}.
\end{proof}

\subsection{Some identities between polynomials}   \label{sec:identity}
Although $g_{\epsilon_1}(r,r,X)$ by definition is a complicated linear combination of  $f_{\epsilon_2}(n,k,X)$. We show in this subsection that in fact $g_{\epsilon_1}(r,r,X)$ has a very simple form (Lemma \ref{lem: g(r,r,X)}). Similarly, although $\Pden'(I_n,L)$  is a complicated linear combination of the special values of $g_{\epsilon_1}(n-t,r,X)$, certain linear combination of $g_{\epsilon_1}(n-t,r,X)$ is of a simple form (Lemma \ref{lem: sum of g}).
By a direct computation, we can check the following lemma.
\begin{lemma} \label{lem:h} For $0 \le s \le n-1$, let
\begin{equation}
    h_{\epsilon_1}(n,s,X)\coloneqq\begin{cases}
          \prod_{l=\frac{n+s+1}2}^{n-1} (q^{2l} X^2 -1) &\ff \text{ $n-s$ is odd},
          \\
          (q^{\frac{n+s}2}X-\epsilon \epsilon_1)\prod_{l=\frac{n+s+2}2}^{n-1} (q^{2l} X^2 -1) &\ff  \text{ $n-s$ is even}.
    \end{cases}
\end{equation}
Then
\begin{align}
  h_{ \epsilon_1}(n,j,qX)&= h_{\epsilon_1}(n+1,j+1,X),
 \\
    q^{\lfloor\frac{n+j+2}2\rfloor} X h_{ \epsilon_1}(n+1, j+1,X)&= h_{ \epsilon_1}(n+1, j,X) +(-1)^{n+j+1} \epsilon \epsilon_1
   h_{\epsilon_1}(n+1, j+1, X).
\end{align}
\end{lemma}

\begin{lemma}\label{lem: g(r,r,X)}
For integers $0<r\le n-1$ and $\epsilon, \epsilon_1 =\pm 1$, we have
$$
g_{\epsilon_1}(r,r,X)= (-1)^{r(n-1)}\epsilon_1^n \epsilon^r \alpha(r) h_{\epsilon_1}(n, r, X).
$$
In particular, $g_{\epsilon_1}(r, r, X)$ is a polynomial of degree $n-r-1$.
\end{lemma}
\begin{proof} We prove the case when $n$ case is odd and $r$ is even, and leave the other three cases to the reader. The idea is the same (a little bit more complicated). In this case,
we need to show
\begin{align}\label{eq: formula 3}
  \frac{g_{\epsilon_1}(r,r,X)}{h_{\epsilon_1}(n,r,X)}=  \sum_{k=0}^{r} (-1)^k q^{\frac{k(k-1)}{2}} \cdot
   \sum_{\epsilon_2 \in \{\pm 1\}} m(U_k^{\epsilon_2},U_{r}^{\epsilon_1})\cdot \alpha(r-k)\beta_{\delta}(r-k)\cdot  \frac{f_{\epsilon_2}(n,k,X)}{h_{\epsilon_1}(n,r,X)}=
       \epsilon_1   \alpha(r),
\end{align}
where $\delta=\delta(r,k,\epsilon_1,\epsilon_2)$.
Since $n$ is odd and $r$ is even, we have
\begin{align*}
    &\frac{f_{\epsilon_2}(n,k,X)}{h_{\epsilon_1}(n,r,X)} =
    \begin{cases}
          \prod_{\ell=\frac{n+1}{2}}^{\frac{n+k-2}{2}}(\q^{2\ell} X^2) \cdot   \q^{\frac{n+k}{2}}X( \q^{\frac{n+k}{2}}X-\epsilon \epsilon_2) \prod_{\ell =\frac{n+2+k}{2}}^{\frac{n+r-1}{2}}(\q^{2\ell}X^2-1)    & \text{ if $k$ is odd},\\
         \prod_{\ell=\frac{n+1}{2}}^{\frac{n-1+k}{2}}(q^{2\ell} X^2)\cdot   \prod_{\ell =\frac{n+1+k}{2}}^{\frac{n+r-1}{2}}(\q^{2\ell}X^2-1)    & \text{ if $k$ is even}.
    \end{cases}
\end{align*}
As a result, dividing \eqref{eq: formula 3} by $q^{(n+\frac{r-1}{2})(\frac{r-1}{2})+\frac{n+r}{2}}\cdot X^{r}$ and setting $Y=(q^{\frac{n+1}{2}}X)^{-1}, j=r-k$,   \eqref{eq: formula 3} is equivalent to
\begin{align}  \label{eq:tildeg}
  \tilde g_{\epsilon_1}(r,Y)\coloneqq  \sum_{j=0}^{r} (-1)^{r-j} q^{\frac{(r-j)(r-j-1)}{2}} \cdot \alpha(j) \cdot
   \sum_{\epsilon_2 \in \{\pm 1\}} m(U_{r-j}^{\epsilon_2},U_{r}^{\epsilon_1})\beta_{\delta}(j)\cdot  \tilde f_{\epsilon_2}(j,Y)=
       \epsilon_1     Y^r,
\end{align}
where  $ \tilde f_{\epsilon_2}(j,Y)$ is a polynomial of degree $j$ defined as follows:
\begin{align*}
  \tilde f_{\epsilon_2}(j,Y)\coloneqq
    \begin{cases}
       \prod_{\ell =\frac{r-j}{2}}^{\frac{r-2}{2}}(1-\q^{-2\ell}Y^2) & \text{ if $j$ is even},\\
          (1 -\epsilon \epsilon_2\q^{-\frac{r-j-1}{2}}Y) \cdot \prod_{\ell =\frac{r-j+1}{2}}^{\frac{r-2}{2}}(1-\q^{-2\ell}Y^2)& \text{ if $j$ is odd}.
    \end{cases}
\end{align*}

Since $\{\tilde f_{1}(j,Y), 0\le j \le r\}$ forms a basis of the space of polynomials with degree $\le r$, there exists unique tuples $(a_{j})$ and $(b_{j}) $ such that
\begin{align*}
f(Y)=\epsilon_1  Y^r =\sum_{j=0}^{r}a_{j}\tilde f_{1}(j,Y), \quad \text{and} \quad     \tilde g_{\epsilon_1}(r,Y) =\sum_{j=0}^{r}b_{j}\tilde f_{1}(j,Y).
\end{align*}
  We need to show $(a_{j})=(b_{j})$. It is easy to check
  $$
  a_r =b_r = (-1)^{\frac{r}2} \epsilon_1 \alpha(r).
  $$
  Now to prove $a_j=b_j$ for all $j$, it suffices to prove that both $a_j$ and $b_j$ satisfy the recursion formula  for $j<r$
  \begin{equation} \label{eq: rel for a m even}
a_j= \begin{cases}
  -q^j \frac{1-q^{-j-2}}{1-q^{j-r}} a_{j+2}  & \text{ if $j$ is even},
 \\
 0 & \text{ if $j$ is odd}.
\end{cases}
\end{equation}

We  start with  $a_j$. First of all, we have
\begin{align}   \label{eq10.22}
  \sum_{j=0}^{r}q^{-r}a_{j} \tilde  f_{1}(j,Y)=q^{-r}f(Y)=f(q^{-1}Y)= \sum_{j=0}^{r} a_{j} \tilde f_{1}(j,q^{-1}Y).
\end{align}
Notice that
\begin{equation*}
    \tilde f_{1}(j,q^{-1}Y)=(1-q^{-r}Y^2)  \tilde f_{1}(j-2,Y).
\end{equation*}
Since
\begin{align*}
    \tilde{f}_{1}(j+1,Y)/  \tilde{f}_{1}(j,Y)=
    \begin{cases}
       1 -\epsilon \q^{-\frac{r-j-2}{2}}Y  & \text{ if $j$ is even},\\
    1 +\epsilon  \q^{-\frac{r-j-1}{2}}Y    & \text{ if $j$ is odd},
    \end{cases}
\end{align*}
we have
\begin{align*}
    Y  \tilde{f}_{1}(j,Y)&=\begin{cases}
          -\epsilon \q^{\frac{r-j-2}{2}}( \tilde{f}_{1}(j+1,Y)- \tilde{f}_{1}(j,Y)) & \text{ if $j$ is even},\\
     \epsilon \q^{\frac{r-j-1}{2}}( \tilde{f}_{1}(j+1,Y)- \tilde{f}_{1}(j,Y))     & \text{ if $j$ is odd},
    \end{cases}
    \\
    Y^2  \tilde{f}_{1}(j,Y)&=\begin{cases}
         -q^{r-j-2}(  \tilde{f}_{1} (j+2,Y)- \tilde{f}_{1} (j,Y))& \text{ if $j$ is even},\\
    -q^{r-j-2}(  \tilde{f}_{1}(j+2,Y)+(q-1)  \tilde{f}_{1}(j+1,Y)-q  \tilde{f}_{1}(j,Y) ) & \text{ if $j$ is odd}.
    \end{cases}
\end{align*}
Therefore,
\begin{align*}
    & \tilde{f}_{1}(j,q^{-1}Y)
    =\begin{cases}
             q^{-j}  \tilde{f}_{1}(j,Y)+(1-q^{-j}) \tilde{f}_{1}(j-2,Y)& \text{ if $j$ is even},\\
    q^{-j}  \tilde{f}_{1}(j,Y)+(q-1)q^{-j}  \tilde{f}_{1}(j-1,Y)+(1-q^{1-j})  \tilde{f}_{1}(j-2,Y)& \text{ if $j$ is odd}.
    \end{cases}
\end{align*}
Plugging this into (\ref{eq10.22}), we obtain
\begin{align*}
    (q^{-r}-q^{-j})a_{j}=
    \begin{cases}
 (q-1)q^{-j-1}a_{j+1}+(1-q^{-j-2})a_{j+2}   & \text{ if $j$ is even},\\
  (1-q^{-j-1})a_{j+2}    & \text{ if $j$ is odd},
    \end{cases}
    \end{align*}
with $a_{r+1} =a_{r+2} =0$. So we have (\ref{eq: rel for a m even}). We remark that in other cases, we have similar recursion formula as above but could not be simplified like \eqref{eq: rel for a m even}.

Now we compute $b_j$ for $j<r$.
Recall that  $r$ is even. First, if $j=0$, we have
\begin{align*}
     \sum_{\epsilon_2 \in \{\pm 1\}} m(U_{r-j}^{\epsilon_2},U_{r}^{\epsilon_1})\beta_{\delta}(j)\cdot  \tilde f_{\epsilon_2}(j,Y)=   \tilde f_{\epsilon_1}(0,Y).
\end{align*}
It is easy to check that
\begin{equation}
\tilde f_{\epsilon_2}(j, Y)
=\begin{cases}
 \tilde f_{1}(j, Y) &  \text{ if $j$ is even},
 \\
 (1-\epsilon \epsilon_2 q^{-\frac{r-j-1}2}) \tilde f_1(j-1, Y) & \text{ if $j$ is odd},
\end{cases}
\end{equation}
Now Lemmas \ref{lem:quotientm}  and \ref {lem: m 0_j sum} imply for $j\neq 0$ and $\delta=\delta(r, j, \epsilon_1, \epsilon_2)$, we have
\begin{align*}
&\sum_{\epsilon_2 \in \{\pm 1\}} m(U_{r-j}^{\epsilon_2},U_{r}^{\epsilon_1})\cdot \beta_{\delta}(j)\cdot  \tilde  f_{\epsilon_2}(j,Y)
=\begin{cases}
 \frac{2\epsilon_1 (-1)^{\frac{j}2}(\epsilon_1 q^{-\frac{j}2}+q^{-\frac{r-j}2})   }{(1+\epsilon_1 q^{-\frac{j}2})(1+q^{-\frac{r-j}2})} m(U_{r-j}^1, U_r^{\epsilon_1}) \tilde f_1(j, Y) &  \text{ if $j$ is even},\\
  2 (-1)^{\frac{j-1}2} m(U_{r-j}^1, U_r^{\epsilon_1}) \tilde f_1(j-1, Y) & \text{ if $j$ is odd}.
\end{cases}
\end{align*}
 Plugging this into the definition of $\tilde g_{\epsilon_1}(r,  X)$ as in (\ref{eq:tildeg}), we obtain
 $$
 \tilde g_{\epsilon_1}(r,  X) =\sum_{j=0}^r b_j \tilde f_1(j, Y)
 $$
with  $b_j=0$ for odd $j$, $b_0=q^{\frac{r(r-1)}2}(1-2 q^{-(r-1)}m(U_1^1,U_r^{\epsilon_1})),$ and
\begin{align*}
b_j=&2(-1)^{r-\frac{j}{2}} q^{\frac{(r-j)(r-j-1)}2}   \frac{( q^{-\frac{j}2}+\epsilon_1  q^{-\frac{r-j}2})   }{(1+\epsilon_1 q^{-\frac{j}2})(1+q^{-\frac{r-j}2})} m(U_{r-j}^1, U_r^{\epsilon_1})  \\
&- 2(-1)^{r-\frac{j}{2}} q^{\frac{(r-j-1)(r-j-2)}2 } \alpha(j+1) m(U_{r-j-1}^1, U_r^{\epsilon_1})
\end{align*}
for even $j\neq 0$.
Applying Lemma \ref{lem:quotientm}, for even $j\neq 0$, we have
\begin{align*}
b_j&=
2(-1)^{r-\frac{j}2} q^{\frac{(r-j)(r-j-1)}2}
\frac{q^{-\frac{r-j}2}(\epsilon_1 + q^{-\frac{r}2})}{(1+\epsilon_1 q^{-\frac{j}2})(1+q^{-\frac{r-j}2})}
\alpha(j) m(U_{r-j}^1, U_r^{\epsilon_1}).
\end{align*}
Now applying Lemma \ref{lem:quotientm} twice, we can check that $b_j$ satisfy (\ref{eq: rel for a m even}).
\end{proof}

Recall that $\m(n,t)\coloneqq \mathrm{min}\{n-t,n-1\}.$
\begin{lemma}\label{lem: sum of g}
For $0\le t\le n$, and $m=n-t$,  we have
  \begin{align}\label{eq: goal 1}
     \sum_{r=0}^{\mathrm{m}(n,t)}(-1)^r q^{\frac{(n-r)(n-r-1)}{2}} q^{rt} \cdot g_{\epsilon_1}(m,r,X)=F_{\epsilon_1}(n,m,X),
\end{align}
where
\begin{align*}
    F_{\epsilon_1}(n,m,X)=
    \begin{cases}
      q^{\frac{n(n-1)}{2}}  f_{\epsilon_1}(n,m,X) & \text{ if $t\not=0$},\\
      (-1)^{n-1}\alpha(n)\sum_{\ell=0}^{n-1}(-q^n X)^\ell & \text{ if $t=0$}.
    \end{cases}
\end{align*}
\end{lemma}
\begin{proof}
We treat the case $t=0$ first. In this case, $\epsilon_1=\epsilon$.  Before we give the details of the proof, we summarize the main idea.
Since $\{h_{\epsilon}(n,s,X), 0\le s \le n-1\}$ forms a basis of the space of polynomials with degree $\le n-1$, there exists unique tuples $(a_{n,j})$ and  $(b_{n,j}) \in \mathbb{Q}^n$  such that
\begin{align} \label{eq:basish}
 F_{\epsilon}(n, n, X)=\sum_{j=0}^{n-1}a_{n,j}h_{\epsilon}(n, j, X), \quad \text{and} \quad    \sum_{r=0}^{n-1}(-1)^r q^{\frac{(n-r)(n-r-1)}{2}}  \cdot g_{\epsilon}(n,r,X) =\sum_{j=0}^{n-1}b_{n,j}h_{\epsilon}(n,j,X).
\end{align}

We need to show $a_{n,j}=b_{n,j}$ for all $n$ and $j$.  We first show that  $a_{n,j}$ satisfy the  recursion relations \eqref{eq:recusiona}, which gives a description of $a_{n+1,j}$ in terms of $a_{n,j}$ and $a_{n,j-1}$. We can directly check $a_{1,0}=b_{1,0}=1$. Then by an induction on $n$, it suffices to show   $b_{n,j}$ also satisfies \eqref{eq:recusiona}.

Now we derive \eqref{eq:recusiona}.
It is easy to check that
\begin{equation}
F_{\epsilon}(n+1, n+1, X)
 = q^{\lfloor\frac{n}{2}\rfloor} (q^{n+1}X) F_\epsilon(n , n, qX) +(-1)^n\alpha(n+1).
\end{equation}
Plugging (\ref{eq:basish}) into the above formula and applying Lemma \ref{lem:h}, we obtain
\begin{equation} \label{eq:recursionh}
 \sum_{j=0}^{n}a_{n+1,j}h_{\epsilon}(n+1, j, X)=q^{\lfloor\frac{n}{2}\rfloor}(q^{n+1}X) \sum_{j=0}^{n-1}a_{n,j}h_{\epsilon}(n+1, j+1,X)+(-1)^n \alpha(n+1) h_\epsilon(n, n-1, X).
\end{equation}
Let
\begin{equation} \label{eq:gamma}
\gamma(n, j)=\lfloor\frac{n}2\rfloor+n+1-\lfloor\frac{n+j+2}2\rfloor
=\begin{cases}
  n-\frac{j}2 &\ff \,  j \hbox{ is even},
  \\
  n-\lfloor\frac{j-(-1)^n}2\rfloor &\ff \,  j \hbox { is odd}.
\end{cases}
\end{equation}
Then Lemma \ref{lem:h} and (\ref{eq:recursionh}) imply
\begin{align*}
 \sum_{j=0}^{n}a_{n+1,j}h_{\epsilon}(n+1, j, X)
 &=\sum_{j=0}^{n-1} a_{n, j} q^{\gamma(n, j)} h_\epsilon(n+1, j, X) +\sum_{j=1}^n (-1)^{n+j} a_{n, j-1}q^{\gamma(n, j-1)} h_\epsilon(n+1, j,X)
 \\
 &\quad
 +(-1)^n \alpha(n+1) h_\epsilon(n+1, n, X).
 \end{align*}
That is
\begin{equation} \label{eq:recusiona}
 \begin{cases}
   a_{n+1, 0}=q^n a_{n, 0},
   \\
   a_{n+1, j}=q^{\gamma(n, j)}a_{n, j}+(-1)^{n+j} q^{\gamma(n, j-1)} a_{n, j-1}, \quad  0< j <n,
   \\
   a_{n+1, n}=q^{\gamma(n, n-1)} a_{n, n-1}  + (-1)^n\alpha(n+1).
 \end{cases}
\end{equation}

Now we compute $b_{n,j}$. A direct computation shows that $b_{n,0}=q^{\frac{n(n-1)}2}$. In the following, we compute $b_j$ for $j\neq 0$.

For $r=0$, we have $$g_\epsilon(n,0,X)=f_1(n,0,X)=\begin{cases}
h_\epsilon(n,0,X) & \text{ if $n$ is odd},\\
h_\epsilon(n,0,X)+(1-\epsilon)h_\epsilon(n,1,X) & \text{ if $n$ is even}.
\end{cases}$$
Now we assume $r\neq 0$.
Recall that by Lemma \ref{lem: ind of g}, we have \begin{align*}
     g_\epsilon(n,r,  X)&=\sum_{\epsilon_3 \in \{\pm 1\}}  m(U_r^{\epsilon_3},U_{n}^{\epsilon })    g_{\epsilon_3}(r,r,X).
\end{align*}
Notice that when $n-r$ is odd, $h_{\epsilon_1}(n, r, X)$ is independent of $\epsilon_1$. Then a direct calculation using Lemma \ref{lem: g(r,r,X)} and the formula for $m(U_r^{\epsilon_1}, U_n^\epsilon)$ gives
\begin{align} \label{eq:bodd}
    g_\epsilon(n,r,  X)
  =\alpha(r)h_\epsilon(n, r, X) m(U_r^\epsilon, U_n^\epsilon)
    \begin{cases}
      \frac{ 2q^{-\frac{r}{2}}}{1+\epsilon q^{-\frac{r}2}}
         &\ff \, n  \equiv r-1 \equiv 1 \pmod 2,\\
          -2\epsilon &\ff \, n \equiv r-1  \equiv 0 \pmod 2.
      \end{cases}
\end{align}
When $n-r$ is even, we have
\begin{equation} \label{eq:h}
h_{\epsilon_3}(n, r, X) = (q^{\frac{n+r}2} X -\epsilon \epsilon_3) h_\epsilon(n, r+1, X).
\end{equation}
So a direct calculation gives
\begin{equation}
g_\epsilon(n, r, X)
= \alpha(r) h_\epsilon(n, r+1, X) m(U_r^{\epsilon}, U_n^\epsilon)
\left( q^{\frac{n+r}2} X -1 + (-1)^n (q^{\frac{n+r}2}X +1) \frac{m(U_r^{-\epsilon}, U_n^\epsilon)}{m(U_r^{\epsilon} , U_n^\epsilon)}\right).
\end{equation}
We have by Lemma \ref{lem: counting subspace}
$$
 \frac{m(U_r^{-\epsilon}, U_n^\epsilon)}{m(U_r^\epsilon, U_n^\epsilon)}
 =\frac{1- q^{-\frac{n-r}2}}{1+  q^{-\frac{n-r}2}}
 \begin{cases}
 1 &\ff \,  n \equiv r \equiv 1 \pmod 2,
 \\
 \frac{1- \epsilon q^{-\frac{r}2}}{1+  \epsilon q^{-\frac{r}2}} &\ff \, n \equiv r \equiv 0 \pmod 2.
 \end{cases}
$$
The equation (\ref{eq:h}) gives
$$
X h_\epsilon(n, r+1, X) = q^{-\frac{n+r}2} ( h_\epsilon(n, r, X) + h_\epsilon(n, r+1, X)).
$$
So when $ n\equiv r \equiv 1 \pmod 2$, we have
$$
g_\epsilon( n,r, X)
=
2   \alpha(r) m(U_r^\epsilon, U_n^\epsilon) \left( \frac{q^{-\frac{n-r}2}}{1+  q^{-\frac{n-r}2}} h_\epsilon(n, r, X) +\frac{q^{-\frac{n-r}2}-1}{1+  q^{-\frac{n-r}2}}  h_\epsilon(n, r+1, X) \right) ,
$$
 and  when $ n\equiv r \equiv 0 \pmod 2$ and $r \neq 0$, we have
\begin{align*}
&g_\epsilon( n,r, X)
\\
&=\frac{2  \alpha(r) m(U_r^\epsilon, U_n^\epsilon)}{(1+ q^{-\frac{n-r}2})(1+\epsilon q^{-\frac{r}2} )}\left( (1+\epsilon q^{-\frac{n}2}) h_\epsilon(n, r, X) + (1- q^{-\frac{n-r}2})(1 - \epsilon q^{-\frac{r}2}) h_\epsilon(n, r+1, X)\right).
\end{align*}
In  summary, we have the numbers $b_{n,j}$ for $j\neq 0$ are given by the following.\newline
If $n$ and $j$ are odd, then
\begin{align*}
    b_{n, j}=  2(-1)^j  q^{\frac{(n-j)(n-j-1)}2} \alpha(j) \frac{m(U_j^\epsilon, U_n^\epsilon)}{1+q^{\frac{n-j}2}}.
\end{align*}
If $n$ is odd and $j$ is even, then
\begin{align*}
     b_{n, j}= 2(-1)^j q^{\frac{(n-j)(n-j-1)}2} \left(  \frac{ \alpha(j) q^{-\frac{j}2} }{1+\epsilon q^{-\frac{j}2}}m(U_j^\epsilon, U_n^\epsilon) - \frac{q^{n-j}\alpha(j-1) (q^{-\frac{n-j+1}2}-1)}{1+q^{-\frac{n-j+1}2}} m(U_{j-1}^\epsilon, U_n^\epsilon) \right).
\end{align*}
If $n$ and $j$ are even, then
\begin{align*}
      b_{n, j} =  \frac{2(-1)^j  q^{\frac{(n-j)(n-j-1)}2}\alpha(j)(1+\epsilon q^{-\frac{n}2}) }{(1+q^{-\frac{n-j}2}) (1 + \epsilon q^{-\frac{j}2})}m(U_j^\epsilon, U_n^\epsilon).
\end{align*}
If $n$ is even, then for $j=1$ we have
\begin{align*}
  b_{n,1}&=  q^{\frac{n(n-1)}2}(1-\epsilon+2\epsilon q^{-(n-1)})m(U_1^1,U_n^\epsilon),
  \end{align*}
and for  odd $j>1$, $b_{n, j}$ is equal to
  \begin{align*}
      &2(-1)^{j+1} q^{\frac{(n-j)(n-j-1)}2}
  \left(\epsilon\alpha(j) m(U_j^\epsilon, U_n^\epsilon)
  + \frac{q^{n-j}\alpha(j-1) (1-q^{-\frac{n-j+1}2})(1-\epsilon q^{-\frac{j-1}2}) }{(1+q^{-\frac{n-j+1}2}) (1 + \epsilon q^{-\frac{j-1}2})}m(U_{j-1}^\epsilon, U_n^\epsilon)\right).
\end{align*}
Using the explicit formulas, direct calculation shows that $b_{n, j}$ satisfies
(\ref{eq:recusiona}).

From now, we assume $t\not =0$ and let $m=n-t$.  The proof is essentially the same as the proof of Lemma \ref{lem: g(r,r,X)} and we only prove the case that $n$ is odd and $m$ is even in detail.
According to Lemma \ref{lem: ind of g}, we have
\begin{align*}
    &\sum_{r=0}^{\m(n,t)}(-1)^r q^{\frac{(n-r)(n-r-1)}{2}} q^{rt} \cdot g_{\epsilon_1}(m,r,X)\\
    &=\sum_{r=0}^{\m(n,t)}(-1)^r q^{\frac{(n-r)(n-r-1)}{2}} q^{rt} \cdot \sum_{\epsilon_3 \in \{\pm 1\}}m(U_r^{\epsilon_3},U_{m}^{\epsilon_1})g_{\epsilon_3}(r,r,X)\\
    &=  q^{\frac{(n-m)(n+m-1)}{2}}\sum_{r=0}^{\m(n,t)}(-1)^r q^{\frac{(m-r)(m-r-1)}{2}}  \cdot \sum_{\epsilon_3 \in \{\pm 1\}}  m(U_r^{\epsilon_3},U_{m}^{\epsilon_1})   g_{\epsilon_3}(r,r,X).
\end{align*}
Assume that $n$ is odd and $m$ is even.  Factoring out $h_{\epsilon_1}(n,m,X)$, replacing $X$ by $q^{\frac{n-1}{2}} X$, and apply Lemma \ref{lem: g(r,r,X)}, we have that \eqref{eq: goal 1} is equivalent to
\begin{align}\label{eq: eq t not 0}
       q^{\frac{(n-m)(n+m-1)}{2}}\sum_{r=0}^{m}(-1)^r q^{\frac{(m-r)(m-r-1)}{2}}  \cdot \sum_{\epsilon_3 \in \{\pm 1\}}  m(U_r^{\epsilon_3},U_{m}^{\epsilon_1})   \alpha(r) g_{\epsilon_3 }'(m,r,X) =  F '(n,m,X),
\end{align}
where
\begin{align*}
  g_{\epsilon_3 }'(m,r,X)=
    \begin{cases}
      \epsilon \epsilon_3     ( \q^{\frac{r+1}{2}}X-\epsilon \epsilon_3) \cdot  \prod_{\ell = \frac{r+3}{2} }^{\frac{m}{2}}(\q^{2\ell}X^2-1)  & \text{ if   $r$ is odd},\\
      \epsilon_3
       \prod_{\ell=\frac{r+2}{2}}^{\frac{m}{2}}(\q^{2\ell}X^2-1)   & \text{ if  $r$ is even},
    \end{cases}
\end{align*}
and
\begin{align*}
  F '(n,m,X) =
      q^{\frac{n(n-1)}{2}}  \prod_{\ell=1}^{\frac{m}{2}}(q^{2\ell} X^2).
\end{align*}
Since $g_{\epsilon}'(r,m,X)$ forms a basis of the space of polynomials with degree $\le m$, there exists unique tuples $(a_j)$ and $(b_j)$ such that
\begin{align*}
     \text{LHS of \eqref{eq: eq t not 0}}=\sum_{j=0}^{m}a_j g_{\epsilon}'(m, j,X), \quad
    \text{RHS of \eqref{eq: eq t not 0}}
   =\sum_{j=0}^{m}b_j g_{\epsilon}'(m, j,X).
\end{align*}
It suffices to show that $a_0=b_0$ and  both $(a_j)$ and $(b_j)$ satisfy the following recursive relation for $0<r\le m$:
\begin{align} \label{eq: new recur tneq 0}
  a_r=\begin{cases}
0  &\text{ if   $r$ is odd,}\\
  \frac{q^{m-r+2}-1}{q^m-q^{m-r}}a_{r-2} & \text{ if  $r$ is even}.
\end{cases}
\end{align}

We derive the recursive relation for $a_j$ first.
Notice that
\begin{align*}
     g_{\epsilon }'(m,r,q X)=(q^{m+2}X^2-1)g_{\epsilon }'(m,r+2, X).
\end{align*}
Moreover,
\begin{align*}
    Xg_{\epsilon }'(m,r, X)&=
    \begin{cases}
    q^{-\frac{r+1}{2}}(\epsilon g_{\epsilon}'(m,r-1, X)-g_{\epsilon}'(m, r,X))& \text{ if   $r$ is odd,}\\
        q^{-\frac{r}{2}}(\epsilon g_{\epsilon}'(m,r-1, X)+g_{\epsilon}'(m,r, X))  & \text{ if  $r$ is even.}
    \end{cases}\\
    X^2g_{\epsilon}'(m,r, X)&=
    \begin{cases}
  q^{-(r+1)}(q g_{\epsilon}'(m,r-2, X) +\epsilon (q-1) g_{\epsilon}'(m,r-1, X)+g_{\epsilon}'(m,r, X))& \text{ if   $r$ is odd,}\\
     q^{-r}(g_{\epsilon}'(m,r-2, X)+g_{\epsilon}'(m,r, X))  & \text{ if  $r$ is even.}
    \end{cases}
\end{align*}
Hence if $m$ is even, then
\begin{align*}
     &g_{\epsilon}'(m,r,q X)\\
     &= \begin{cases}
 q^{m-r} g_{\epsilon}'(m,r, X)+\epsilon(q-1)q^{m-r-1} g_{\epsilon}'(m, r+1,X)+(q^{m-r-1}-1)g_{\epsilon}'(m, r+2,X) & \text{ if   $r$ is odd,}\\
     q^{m-r} g_{\epsilon}'(m,r, X) +(q^{m-r}-1)g_{\epsilon}'(m,r+2, X)  & \text{ if  $r$ is even.}
    \end{cases}
\end{align*}
Hence for $1<r\le m+1$
\begin{align}\label{eq: recursive t not 0}
(q^m-q^{m-r}) a_r=\begin{cases}
(q^{m-r+1}-1)a_{r-2}  &\text{ if   $r$ is odd,}\\
 \epsilon (q-1)q^{m-r}a_{r-1} +(q^{m-r+2}-1)a_{r-2} & \text{ if  $r$ is even},
\end{cases}
\end{align}
with $a_{m+1}=0$, which implies \eqref{eq: new recur tneq 0}.

Now we compute $b_j$.
First, when $r=0$,
\begin{align*}
     \sum_{\epsilon_3 \in \{\pm 1\}}  m(U_r^{\epsilon_3},U_{m}^{\epsilon_1})  g_{\epsilon_3}'(m,0,X)=\epsilon g_\epsilon'(m,0,X).
\end{align*}

For $r\neq 0$,
\begin{align*}
  \sum_{\epsilon_3 \in \{\pm 1\}}  m(U_r^{\epsilon_3},U_{m}^{\epsilon_1})  g_{\epsilon_3}'(m,r,X)=
  \begin{cases}
 -2\epsilon m(U_r^{\epsilon},U_{m}^{\epsilon_1}) g_{\epsilon}'(m,r+1, X)  &\text{ if   $r$ is odd,}\\
   \epsilon( m(U_r^{1},U_{m}^{\epsilon_1})-m(U_r^{-1},U_{m}^{\epsilon_1}) ) g_{\epsilon}'(m,r, X) & \text{ if  $r$ is even}.
\end{cases}
\end{align*}
Then
\begin{align*}
    b_r=
    \begin{cases}
 0  &\text{ if   $r$ is odd,}\\
 a_{r-1}'+a_r'  & \text{ if  $r$ is even.}
\end{cases}
\end{align*}
where
\begin{align*}
      a_r'=  q^{\frac{(n-m)(n+m-1)}{2}}
    \begin{cases}
 -2(-1)^{r}q^{\frac{(m-r)(m-r-1)}{2}} \alpha(r) \epsilon m(U_{r-1}^{\epsilon},U_{m}^{\epsilon_1})   &\text{ if   $r$ is odd,}\\
   (-1)^{r} \alpha(r) q^{\frac{(m-r)(m-r-1)}{2}} \epsilon( m(U_r^{1},U_{m}^{\epsilon_1})-m(U_r^{-1},U_{m}^{\epsilon_1}) ) & \text{ if  $r$ is even.}
\end{cases}
\end{align*}
Finally, a direct calculation shows that $a_0=b_0=\epsilon q^{\frac{n(n-1)}2}$ and $b_r$ satisfies \eqref{eq: recursive t not 0}.
\end{proof}

\subsection{Proof of  Theorem \ref{thm: formula of ppden}}  \label{proof}
Now we are ready to prove  Theorem \ref{thm: formula of ppden}.

Recall that
\begin{align*}
    \Pden'(L)=\frac{\Pden'(I_n,L)}{\Pden(I_n ,I_n )},
\end{align*}
and
\begin{align*}
    \Pden(I_{n},I_{n})=2\alpha(n)I_{\epsilon}(n),
\end{align*}
where $I_\epsilon(n)$ is defined in Lemma \ref{lem: tran to poly}.

We first assume that  $L=H^j\obot L_1$ with $j>0$ where $L_1=I_{n_1 -t} \obot L_2$ is an integral lattice of rank $n_1$ and  type $t>0$ (the other non-integral cases were taken  care of in the summary of the proof at the beginning of this section). By Lemma \ref{lem: L=H^i obot L_2}, we have
\begin{align*}
    \Pden'(I_n,L)=2\Big(\prod_{\ell=1}^{j-1}(1-\q^{2\ell} )\Big)\Pden(I_n,L_1,\q^{2j}).
\end{align*}
It suffices to show $\Pden(I_n,L_1,\q^{2j})=0$. By Theorem \ref{thm: decom of Pden} and Lemma \ref{lem: Pden' to g}, we have
\begin{align}\label{eq: non integral pf of main}  \notag
  \Pden(I_{n},L,q^{2j})&=\sum_{i=0}^{n_1}\Pden^{n_1-i} (I_{n}^{-\epsilon},L,q^{2j}) \\
  &=I_{\epsilon}(n,n_1)  \sum_{r=0}^{n_1-t}\sum_{i=0}^{n_1} \binom{n_1-r}{i-r}_q (-1)^{n_1+i-r} q^{\frac{(i-r)(i-r-1)}{2}}  q^{rt} \cdot g_{\epsilon_1}(n_1,n_1-t,r,q^{j-i}).
\end{align}
Notice that the assumption $t>0$ implies that $r\le n_1-t\le n_1-1$. Hence $g_{\epsilon_1}(n_1,n-t,r,X)$ is a polynomial of degree $n_1-r-1$ by Lemma \ref{lem: g(r,r,X)}. Then we may apply Corollary \ref{cor: qbinomial vanish} to conclude
\begin{align*}
    \sum_{i=0}^{n_1} \binom{n_1-r}{i-r}_q (-1)^{n_1+i-r} q^{\frac{(i-r)(i-r-1)}{2}}  q^{rt} \cdot g_{\epsilon_1}(n_1,n_1-t,r,q^{\frac{n-n_1}2-i})=0.
\end{align*}
Hence
$$
 \Pden'(L) = \frac{2\Big(\prod_{\ell=1}^{j-1}(1-\q^{2\ell} )\Big)}{\Pden(I_n, I_n)}\Pden(I_{n},L,q^{n-n_1})=0.
$$

Next, we assume $L=I_{n-t} \obot L_2 $ is integral of rank $n$ and type $t$ (Cases (2) and (3) or equivalently Case (b) in the summary of the proof at the beginning of this section). Similarly, by Theorem \ref{thm: decom of Pden} and Lemma \ref{lem: Pden' to g}, we have
\begin{align} \label{eq:PdenD}
    \Pden'(I_n,L)  &= \sum_{i=0}^{n-1}(\Pden^{n-i})'(I_n,L)  \notag\\
    &=2 I_\epsilon(n)  \sum_{r=0}^{\m(n,t)}\sum_{i=0}^{n-1} \binom{n-r}{i-r}_q   (-1)^{n-1+i+r} q^{\frac{(i-r)(i-r-1)}{2}}  q^{rt} \cdot g_{\epsilon_1}(n,n-t,r,q^{-i}).
\end{align}
Here, recall that $\m(n,t)\coloneqq \mathrm{min}\{n-t,n-1\}$. Applying Corollary \ref{cor: qbinomial vanish} as before, we have
\begin{align*}
    \sum_{i=0}^{n}\binom{n-r}{i-r}_q (-1)^{r }(-1)^{n-1+i} q^{\frac{(i-r)(i-r-1)}{2}}  q^{rt} \cdot g_{\epsilon_1}(n,n-t,r,q^{-i})=0.
\end{align*}
Hence,
\begin{align*}
   \Pden'(I_n,L) =2 I_\epsilon(n) \sum_{r=0}^{\m(n,t)} (-1)^r q^{\frac{(n-r)(n-r-1)}{2}}  q^{rt} \cdot g_{\epsilon_1}(n,n-t,r,q^{-n}).
\end{align*}
By Lemma \ref{lem: sum of g}, if $t\not =0$, then
\begin{align*}
   \sum_{r=0}^{\m(n,t)} (-1)^r q^{\frac{(n-r)(n-r-1)}{2}}  q^{rt} \cdot g_{\epsilon_1}(n,n-t,r,q^{-n})
   &=F_{\epsilon_1}(n,n-t,q^{-n})\\
    &=\alpha(n)\cdot\begin{cases}
	 \prod_{\ell=1}^{\frac{t-1}{2}}(1-\q^{2\ell}) & \text{  if $t$ is odd},\\
		(1-\epsilon \epsilon_1 \q^{\frac{t}{2}})\prod_{\ell=1}^{\frac{t}{2}-1}(1-q^{2\ell}) & \text{ if  $t$ is even}.
			\end{cases}
\end{align*}
Notice that if $t$ is even, then $\chi(I_{n-t}^{\epsilon_1})\chi(L_2)=\epsilon$. Hence $\epsilon \epsilon_1 =\chi(L_2)$.

If $t=0$, then by Lemma \ref{lem: sum of g},
\begin{align*}
    \sum_{r=0}^{\m(n,t)} (-1)^r q^{\frac{(n-r)(n-r-1)}{2}}  q^{rt} \cdot g_{\epsilon_1}(n,n,r,q^{-n})
    =F_{\epsilon_1}(n,n,q^{-n}) =\alpha(n)\cdot \begin{cases}
1 & \text{  if $n$ is odd},\\
	0 & \text{ if  $n$ is even}.
			\end{cases}		
\end{align*}
This proves the theorem.

\section{Fourier transform: the analytic side}\label{sec: partial FT}
In this section, we study the partial Fourier transform of the vertical part of the analytic side following Section 8 of \cite{LZ2}.  The main result is Theorem \ref{prop: part FT of denLflat}.

\begin{definition}\label{def: ver of pden}
For a non-degenerate lattice $L^\flat\subset \bV$ of rank $n-1$, and $x\in \bV\setminus L_{F}^\flat$,  we define
\begin{align*}
		\pden_{L^{\flat},\mathscr{V}}(x)=\sum_{\substack{L^\flat \subset L' \subset L'^\sharp\\ L'^\flat \not\in \Hor(L^\flat)}}\ppden(L') 1_{L'}(x),
	\end{align*}
	where $L'^\flat=L'\cap L_F^\flat.$
\end{definition}

	\begin{theorem}\label{prop: part FT of denLflat}
Let $L^\flat\subset \bV$ be a non-degenerate lattice of rank $n-1$, and let $\bW=(L^\flat_F)^\bot$ be the perpendicular space of $L^\flat_F$ in $\bV$.  Recall the partial Fourier transform
		\begin{align*}
		\pden^{\perp}_{L^\flat,\mathscr V}(x)\coloneqq\int_{L^\flat_F} \pden_{L^\flat,\mathscr V}(x+y)\,dy, \quad x \in \bW \setminus\{ 0\}.
		\end{align*}
  Then $\pden^{\perp}_{L^\flat,\mathscr V}(x)$ is constant on $\bW^{\ge 0}\setminus \{0\}$ and is zero for $x \in  \bW^{<0}$.
	\end{theorem}
	
	\begin{proof}
It suffices to show that if $\val(x)> 0$, then
		\begin{align*}
			\pden^{\perp}_{L^\flat,\mathscr V}(x)-\pden^{\perp}_{L^\flat,\mathscr V}(\pi^{-1}x)=0.
		\end{align*}
By definition, we have
		\begin{align*}
			\pden^{\perp}_{L^\flat,\mathscr V}(x)
		&=\int_{L^\flat_F}\sum_{\substack{L^\flat\subset L' \subset L'^{\sharp}\\ L'^{\flat}\not \in \Hor(L^\flat)}} \ppden(L')1_{L'}(x+y)\,dy,
		\end{align*}
where $L'$ runs over lattices of rank $n$ in $L_F^\flat+\langle x\rangle$.

Recall that $\mathrm{Pr}_{L_F^\flat}$ denotes the projection to $L_F^\flat$. We rewrite the summation based on $L'\cap L_F^\flat$ and $\mathrm{Pr}_{L_F^\flat}(L')$. For lattices $L'^{\flat}\subset \widetilde{L}'^{\flat}$ in $L_F^\flat$  of rank $n-1$,	let
		\begin{align*}
		\mathrm{Lat}(L'^\flat,\widetilde{L}'^\flat)\coloneqq\{L'\subset \bV \mid L'\cap L^\flat_F=L'^\flat, \quad  \mathrm{Pr}_{L_F^\flat}(L')=\widetilde{L}'^\flat\}.
		\end{align*}
Then by Lemmas $7.2.1$ and $7.2.2$ of \cite{LZ}, we have
	
		\begin{align*}
		 \pden^{\perp}_{L^\flat,\mathscr V}(x) =\sum_{\substack{L^\flat\subset L'^\flat \\ L'^{\flat}\not \in \Hor(L^\flat)}}\sum_{\substack{L'^\flat\subset\widetilde{L}'^\flat\\\text{$\widetilde{L}'^\flat/L'^\flat$ cyclic}}}\sum_{L'\in \mathrm{Lat}(L'^\flat,\widetilde{L}'^\flat)} \ppden(L')\int_{L^\flat_F} 1_{L'}(x+y)\,dy.
		\end{align*}
Here we can switch the order of the sum and integral because there are only finitely many nonzero terms in the sum for a fixed $x$. Since $\tilde{L}^{\prime \flat}/L^{\prime \flat} $ is cyclic, it has a generator $u^\flat \in L_F^\flat$. Moreover,   for $L' \in  \mathrm{Lat}(L'^\flat,\widetilde{L}'^\flat)$, we can write $L'=L^{\prime \flat} + \langle u \rangle$ with $u=u^\flat+u^\perp \in \bV$ where $0\ne u^\perp \in \bW$.
Moreover,
	write $x=\alpha u^\perp$ with $\alpha \in F^{\times}$, then
		$$x+y =\alpha u + (y-\alpha u^\flat) \in L'$$ if and only if $\alpha \in O_F$ and $y-\alpha u^\flat\in L'^\flat$. As a result, we have
		\begin{align*}
			\int_{L^\flat_F}1_{L'}(x+y)-1_{L'}(\pi^{-1}x+y)dy=\begin{cases}
				\mathrm{vol}(L'^\flat), & \text{ if $\langle x \rangle = \langle u^\perp\rangle$},\\
				0, & \text{ otherwise}.
			\end{cases}
		\end{align*}
		Therefore, we have
		\begin{align} \label{eq:Diff}
			\pden^{\perp}_{L^\flat,\mathscr V}(x)-\pden^{\perp}_{L^\flat,\mathscr V}(\pi^{-1}x)=\sum_{\substack{L^\flat\subset L'^\flat \\ L'^{\flat}\not \in \Hor(L^\flat)}}\mathrm{vol}(L'^\flat)D(L^{\prime \flat})(x),
		\end{align}
		where
		\begin{align}\label{eq: D}
	D(L^{\prime \flat})(x)&=\sum_{\substack{L'^\flat\subset\widetilde{L}'^\flat \\\text{$\widetilde{L}'^\flat/L'^\flat$ cyclic}}}\sum_{\substack{u^\perp\in \langle x\rangle \text{ generator}\\L'=L'^\flat+\langle u^\flat+u^\perp\rangle}} \ppden(L')	
  =\sum_{\substack{u^\flat \in (L'^\flat)^\sharp/L'^\flat\\\val(u^\flat)\ge 0}} \ppden(L^{\prime\flat}+ \langle u^\flat + x \rangle).
		\end{align}
Here the last step uses the fact that $L'=L^{\prime\flat}+ \langle u^\flat + x \rangle$ is integral if and only if  $u^\flat \in (L'^\flat)^\sharp/L'^\flat$ and  $\val(u^\flat)\ge 0$ (since $\val(x)>0$).

It suffices to show $D(L^{\prime \flat})(x)=0$ for any  $L'^\flat$ such that  $L^\flat\subset L'^\flat$ and $L'^{\flat}\not \in \Hor(L^\flat)$.   To show this, we write $L'^\flat=I_{n_1}^{\epsilon_1}\obot L_2$ where $L_2$ is of full type (of rank $n-1-n_1$). Let $u_2$ be the projection of $u^\flat$ to $L_2$.
Then  $$(L'^\flat)^\sharp/L'^\flat=(I_{n_1}^{\epsilon_1}\obot (L_2)^\sharp)/L'^\flat\cong L_2^\sharp/L_2\quad \text{ and }\quad L'^\flat +\langle u^\flat+x\rangle=L'^\flat +\langle u_2+x\rangle.$$
We consider a partition of
$$
(L_2^\sharp)^\circ / L_2=S^+(L_2)\sqcup S^0(L_2)\sqcup S^-(L_2)
$$
with
		\begin{align*}
		S^+(L_2)=(\pi L_2^{\sharp})^{\circ \circ} / L_2, \quad    S^0(L_2)=((\pi L_2^{\sharp})^{\circ} -(\pi L_2^{\sharp})^{\circ  \circ}) / L_2,  \quad S^-(L_2)=  ((L_2^{\sharp})^{\circ} -(\pi L_2^{\sharp})^{\circ}) / L_2.
		\end{align*}
Here, for a lattice $L$,
		\begin{align*}
		    L^{\circ}\coloneqq \{x\in L\mid \val(x)\ge 0\}\text{ and } L^{\circ \circ}\coloneqq \{x\in L\mid \val(x)>0\}.
		\end{align*}
In general, for a full type lattice $L_2$, we also  define
\begin{align}\label{eq: mu 1}
		 \mu^{+}(L):=|(\pi L^{\sharp})^{\circ \circ} / L|, \quad \mu^{0}(L):=|((\pi L^{\sharp})^{\circ} -(\pi L^{\sharp})^{\circ \circ}) / L|, \quad \mu^{-}(L)\coloneqq|((L^{\sharp})^{\circ} -(\pi L^{\sharp})^{\circ}) / L| .
		 \end{align}
		For $\nu\in\{\pm 1\}$, let
	\begin{align}\label{eq: mu2}
		\mu^{0, \nu}(L):=|\{u \in(\pi L^{\sharp})^{\circ} -(\pi L^{\sharp})^{\circ \circ}: \chi((u,u) )=\nu\} / L|.
		\end{align}
Since $L'= L^{\prime \flat}+ \langle u_2 +x \rangle$ with $u_2$ integral and $\val(x) >0$, it is not hard to check that
		\begin{align*}
			t(L')=\begin{cases}
				t(L'^\flat)+1 &\ff \, u_2 \in S^+(L_2),
				\\
				t(L'^\flat)  &\ff \, u_2 \in S^0(L_2),
				\\
				t(L'^\flat)-1 &\ff \, u_2 \in S^-(L_2).
			\end{cases}
		\end{align*}
Set $t =t(L^{\prime \flat})$. There are two cases.

When $t$ is odd,  we can write
$$
L'=L^{\prime \flat} + \langle u_2+ x \rangle
=\begin{cases}
I_{n_1}^{\epsilon_1} \obot L_2' &\ff \, u \in S_2^+(L_2),
\\
I_{n_1}^{\epsilon_1} \obot I_2^1 \obot L_2' &\ff \, u \in S_2^-(L_2).
\end{cases}
$$
In both cases, a simple calculation gives
$$
\chi(L_2') = \epsilon \epsilon_1.
$$
For $t>1$, by Theorem \ref{thm: formula of ppden},
\begin{align*}
    \ppden(L')=  (1-\epsilon\epsilon_1  q^{\frac{t-1}2} )\prod_{\ell=1}^{\frac{t-1}{2}-1}(1-q^{2\ell})\cdot
    \begin{cases}
(1-\epsilon\epsilon_1  q^{\frac{t+1}2})  (1+\epsilon\epsilon_1  q^{\frac{t-1}2}) & \text{ if $u_2\in S^+(L_2)$},\\
    1+\epsilon\epsilon_1  q^{\frac{t-1}2}  & \text{ if $u_2\in S^{0}(L_2)$},\\
    1  & \text{ if $u_2\in S^-(L_2)$}.
    \end{cases}
\end{align*}
For $t=1$, $S^-(L_2)$ is empty and  $L^{\prime \flat}\not \in \mathrm{Hor}(L^\flat)$ implies that $L'= L_{n_1}^{\epsilon_1} \obot L_2'$ with $\chi(L_2') =1$, i.e. $\epsilon_1=\epsilon$.  In this case, by Theorem \ref{thm: formula of ppden},
 $$
 \ppden(L') = \begin{cases}
  1-  q &\ff \,  u \in S^+(L_2),
  \\
   1 &\ff \,  u \in S^0(L_2).
 \end{cases}
$$
Hence by \eqref{eq: D},   we have $D(L^{\prime \flat})=0$ if
\begin{equation} \label{eq:oddD}
 (1-\epsilon\epsilon_1  q^{\frac{t+1}2})(1+\epsilon\epsilon_1  q^{\frac{t-1}2}) \mu^+(L_2)
 +  (1+\epsilon\epsilon_1  q^{\frac{t-1}2}) \mu^0(L_2)
 + \mu^-(L_2)=0.
\end{equation}

When $t =t(L^{\prime \flat})$ is even, $L^{\prime \flat}\not \in \mathrm{Hor}(L^\flat)$ implies that $t>0$. Moreover, if $u_2\in S^0(L_2)$,  we have a decomposition (since $u_2 \in L_2$ is perpendicular to $I_{n_1}^{\epsilon_1}$)
$$
L'=L^{\prime \flat}+\langle u_2+x \rangle =I_{n_1}^{\epsilon_1} \obot \langle u_2 +x \rangle \obot L_2'
$$
for some full type lattice $L_2'$ of rank $t$. Then a direct calculation gives
\begin{equation*}
    \chi(L_2') = (-1)^{n_1} \epsilon_1 \epsilon \chi(u_2) .
\end{equation*}
So we have  by Theorem \ref{thm: formula of ppden},
\begin{align*}
    \ppden(L')=\prod_{\ell=1}^{\frac{t}{2}-1}(1-q^{2\ell})\cdot
    \begin{cases}
    (1-q^{t}) & \text{ if $u_2\in S^+(L_2)$},\\
    1-(-1)^{n_1} \epsilon_1 \epsilon \chi(u_2) q^{\frac{t}2}  & \text{ if $u_2\in S^{0}(L_2)$},\\
    1  & \text{ if $u_2\in S^-(L_2)$}.
    \end{cases}
\end{align*}
Hence by \eqref{eq: D},   we have $D(L^{\prime \flat})=0$ if
 \begin{equation} \label{eq:evenD}
    (1-q^t)\mu^+(L_2)
      +    (1- (-1)^{n_1} \epsilon_1 \epsilon   q^{\frac{t}2})\mu^{0,1}(L_2) + (1+ (-1)^{n_1} \epsilon_1 \epsilon   q^{\frac{t}2})\mu^{0,-1}(L_2)
      +  \mu^-(L_2)=0.
\end{equation}

Now  (\ref{eq:oddD}) follows from  Proposition \ref{prop: lat counting t odd} and  (\ref{eq:evenD}) follows from Proposition \ref{prop: lat counting t even}. Hence we have $D(L^{\prime \flat})=0$ for $L'^\flat$ such that  $L^\flat\subset L'^\flat$ and $L'^{\flat}\not \in \Hor(L^\flat)$. Now the theorem follows from (\ref{eq:Diff}).
	\end{proof}
	
To complete the proof of Theorem  \ref{prop: part FT of denLflat}, we are left to state and  prove Propositions \ref{prop: lat counting t odd} and \ref{prop: lat counting t even}.
	\begin{definition}
		Let $L$ and $L'$ be lattices of full type such that $L\subset L'\subset \pi^{-1} L$. For $? \in\{+, 0,-,\{0,+ 1\},\{0,- 1\}\}$, define
		\begin{align*}
			\mu^{?}(L, L'):=\mu^{?}(L)-[L':L] \mu^{?}(L'),
		\end{align*}
		where $\mu^{?}(L)$ is defined in \eqref{eq: mu 1} and \eqref{eq: mu2}.
	\end{definition}
	
	\begin{lemma}\label{lem: mu+ + mu0 +mu-}
	Let $L$ be a full type lattice of rank $t$. Then
	$$
	\mu^{+}(L)+\mu^{0}(L)+\mu^{-}(L)=\q^{t}\cdot\mu^{+}(L).
	$$
	Let $L$ and $L'$ be full type lattices of rank $t$ such that $L\subset L'\subset \pi^{-1} L$. Then
		\begin{align*}
		\mu^{+}(L,L')+\mu^{0}(L,L')+\mu^{-}(L,L')=\q^{t}\cdot\mu^{+}(L,L').
		\end{align*}
	\end{lemma}
	\begin{proof}
		It suffices to show that  the following map
		\begin{align*}
			(L^{\sharp})^{\circ}/L
			{\rightarrow} (\pi L^{\sharp})^{\circ \circ} / L, \quad x \mapsto \pi x
		\end{align*}
		is surjective and every fiber of this map has size $\q^t$. For $x\in (\pi L^{\sharp})^{\circ \circ}= \pi (L^{\sharp})^{\circ  }$, the fiber at $x$ is
		\begin{align*}
			\{\pi^{-1}(y+x) \in(L^{\sharp})^{\circ}: y \in \overline{L}\}.
		\end{align*}
		Since $x\in \pi (L^{\sharp})^{\circ  }$,
		\begin{align*}
		    \pi^{-1}(y+x) \in(L^{\sharp})^{\circ}\iff (y+x)\in \pi (L^{\sharp})^{\circ} \iff y\in \pi (L^{\sharp})^{\circ}.
		\end{align*}
Moreover, the assumption that $L$ is a full type lattice implies that $L\subset \pi(L^{\sharp})^{\circ}$. Hence
\begin{align*}
    	|\{\pi^{-1}(y+x) \in(L^{\sharp})^{\circ}: y \in \overline{L}\}|=|\overline{L}|=\q^t.
\end{align*}
This proves the first statement. The second statement follows from the first and the definition of $\mu^?(L,L')$.
	\end{proof}

	\begin{definition}
	Let $L$ be a full type lattice of rank $t$. We call $L$ maximal of type $t$ if $t(L')<t$ for any $L'$ such that $L\subsetneq L'\subset L_F$.
	\end{definition}

	\begin{lemma}\label{lem: mu+ + mu0=qmu+}
		If $L$ is non-maximal full type lattice of rank $t$, then there exists a $L'$ such that $L\subset L'\subset \pi^{-1} L$ and
		\begin{align*}
			\mu^{+}(L,L')+\mu^{0}(L,L')=\q\cdot \mu^{+}(L,L').
		\end{align*}
	\end{lemma}
	\begin{proof}
		We need to find a $L'$ such that $L\subset L'\subset \pi^{-1} L$ and
		\begin{align*}
			|((\pi L^{\sharp})^{\circ} - (\pi L'^{\sharp})^{\circ} )/ L|= \q\cdot |((\pi L^{\sharp})^{\circ \circ}- (\pi L'^{\sharp})^{\circ \circ}) / L|.
		\end{align*}
Let $(a_1,\cdots,a_t)$ be the fundamental invariants of $L$.	We consider two cases seperately.\newline
		$(\mathrm{i})$ If $a_t$ is even and $a_t\ge 4$, then we may choose a normal basis $\{\ell_1,\cdots,\ell_t\}$ of $L$ such that $ \langle \ell_1,\cdots,\ell_{t-1}\rangle\perp \ell_t$. Write $(\ell_t,\ell_t)=u_t(-\pi_0)^{\frac{a_t}{2}}$. In this case, we choose $L'=\langle \ell_1,\cdots,\ell_{t-1},\pi^{-1}\ell_t\rangle$, with fundamental invariants  $(a_1,\cdots,a_{t-1},a_{t}-2)$.
		Then
		\begin{align*}
		\pi L^{\sharp}= \langle \pi^{-a_1+1}\ell_1,\cdots, \pi^{-a_t+1}\ell_t\rangle \text{ and } \pi L'^{\sharp}=	\langle \pi^{-a_1+1}\ell_1,\cdots, \pi^{-a_{t-1}+1}\ell_{t-1},\pi^{-a_{t}+2}\ell_{t}\rangle.
		\end{align*}
		For  a fixed $x_0=\sum_{1\le i< t}s_i \pi^{-a_i+1}\ell_i$ where $s_i\in O_F$,  let
		\begin{align*}
			S_{x_0}^{\circ}&:=\{x \in (\pi L^{\sharp})^{\circ}   - (\pi L'^\sharp)^{\circ }: x=x_{0}+s_{t} \pi^{-a_{t}+1} \ell_{t}, s_{t} \in O_{F}\} / L,\\
			S_{x_0}^{\circ\circ}&:=\{x \in (\pi L^{\sharp})^{\circ\circ}   - (\pi L'^\sharp)^{\circ \circ}: x=x_{0}+s_{t} \pi^{-a_{t}+1} \ell_{t}, s_{t} \in O_{F}\} / L.
		\end{align*}
	It suffices to show $ |S_{x_0}^{\circ}|= \q\cdot |S_{x_0}^{\circ\circ}|$.
		Notice that $x=x_{0}+s_{t} \pi^{-a_{t}+1} \ell_{t}\in S^{\circ}_{x_0}$ if and only if
		\begin{align*}
			s_t\in O_F^{\times},\quad  (x,x)=u_t(-\pi_0)^{-\frac{a_t}{2}+1}(u_t^{-1}(-\pi_0)^{\frac{a_t}{2}-1}(x_0,x_0)+s_t\bar{s}_t) \in O_{F_0} ,
		\end{align*}
		and $x\in S^{\circ \circ}_{x_0}$ if and only if
		\begin{align*}
			s_t\in O_F^{\times},\quad   (x,x)=u_t(-\pi_0)^{-\frac{a_t}{2}+1}(u_t^{-1}(-\pi_0)^{\frac{a_t}{2}-1}(x_0,x_0)+s_t\bar{s}_t) \in (\pi_0).
		\end{align*}
Consider the $\pi$-adic expansions
		\begin{align*}
			s_t=\sum_{i\ge 0}^{\infty}b_i\pi^i,\quad -u_t^{-1}(-\pi_0)^{\frac{a_t}{2}-1}(x_0,x_0)=\sum_{i\ge 0}^{\infty} c_i \pi^i,
		\end{align*}
		where $b_i,c_i \in O_{F_0}/(\pi_0)$.
		Then $x\in S^{\circ}_{x_0}$ if and only if $s_t\in O_F^{\times}$, and
		\begin{align*}
			c_0&=b^2_0,\\
			c_1&=b_1b_0-b_0b_1,\\
			c_2&=b_2b_0-b_1b_1+b_
			0b_2,\\
			&\cdots\\
			c_{a_t-3}&=\sum_{i=0}^{a_t-3}(-1)^ib_{a_t-3-i}b_i.
		\end{align*}
		Similarly,  $x\in S^{\circ\circ}_{x_0}$ if and only if $x\in S^{\circ}_{x_0}$ and
		\begin{align*}
			c_{a_t-2}-\sum_{i=1}^{a_t-3}(-1)^ib_{a_t-2-i}b_i=2b_{a_t-2}b_0.
		\end{align*}
		Since $s_t\in O_F^\times$, $b_0 \neq 0$ and $b_{a_t-2}$ is uniquely determined by the above equation. Hence $ |S_{x_0}^{\circ}|= \q\cdot |S_{x_0}^{\circ\circ}|$ as a result.

		$(\mathrm{ii})$ If $a_t$ is odd (since $L$ is non-maximal, $a_t>1$ in this case) or $a_t=2$, then we may choose  a normal basis $\{\ell_1,\cdots,\ell_t\}$ of $L$ such that the moment matrix of $\{\ell_{t-1},\ell_{t}\}$ is $\cH_{a_t}$, where $\cH_{a_t}\coloneqq\begin{pmatrix}0&\pi^{a_t}\\(-\pi)^{a_t}&0\end{pmatrix}$. We may choose $L'=\langle \ell_1,\cdots,\ell_{t-2},\pi^{-1}\ell_{t-1},\ell_{t}\rangle$ with fundamental invariants $(a_1,\cdots,a_{t-2},a_{t}-1,a_t-1)$. In this case,
		\begin{align*}
			\pi L^{\sharp}= \langle \pi^{-a_1+1}\ell_1,\cdots, \pi^{-a_t+1}\ell_t\rangle \text{ and } \pi L'^{\sharp}= \langle \pi^{-a_1+1}\ell_1,\cdots, \pi^{-a_{t}+1}\ell_{t-1},\pi^{-a_{t}+2}\ell_{t}\rangle.
		\end{align*}
		For  a fixed $x_0=\sum_{1\le i< t-1}s_i \pi^{-a_i+1}\ell_i$ where $s_i\in O_F$,  let
		\begin{align*}
			S_{x_0}^{\circ}&:=\{x \in (\pi L^{\sharp})^{\circ}   - (\pi L'^\sharp)^\circ:  x=x_{0}+s_{t-1}\pi^{-a_{t}+1}\ell_{t-1}+s_{t} \pi^{-a_{t}+1} \ell_{t},\text{ where } s_{t-1}, s_{t}\in O_{F}\} / L,\\
			S_{x_0}^{\circ\circ}&:=\{x \in (\pi L^{\sharp})^{\circ\circ}   - (\pi L'^\sharp)^{\circ \circ}:x=x_{0}+s_{t-1}\pi^{-a_{t}+1}\ell_{t-1}+s_{t} \pi^{-a_{t}+1} \ell_{t},\text{ where } s_{t-1}, s_{t}\in O_{F}\} / L.
		\end{align*}
	It suffices to show $ |S_{x_0}^{\circ}|= \q\cdot |S_{x_0}^{\circ\circ}|$.   Notice that $x=x_{0}+s_{t-1}\pi^{-a_{t}+1}\ell_{t-1}+s_{t} \pi^{-a_{t}+1} \ell_{t}\in S^{\circ}_{x_0}$ if and only if
		\begin{align*}
			s_t\in O_F^{\times},\quad  (x,x)=(x_0,x_0)+(s_{t-1}\bar{s}_t+(-1)^{a_t}\bar{s}_{t-1}s_t)(-1)^{-a_t+1}\pi^{-a_t+2} \in O_{F_0},
		\end{align*}
and
$x\in S^{\circ\circ}_{x_0}$ if and only if
		\begin{align*}
			s_t\in O_F^{\times},\quad  (x,x)=(x_0,x_0)+(s_{t-1}\bar{s}_t+(-1)^{a_t}\bar{s}_{t-1}s_t)(-1)^{-a_t+1}\pi^{-a_t+2} \in (\pi_0).
		\end{align*}
			Write
		\begin{align*}
			s_{t-1}=\sum_{i\ge 0}^{\infty}b_i\pi^i,\quad s_{t}=\sum_{i\ge 0}^{\infty}c_i\pi^i,\quad -(-1)^{a_t-1}\pi^{a_t-2}(x_0,x_0)=\sum_{i\ge 0}^{\infty} d_i \pi^i,
		\end{align*}
		where $b_i,c_i,d_i \in O_{F_0}/(\pi_0)$. Then $x\in S_{x_0}^{\circ\circ}$ if and only if $x\in S_{x_0}^{\circ}$ and
		\begin{align*}
		d_{a_t-2}+S=-2b_{a_t-2}c_0.
		\end{align*}
where $S$ is certain expression involving $b_0,\cdots,b_{a_t-3}$ and $c_1,\cdots,c_{a_t-2}$.
Since $s_t\in O_F^{\times}$, $c_0\not =0$. Hence, for any given $S$, the number of choices of $b_{a_t-2}$ is determined if $x\in S^{\circ\circ}_{x_0}$. As a result, $|S_{x_0}^{\circ}|= \q\cdot |S_{x_0}^{\circ\circ}|$.
	\end{proof}

	\begin{proposition}\label{prop: lat counting t odd}
		Assume that $t\ge 1$ is odd and $L$ is a full type lattice of rank $t$.
		Then for any $\chi \in \{\pm 1 \}$, we have
		$$
		(1-\chi q^{\frac{t+1}{2}})(1+\chi q^{\frac{t-1}{2}}) \mu^{+}(L)+(1+\chi q^{\frac{t-1}{2}}) \mu^{0}(L)+\mu^{-}(L)=0.
		$$
	\end{proposition}
	\begin{proof}
		We prove this for maximal $L$ first. We can choose a basis $\{\ell_1,\cdots,\ell_t\}$ of $L$ with moment matrix $\diag(\cH_1^{\frac{t-1}{2}}, u_{t}(-\pi_0))$. Set $L_1=\langle\ell_1,\cdots,\ell_{t-1}\rangle$ and $L_2=\langle\ell_{t}\rangle$. Then we can directly compute that
		\begin{align*}
			(\pi L^\sharp)^{\circ \circ}=L ,\quad (\pi L^\sharp)^\circ=L_1\obot \pi^{-1}L_2,\quad (L^\sharp)^{\circ}=\pi^{-1}L_1\obot \pi^{-1}L_2.
		\end{align*}
		Hence
		\begin{align*}
			\mu^+(L)=|(\pi L^{\sharp})^{\circ \circ} / L|=1,   \quad \mu^0(L)=|((\pi L^\sharp)^\circ - (\pi L^\sharp)^{\circ \circ} )/L|=\q-1,
		\end{align*}
		and
		\begin{align*}
		     \mu^-(L)=|((L^{\sharp})^{\circ} -(\pi L^{\sharp})^{\circ}) / L|=\q^t-\q.
		\end{align*}
		As a result,
		\begin{align*}
			&(1-\chi q^{\frac{t+1}{2}})(1+\chi q^{\frac{t-1}{2}}) \mu^{+}(L)+(1+\chi q^{\frac{t-1}{2}}) \mu^{0}(L)+\mu^{-}(L)\\
			&=(1-\chi q^{\frac{t+1}{2}})(1+\chi q^{\frac{t-1}{ 2}})  +(1+\chi q^{\frac{t-1}{2}}) (\q-1)+\q^t-\q\\
			&=0.
		\end{align*}
		
		Now we assume $L$ is not maximal and the proposition holds for $L'$ such that $L\subsetneq L'$ by an induction on $\val(L)$. With this assumption, it suffices to show
		\begin{align*}
			(1-\chi q^{\frac{t+1}{2}})(1+\chi q^{\frac{t-1}{2}}) \mu^{+}(L,L')+(1+\chi q^{\frac{t-1}{2}}) \mu^{0}(L,L')+\mu^{-}(L,L')=0,
		\end{align*}
		which follows from a combination of Lemmas \ref{lem: mu+ + mu0 +mu-} and \ref{lem: mu+ + mu0=qmu+}.
	\end{proof}

		\begin{lemma}\label{lem:mu+=mu-}
		If $L$ is non-maximal full type lattice of rank $t$, then there exists a $L'$ such that $L\subset L'\subset \pi^{-1} L$ and
		\begin{align*}
			\mu^{0,+}(L,L')=\mu^{0,-}(L,L').
		\end{align*}
	\end{lemma}
	\begin{proof}
		Let
		\begin{align*}
			S^{\nu}:=\{x \in ((\pi L^{\sharp})^{\circ}  - (\pi L^\sharp)^{\circ \circ}) - ((\pi L'^{\sharp})^{\circ}- (\pi L'^\sharp)^{\circ \circ}):   \chi(\langle x\rangle )=\nu \} / L.
		\end{align*}
		We need to show $|S^{+1}|=|S^{-1}|$.
		Let $\{\ell_1,\cdots,\ell_t\}$ be a normal basis of $L$, and let $\{a_1,\cdots, a_t\}$ be the set of fundamental invariants of $L$. We consider two cases seperately.\newline
		$(\mathrm{i})$ If $a_t$ is even and $a_t\ge 4$, then we may choose a normal basis $\{\ell_1,\cdots,\ell_t\}$ of $L$ such that $ \langle \ell_1,\cdots,\ell_{t-1}\rangle\perp\ell_t $. In this case, we choose $L'=\langle \ell_1,\cdots,\ell_{t-1},\pi^{-1}\ell_t\rangle$, with fundamental invariants  $(a_1,\cdots,a_{t-1},a_{t}-2)$.
		Then
		\begin{align*}
			\pi L^{\sharp}= \langle \pi^{-a_1+1}\ell_1,\cdots, \pi^{-a_t+1}\ell_t\rangle \text{ and } \pi L'^{\sharp}=	\langle \pi^{-a_1+1}\ell_1,\cdots, \pi^{-a_{t-1}+1}\ell_{t-1},\pi^{-a_{t}+2}\ell_{t}\rangle.
		\end{align*}
		For a fixed $x_0=\sum_{1\le i< t}s_i \pi^{-a_i+1}\ell_i$ where $s_i\in O_F$, we set
		\begin{align*}
			S_{x_0}^{\nu}:=\{x \in ((\pi L^{\sharp})^{\circ}  - (\pi L^\sharp)^{\circ \circ}) - ((\pi L'^{\sharp})^{\circ}- (\pi L'^\sharp)^{\circ \circ}): x=x_{0}+s_{t} \pi^{-a_{t}+1} \ell_{t}, s_{t} \in O_{F}, \chi(\langle x\rangle )=\nu \} / L.
		\end{align*}
		We need to show $ |S_{x_0}^{+1}|=|S_{x_0}^{-1}|$. Write $(\ell_t,\ell_t)=u_t(-\pi_0)^{\frac{a_t}{2}}$. Notice that $x=x_{0}+s_{t} \pi^{-a_{t}+1} \ell_{t}\in S^{\nu}_{x_0}$ if and only if
		\begin{align}\label{eq: conditions}
			s_t\in O_F^{\times},\quad  (x,x) \in O_{F_0}^{\times},\quad \chi((x,x))=   \nu.
		\end{align}
	Notice that
		\begin{align*}
		(x,x)= u_t(-\pi_0)^{-\frac{a_t}{2}+1}(u_t^{-1}(-\pi_0)^{\frac{a_t}{2}-1}(x_0,x_0)+s_t\bar{s}_t).
		\end{align*}
	Write
		\begin{align*}
			s_t=\sum_{i\ge 0}^{\infty}b_i\pi^i,\quad \text{ and } -u_t^{-1}(-\pi_0)^{\frac{a_t}{2}-1}(x_0,x_0)=\sum_{i\ge 0}^{\infty} c_i \pi^i,
		\end{align*}
		where $b_i,c_i \in  O_{F_0}/(\pi_0)$.
		Then the conditions in \eqref{eq: conditions} are equivalent to the following equations:
		\begin{align*}
			c_0&=b^2_0\not =0,\\
			c_1&=b_1b_0-b_0b_1,\\
			c_2&=b_2b_0-b_1b_1+b_
			0b_2,\\
			&\cdots\\
			c_{a_t-3}&=\sum_{i=0}^{a_t-3}(-1)^ib_{a_t-3-i}b_i,\\
			\nu &=\chi\Big(u_t(-\pi_0)^{-\frac{a_t}{2}+1}\big(-c_{a_t-2}+\sum_{i=0}^{a_t-2}(-1)^ib_{a_t-2-i}b_i\big)\Big).
		\end{align*}
		Since $a_t$ is even by assumption, the last equation is the same with
		\begin{align*}
			\nu   &=\chi\Big(u_t(-\pi_0)^{-\frac{a_t}{2}+1}\big(-c_{a_t-2}+\sum_{i=1}^{a_t-3}(-1)^ib_{a_t-2-i}b_i+2b_0b_{a_t-2}\big)\Big).
		\end{align*}
		Notice that the possible choices of $\{b_0,\cdots,b_{a_t-3}\}$ are determined by the first $a_t-2$ equations. And for a given choice of  $\{b_0,\cdots,b_{a_t-3}\}$, the number of choices of $b_{a_t-2}$ that satisfies the last equation is clearly independent of $\nu$ since $b_0\not = 0$.\newline
		$(\mathrm{ii})$ If $a_t$ is odd or $a_t=2$, then we may choose  a normal basis $\{\ell_1,\cdots,\ell_t\}$ of $L$ such that the moment matrix of $\{\ell_{t-1},\ell_{t}\}$ is $\cH_{a_t}$, where $\cH_{a_t}\coloneqq\begin{pmatrix}0&\pi^{a_t}\\(-\pi)^{a_t}&0\end{pmatrix}$. We may choose $L'=\langle \ell_1,\cdots,\ell_{t-2},\pi^{-1}\ell_{t-1},\ell_{t}\rangle$ with fundamental invariants $(a_1,\cdots,a_{t-2},a_{t-1}-1,a_t-1)$. In this case,
		\begin{align*}
			\pi L^{\sharp}= \langle \pi^{-a_1+1}\ell_1,\cdots, \pi^{-a_t+1}\ell_t\rangle \text{ and } \pi L'^{\sharp}= \langle \pi^{-a_1+1}\ell_1,\cdots, \pi^{-a_{t-1}+1}\ell_{t-1},\pi^{-a_{t}+2}\ell_{t}\rangle.
		\end{align*}
		For  a fixed  $x_0=\sum_{1\le i< t-1}s_i \pi^{-a_i+1}\ell_i$, we set
		\begin{align*}
			S_{x_0}^{\nu}:=\{x \in ((\pi L^{\sharp})^{\circ}  - (\pi L^\sharp)^{\circ \circ}) - ((\pi L'^{\sharp})^{\circ}- (\pi L'^\sharp)^{\circ \circ}): &x=x_{0}+s_{t-1}\pi^{-a_{t-1}+1}\ell_{t-1}+s_{t} \pi^{-a_{t}+1} \ell_{t}, \\
			&\quad\quad\quad\quad\quad s_{t-1}, s_{t}\in O_{F}, \chi(\langle x\rangle  )=\nu\} / L.
		\end{align*}
	It suffices to show $|S_{x_0}^{+1}|= |S_{x_0}^{-1}|$.   Notice that $x=x_{0}+s_{t-1}\pi^{-a_{t-1}+1}\ell_{t-1}+s_{t} \pi^{-a_{t}+1} \ell_{t}\in S^{\nu}_{x_0}$ if and only if
		\begin{align}\label{eq: condition}
			s_t\in O_F^{\times},\quad  (x,x)=(x_0,x_0)+(s_{t-1}\bar{s}_t+(-1)^{a_t}\bar{s}_{t-1}s_t)(-1)^{-a_t+1}\pi^{-a_t+2} \in O_{F_0}^{\times},\quad \chi((x,x))=\nu.
		\end{align}
		Write
		\begin{align*}
			s_{t-1}=\sum_{i\ge 0}^{\infty}b_i\pi^i,\quad s_{t}=\sum_{i\ge 0}^{\infty}c_i\pi^i,\quad -(-1)^{a_t-1}\pi^{a_t-2}(x_0,x_0)=\sum_{i\ge 0}^{\infty} d_i \pi^i,
		\end{align*}
		where $b_i,c_i,d_i \in O_{F_0}/(\pi_0)$. Then the condition in \eqref{eq: condition} is equivalent to the following equations:
		\begin{align*}
		  d_0&=b_0c_0+(-1)^{a_t}b_0c_0,\\
		  d_1&=b_1c_0-b_0c_1+(-1)^{a_t}(-b_1c_0+b_0c_0),\\
		    &\cdots\\
		 \nu &=\chi\Big((-1)^{a_t-1}\pi^{-a_t+2}\big(-d_{a_t-2}+S+2b_{a_t-2}c_0\big) \Big),
		\end{align*}
where $S$ is certain expression involving $b_0,\cdots,b_{a_t-3}$ and $c_1,\cdots,c_{a_t-2}$.
Since $s_t\in O_F^{\times}$, $c_0\not =0$. Hence, for any given $S$, the number of choices of $b_{a_t-2}$ that satisfies the last equation is clearly independent of $\nu$.
	\end{proof}

	\begin{proposition}\label{prop: lat counting t even}
		Assume that $t\ge 1$ is even and that $L$ is a full type lattice of rank $t$. Then for any $\chi \in \{\pm 1\},$ we have
		\begin{align*}
			(1-\q^t)\mu^{+}(L)+(1-\chi \q^{\frac{t}{2}})\mu^{0,+1}(L)+(1+\chi \q^{\frac{t}{2}})\mu^{0,-1}(L)+\mu^{-}(L)=0.
		\end{align*}
	\end{proposition}
	\begin{proof}
		We prove this for maximal $L$ first. There are two cases we need to consider. \newline
		(i) If we can choose a basis $\{\ell_1,\cdots,\ell_t\}$ of $L$ with moment matrix $\diag(\cH_1^{\frac{t}{2}-1}, u_{t-1}(-\pi_0), u_{t}(-\pi_0))$ where $\chi(-u_{t-1}u_t)=-1$, then set $L_1=\langle\ell_1,\cdots,\ell_{t-2}\rangle$ and $L_2=\langle\ell_{t-1},\ell_{t}\rangle$. In this case, a direct computation shows that
		\begin{align*}
			(\pi L^\sharp)^{\circ \circ}=L ,\quad (\pi L^\sharp)^\circ=L_1\obot \pi^{-1}L_2,\quad (L^\sharp)^{\circ}=\pi^{-1}L.
		\end{align*}
		Hence
		\begin{align}\label{eq: mu+,mu- 1}
			\mu^+(L)=|(\pi L^{\sharp})^{\circ \circ} / L|=1,   \quad \mu^-(L)=|((L^{\sharp})^{\circ} -(\pi L^{\sharp})^{\circ}) / L|=q^{t}-\q^2.
		\end{align}
		Moreover,
		\begin{align*}
			\mu^{0,\nu}(L)=|\{(x,y)\in \mathbb{F}_\q^2- (0,0)\mid \chi(u_{t-1}x^2+u_ty^2)=\nu\}|.
		\end{align*}
		It is well known that
		\begin{align*}
			|\{(x,y)\in \mathbb{F}_\q^2- (0,0)\mid u_{t-1}x^2+u_ty^2=1\}|=\q-\chi(-u_{t-1}u_t)=\q+1.
		\end{align*}
		Hence
		\begin{align}\label{eq: mu0+=mu0-}
			\mu^{0,+1}(L)= \mu^{0,-1}(L)=\frac{\q^2-1}{2}.
		\end{align}
		Combining \eqref{eq: mu+,mu- 1} and \eqref{eq: mu0+=mu0-}, we have
		\begin{align*}
			&(1-\q^t)\mu^{+}(L)+(1-\chi \q^{\frac{t}{2}})\mu^{0,+}(L)+(1+\chi \q^{\frac{t}{2}})\mu^{0,-}(L)+\mu^{-}(L)\\
			&=(1-q^t)+(\q^{2}-1)+(\q^t-\q^2)=0.
		\end{align*}
		(ii) If we can choose a basis $\{\ell_1,\cdots,\ell_t\}$ of $L$ with moment matrix $ \cH_1^{\frac{t}{2}}$, then we can directly compute that
		\begin{align*}
			(\pi L^\sharp)^{\circ \circ}=L ,\quad (\pi L^\sharp)^\circ=L,\quad (L^\sharp)^{\circ}=\pi^{-1}L.
		\end{align*}
		Hence
		\begin{align*}
			\mu^+(L)=|(\pi L^{\sharp})^{\circ \circ} / L|=1,   \quad \mu^0(L)=0,\quad  \mu^-(L)=|((L^{\sharp})^{\circ} -(\pi L^{\sharp})^{\circ}) / L|=q^{t}-1.
		\end{align*}
		As a result we have
		\begin{align*}
			(1-\q^t)\mu^{+}(L)+(1-\chi \q^{\frac{t}{2}})\mu^{0,+}(L)+(1+\chi \q^{\frac{t}{2}})\mu^{0,-}(L)+\mu^{-}(L)=(1-q^t)+(\q^t-1)=0.
		\end{align*}

		Now we assume $L$ is not maximal and the proposition holds for $L'$ such that $L\subsetneq L'$ by an induction on $\val(L)$. With this assumption, it suffices to show
		\begin{align*}
			(1-\q^t)\mu^{+}(L,L')+(1-\chi \q^{\frac{t}{2}})\mu^{0,+1}(L,L')+(1+\chi \q^{\frac{t}{2}})\mu^{0,-1}(L,L')+\mu^{-}(L,L')=0,
		\end{align*}
		which follows from a combination of Lemmas \ref{lem: mu+ + mu0 +mu-} and \ref{lem:mu+=mu-}.
	\end{proof}

\section{Proof of the main theorem}\label{sec: pf of main theorem}
We prove the main theorem in this section by an induction on $\val(L)$ using the results we obtained about the partial Fourier transform in previous sections.

\subsection{Comparison of horizontal intersection numbers} 	

\begin{lemma}\label{lem: cancellation of Phi} Let $L\subset \bV$ be a lattice.
If $L=L_1\obot L_2$ where $L_1$ is unimodular, then
\begin{align*}
    \Int(L)-\pden(L)=\Int(L_2)-\pden(L_2).
\end{align*}
\end{lemma}
\begin{proof}
The lemma follows from comparing \eqref{eq: pden(L)-pden(L_2)} with \eqref{eq: geom cancel}.
\end{proof}

\begin{definition}\label{def:horizontal Den}
Let $L^\flat\subset \bV$ be a non-degenerate lattice of rank $n-1$, and $x\in \bV\setminus L_{F}^\flat$. Define
\begin{equation}\label{eq:horizontal Den}
    \pden_{L^{\flat},\mathscr{H}}(x)=\sum_{\substack{L^\flat \subset L' \subset L'^\sharp\\ L'^\flat \in \Hor(L^\flat)}}\ppden(L') 1_{L'}(x).
\end{equation}
\end{definition}

\begin{lemma}\label{lem:main thm for horizontal lattice}
If $L^\flat\subset \bV$ is horizontal, then
\[\Int_{L^{\flat},\mathscr{H}}(x)= \pden_{L^{\flat},\mathscr{H}}(x),\]
where $\Int_{L^{\flat},\mathscr{H}}$ is defined in Definition \ref{def:horizontal and vertical intersection number}.
\end{lemma}
\begin{proof}
Let $L=L^\flat\oplus \langle x\rangle$.
By Lemma \ref{lem:modified horizontal part}, we know
\begin{equation}\label{eq:horizontal intersection is intersection}
    \Int_{L^{\flat},\mathscr{H}}(x)=\Int_{L^{\flat}}(x)=\Int(L).
\end{equation}
On the other hand, since $L^\flat$ is horizontal, by Lemma \ref{lem:alternative def of horizontal lattice} any integral lattice of rank $n-1$ containing $L^\flat$ is horizontal, hence we have
\begin{equation}\label{eq:horizontal density is density}
    \pden_{L^{\flat},\mathscr{H}}(x)=\pden_{L^{\flat}}(x)=\pden(L).
\end{equation}
So it suffices to prove $\Int(L)=\pden(L)$.

When $n=2$, by \eqref{eq:horizontal intersection is intersection} and \eqref{eq:horizontal density is density},
the lemma is a consequence of \cite[Theorem 1.1]{Shi2} and \cite[Theorem 1.1]{HSY}.
When $n>2$, $L^\flat$  has a unimodular direct summand $L_1$ of rank $n-2$ such that $L=L_1\obot L_2$ and $L_2^\flat:=L_{2,F}\cap L^\flat$ is a horizontal lattice in $L_{2,F}$. The lemma follows from the the case $n=2$ and Lemma \ref{lem: cancellation of Phi}.
\end{proof}

\begin{lemma}\label{lem:intersection horizontal difference cycle}
If $M^\flat\subset \bV$ is horizontal, then
\[\chi(\cN,{}^\bL\cZ(M^\flat)^\circ\cdot \cZ(x))=\sum_{\substack{M^\flat\subset L'\subset L'^\sharp\\ L'^\flat=M^\flat}} \ppden(L') 1_{L'}(x),\]
where ${}^\bL\cZ(M^\flat)^\circ$ is as in Definition \ref{def:horizontal difference cycle}.
\end{lemma}
\begin{proof}
By Definition \ref{def:horizontal and vertical intersection number}, we have
\begin{equation}
    \Int_{M^{\flat},\mathscr{H}}(x)=\chi(\cN,{}^\bL\cZ(M^\flat)^\circ\cdot \cZ(x))+\sum_{\substack{L'^\flat\in \Hor(M^\flat)\\ L'^\flat\neq M^\flat}} \chi(\cN,{}^\bL\cZ(L'^\flat)^\circ\cdot \cZ(x)).
\end{equation}
We now prove the lemma by induction on $\val(M^\flat)$. When $M^\flat$ is unimodular, the lemma is the same as Lemma \ref{lem:main thm for horizontal lattice}. In general notice that  any integral lattice of rank $n-1$ containing $M^\flat$ is horizontal by Lemma \ref{lem:alternative def of horizontal lattice}. Applying the induction hypothesis to the right hand side of the above formula and applying Lemma \ref{lem:main thm for horizontal lattice} to the left hand had, we obtain
\begin{equation}\label{eq:int L H decomposition}
  \sum_{\substack{M^\flat \subset L' \subset L'^\sharp\\ L'^\flat \in \Hor(M^\flat)}}\ppden(L') 1_{L'}(x)=\chi(\cN,{}^\bL\cZ(M^\flat)^\circ\cdot \cZ(x))+\sum_{\substack{M^\flat\subset L'\subset L'^\vee\\ L'^\flat\neq M^\flat}} \ppden(L') 1_{L'}(x).
\end{equation}
Subtract the left hand side by the second term of the right hand side of the equation, the lemma is proved.
\end{proof}

\begin{theorem}\label{thm: hor part}
    For a non-degenerate lattice $L^\flat\subset \bV$ of rank $n-1$, and $x\in \bV\setminus L_{F}^\flat$, we have
    \begin{align*}
       \Int_{L^{\flat},\mathscr{H}}(x)= \pden_{L^{\flat},\mathscr{H}}(x).
    \end{align*}
\end{theorem}
\begin{proof}
By the definition of $\Int_{L^{\flat},\mathscr{H}}(x)$, we have
\[\Int_{L^{\flat},\mathscr{H}}(x)=\sum_{M^\flat\in \Hor(L^\flat)}  \chi(\cN,{}^\bL\cZ(M^\flat)^\circ\cdot \cZ(x)).\]
The theorem now follows from \eqref{eq:horizontal Den} and Lemma \ref{lem:intersection horizontal difference cycle}.
\end{proof}

\subsection{Proof of the main theorem}	
The following is an analogue of Lemma 9.3.1 of \cite{LZ2}.
\begin{lemma}\label{lem: ind on val of Lflat}
  Let $L^\flat\subset \bV$ be a non-degenerate lattice of rank $n-1$ and $\mathbb{W}=(L^\flat_{F})^{\perp}$.   For $x\not \in L^\flat \obot \mathbb{W}$, there exists an $O_F$-lattice  $L'^\flat$ of rank $n-1$ and $x'\in \bV$ such that
    \begin{align*}
        \mathrm{val}(L'^\flat)<\mathrm{val}(L^\flat) \text{ and } L'^\flat + \langle x'\rangle =L^\flat + \langle x \rangle.
    \end{align*}
\end{lemma}	
\begin{proof}
 Assume that $L^\flat\subset \bV$ has fundamental invariants $(a_1,\cdots, a_{n-1})$. Let $\{\ell_1,\cdots,\ell_{n-1}\}$ be a basis of $L^\flat$ whose moment matrix is
$$\diag(\cH_{b_1},\cH_{b_3},\cdots,\cH_{b_{2s-1}},u_{2s+1}\pi^{b_{2s+1}},\cdots,u_{n-1}\pi^{b_{n-1}}),$$
where $b_1,\cdots,b_{2s-1}$ are odd and $\cH_j=\begin{pmatrix}0&\pi^{j}\\(-\pi)^{j}&0\end{pmatrix}$. Notice that $\{b_1,\cdots,b_{n-1}\}=\{a_1,\cdots,a_{n-1}\}$.
The moment matrix of $\{\ell_1,\cdots,\ell_{n-1},x\}$ is
\begin{align*}
  T=\left(\begin{array}{cccc}
\cH_{b_1} & & & \left(\ell_{1}, x\right) \\
& \ddots & & \vdots \\
& & u_{n-1}\pi^{b_{n-1}} & \left(\ell_{n-1}, x\right) \\
\left(x, \ell_{1}\right) & \cdots & \left(x, \ell_{n-1}\right) & (x, x)
\end{array}\right).
\end{align*}
Assume $(a_1',\cdots,a_n')$ is the fundamental invariants of $L^\flat +\langle x \rangle$. According to Lemma 2.23 of \cite{LL2}, $a_1'+\cdots+a_{n-1}'$ equals the minimal valuation of the $(n-1)\times (n-1)$ minors of $T$.

Write $x=x^\flat +x^\perp$ where   $x^\flat \in L_F^\flat$ and $x^\perp \in \bW$.   If $x^\flat \not \in L^\flat$, then we can write $x^\flat=\sum_{j=1}^{n-1} \alpha_j \ell_j$ where $\alpha_i\not \in O_F$ for some $i$. First, we assume $\alpha_i \not\in O_F$ for some $i\le 2s$. The valuation of the $(n,i)$-th minor of $T$ (removing $n$-th row and $i$-th column) equals to
\begin{align*}
   \begin{cases}
   \val_\pi((\ell_{i+1},x))-b_i+(b_1+\cdots+b_{n-1})& \text{ if $i$ is odd},\\
\val_\pi((\ell_{i-1},x))-b_i+(b_1+\cdots+b_{n-1})& \text{ if $i$ is even}.
   \end{cases}
\end{align*}
Since $\alpha_i \not\in O_F$,  we have $\val_\pi ((\ell_{i+1},x))<b_i$ if $i$ is odd and $\val_\pi ((\ell_{i-1},x))<b_i$ if $i$ is even. In particular, $\sum_{j=1}^{n-1}a_j'<\sum_{j=1}^{n-1}a_j$. Now if we choose a normal basis $\{\ell_1',\cdots,\ell_{n}'\}$ of $L^\flat +\langle x \rangle$, then  $L'^\flat=\langle \ell_1',\cdots,\ell_{n-1}'\}$ and $x'=\ell_n'$ satisfy the property we want.

Now we assume $\alpha_i\not\in O_F$ for some $2s<i\le n-1$. The valuation of the $(n,i)$-th minor of $T$ equals to
\begin{align*}
    \val_\pi(\ell_i,x)-b_i+(b_1+\cdots+b_{n-1}).
\end{align*}
Since $\alpha_i\not\in O_F$, $ \val_\pi(\ell_i,x)<b_i$, hence $\sum_{j=1}^{n-1}a_j'<\sum_{j=1}^{n-1}a_j$. Now if we choose a normal basis $\{\ell_1',\cdots,\ell_{n}'\}$ of $L^\flat +\langle x \rangle$, then  $L'^\flat=\langle \ell_1',\cdots,\ell_{n-1}'\}$ and $x'=\ell_n'$ satisfy the property we want.
\end{proof}

For any $L$, we can write it as $L^\flat +\langle x \rangle$ where $L^\flat$ is a non-degenerate hermitian $O_F$-lattice of rank $n-1$, and $x\in \bV\setminus L^\flat$. Therefore, in order to show $\Int(L)=\pden(L)$, it suffices to show the following theorem.
\begin{theorem}\label{thm: main thm}
    Let $L^\flat\subset \bV$ be a non-degenerate lattice of rank $n-1$. For any $x \in \bV\setminus L_F^\flat$, we have
    \begin{align*}
      \Int_{L^\flat}(x)=\pden_{L^\flat}(x).
    \end{align*}
\end{theorem}	
\begin{proof}
For $x \in \bV\setminus L_F^\flat$, let $\Phi_{L^\flat}(x)= \Int_{L^\flat}(x)-\pden_{L^\flat}(x)$.
We need to show $\Phi_{L^\flat}(x)=0$. We prove the theorem by an induction on $\val(L^\flat)$. If $L^\flat$ is not integral, then $ \Int_{L^\flat}(x)=0$ as $\cZ(L)$ is empty by Proposition \ref{prop:reducedlocus}. Moreover $\pden_{L^\flat}(x)=0$ by Corollary \ref{cor: int of pden}. Hence the theorem is true in this case.

Now we assume $L^\flat$ is integral. By induction hypothesis and Lemma \ref{lem: ind on val of Lflat}, we may assume $\mathrm{supp}(\Phi_{L^\flat})\subset L^\flat \obot \mathbb{W}$ where $\mathbb{W}=(L^\flat_{F})^\perp$. Since $\Phi_{L^\flat}(x)$ is invariant under the translation of $L^\flat$, we may write
\begin{align}\label{eq: Phi=1 otimes phiw}
  \Phi_{L^\flat}(x)=1_{L^\flat}\otimes \phi_{\mathbb{W}}(x),
\end{align}
where $\phi_{\mathbb{W}}(x)$ is a function on $\mathbb{W}\setminus \{0\}$.  Then we have  by definition
$$
\Phi^{\perp}_{L^\flat} = \vol(L^\flat) \phi_{\bW}.
$$
Theorem \ref{thm: hor part} implies that
\[\Phi_{L^\flat}(x)=\Phi_{L^\flat,\mathscr{V}}(x):= \Int_{L^\flat,\mathscr{V}}(x)-\pden_{L^\flat,\mathscr{V}}(x).\]
Hence
\begin{align}\label{eq: eq of supp}
 \vol(L^\flat)   \phi_{\mathbb{W}} =  \Phi^{\perp}_{L^\flat,\mathscr{V}}.
\end{align}
By  induction on the rank of $L$ and Lemma \ref{lem: cancellation of Phi}, we may assume $\Phi_{L^\flat}(x)=0$, hence $\phi_{\mathbb{W}}(x)=0$  for $x\in \mathbb{W}^{= 0}$. Combining this  with the non-integral case, we know $\phi_{\mathbb{W}}(x)=0$ for $x\in \mathbb{W}^{\le 0}$. As a result, we have $\Phi^{\perp}_{L^\flat,\mathscr{V}}(x)=0$ for  $x\in \mathbb{W}^{\le 0}$ by \eqref{eq: eq of supp}. By  Theorem \ref{prop: part FT of denLflat} and Theorem \ref{lem:Int vertical is schwartz}, we have $\Phi^{\perp}_{L^\flat,\mathscr{V}}(x)=0$ for $x\in \bW^{\ge 0}\setminus \{0\}$. Hence $\Phi^{\perp}_{L^\flat,\mathscr{V}}(x)=0$ for all $x\in \mathbb{W}\setminus \{0\}$. Consequently, $\phi_\bW(x)=0$ by \eqref{eq: eq of supp}.
\end{proof}

Combining this theorem with Theorem \ref{thm:Int Gamma_0 invariant}, we have the following corollary.
\begin{corollary}
    Let $L^\flat \subset \bV$ be a non-degenerate lattice of rank $n-1$. Then   $\pden_{L^\flat,\mathscr{V}} (x)\in \mathscr{S}(\bV)^{\ge -1}$ is a Schwartz function.
\end{corollary}

\section{Global applications}\label{sec:global application}
\subsection{Shimura varieties}\label{sec:shimura-varieties}
In this section, we switch to global situation and  will closely follow \cite{RSZdiagonal} and \cite{RSZsurvey}. Let $F$ be a CM number field with  maximal  totally real subfield $F_0$. We fix  a CM type $\Phi\subset \Hom(F, \overline{\mathbb{Q}})$ of $F$ and a distinguished element $\phi_0\in \Phi$. We fix an embedding $\overline{\mathbb{Q}}\hookrightarrow \mathbb{C}$ and identify the CM type $\Phi$ with the set of archimedean places of $F$, and also with the set of archimedean places of $F_0$. Let $V$ be an $F/F_0$-hermitian space of dimension $n\ge2$. Let $V_\phi=V \otimes_{F, \phi}\mathbb{C}$ be the associated $\mathbb{C}/\mathbb{R}$-hermitian space for $\phi\in\Phi$. Assume the signature of $V_\phi$ is given by
$$  (r_\phi,r_{\bar\phi})=\begin{cases}
  (n-1,1), & \phi=\phi_0,\\
  (n, 0), & \phi\in \Phi\setminus\{\phi_0\}.
\end{cases}$$
Define a variant $G^\mathbb{Q}$ of the unitary similitude group $\mathrm{GU}(V)$ by $$G^\mathbb{Q}\coloneqq \{g\in \Res_{F_0/\mathbb{Q}}\mathrm{GU}(V): c(g)\in \mathbb{G}_m\},$$  where $c$ denotes the similitude character. Define a cocharacter $$h_{G^\mathbb{Q}}: \mathbb{C}^\times\rightarrow G^\mathbb{Q}(\mathbb{R})\subset \prod_{\phi\in\Phi}\mathrm{GU}(V_\phi)(\mathbb{R})\simeq\prod_{\phi\in\Phi} \mathrm{GU}(r_\phi,r_{\bar \phi})(\mathbb{R}),$$ where its $\phi$-component is given by $$h_{G^\mathbb{Q},\phi}(z)=\diag\{ z\cdot 1_{r_\phi},\bar z \cdot 1_{r_{\bar\phi}}\}$$ under the decomposition of $V_\phi$ into positive definite and negative definite parts. Then its $G^\mathbb{Q}(\mathbb{R})$-conjugacy class defines a Shimura datum $(G^\mathbb{Q},\{h_{G^\mathbb{Q}}\})$. Let $E_{r}=E(G^\mathbb{Q}, \{h_{G^\mathbb{Q}}\})$ be the reflex field, i.e., the subfield of $\overline{\mathbb{Q}}$ fixed by $\{\sigma\in\mathrm{Aut}(\overline{\mathbb{Q}}/\mathbb{Q}): \sigma^*(r)=r\}$ where $r: \Hom(F, \overline{\mathbb{Q}})\rightarrow \mathbb{Z}$ is the function defined by $r(\phi)=r_\phi$.

We similarly define the group $Z^\mathbb{Q}$ (a torus) associated to a totally positive definite $F/F_0$-hermitian space of dimension 1 (i.e., of signature $\{(1,0)_{\phi\in\Phi}\}$) and a cocharacter $h_{Z^\mathbb{Q}}$ of $Z^\mathbb{Q}$. The reflex field $E_\Phi=E(Z^\mathbb{Q}, \{h_{Z^\mathbb{Q}}\})$ is equal to the reflex field of the CM type $\Phi$, i.e., the subfield of $\overline{\mathbb{Q}}$ fixed by $\{\sigma\in\Gal(\overline{\mathbb{Q}}/\mathbb{Q}): \sigma\circ\Phi=\Phi\}$.

Now define a Shimura datum $(\wit G, \{h_{\wit G}\})$ by
\[\wit G\coloneqq Z^\mathbb{Q}\times_{\mathbb{G}_m} G^\mathbb{Q}=\{(z,g)\in Z^\mathbb{Q}\times G^\mathbb{Q}\mid \mathrm{Nm}_{F/F_0}(z)=c(g)\},\quad h_{\wit G}=(h_{Z^\mathbb{Q}}, h_{G^\mathbb{Q}}).\]
Then $\wit{G}\cong Z^\Q\times G$ where $G=\Res_{F_0/\Q} \mathrm{U}(V)$.
Its reflex field $E$ is equal to the composite $E_rE_\Phi$, and the CM field $F$ becomes a subfield of $E$ via the embedding $\phi_0$. We remark that $E=F$ when $F/\mathbb{Q}$ is Galois, or when $F=F_0 \kappa$ for some imaginary quadratic $\kappa/\mathbb{Q}$ and the CM type $\Phi$ is induced from a CM type of $\kappa/\mathbb{Q}$ (e.g., when $F_0=\mathbb{Q}$). Assume that $K_{Z^\mathbb{Q}}\subset Z^\mathbb{Q}(\mathbb{A}_f)$ is the unique maximal open compact subgroup, and $K_G=\prod_{v} K_{G,v}$ where $v$ runs over finite places of $F_0$ is a compact open subgroup of $G(\bA_f)$. Let $K=K_{Z^\Q}\times K_G\subset \wit{G}(\mathbb{A}_f)$. Then the associated Shimura variety $\Sh_{\KG}=\Sh_{\KG}(\wit G,\{h_{\wit G}\})$ is a Deligne-Mumford stack of dimension $n-1$ and has a canonical model over $\Spec E$.

\subsection{Integral model}
In this subsection we run through the set-up of \cite[\S 14]{LZ}.
Let $\mathbf{m}=(m_v)_v$ be a collection of integers $m_v\geq 0$ indexed by finite places of $F_0$ such that $m_v=0$ for all but finitely many places and $m_v=0$ for all places $v$ that are nonsplit in $F$.
Let $\Lambda$ be an $O_F$-lattice of $V$.
Assume that for any finite place $v$ of $F_0$ (with residue characteristic $p$), the following conditions are satisfied where $\tilde{\nu}:\bar\Q\rightarrow \bar\Q_p$ is an embedding that induces a place $\nu$ of $E$.
\begin{enumerate}[label=(G\arabic*)]\setcounter{enumi}{-1}
\item\label{item:G0} If $p=2$, then $v$ is unramified in $F$.
\item\label{item:G1} If $v$ is inert in $F$ and $V_v$ is split, then  $\Lambda_v\subset V_v$ is self-dual and $K_{G,v}$ is the stabilizer of $\Lambda_v$. If $v$ is further ramified over $p$ and $\nu$ is any place of $E$ above $v$, then the subset $\{\phi\in \Phi: \tilde{\nu}\circ\phi \text{ induces } w\}\subset\Hom(F_{w}, \overline{\mathbb{Q}}_p)$ is the pullback of a CM type $\Phi^\mathrm{ur}\subset \Hom(F_{w}^\mathrm{ur}, \overline{\mathbb{Q}}_p)$ of $F_{w}^\mathrm{ur}$. Here $w$ is the place of $F$ above $v$ and $F_{w}^\mathrm{ur}$ is the maximal subfield  of $F_{w}$ unramified over $\mathbb{Q}_p$.
\item\label{item:G2} If $v$ is inert in $F$ and $V_v$ is nonsplit, then $v$ is unramified over $p$ and $\Lambda_v\subset V_v$ is almost self-dual, i.e., $\Lambda_v^\sharp/\Lambda_v$ is a $1$ dimensional space over the residue field of $F_w$. Moreover $K_{G,v}$ is the stabilizer of $\Lambda_v$.
\item\label{item:G3} If $v$ is ramified in $F$, then $v$ is unramified over $p$ and $\Lambda_v\subset V_v$ is unimodular.
\item\label{item:G4} If $v$ is split in $F$ and $m_v=0$, then $\mathrm{U}(V)(F_{0,v})\cong \mathrm{GL}_n(F_{0,v})$ and we assume $\Lambda_v\subset V_v$ is self-dual. Let $K_{G,v}\cong \mathrm{GL}_n(O_{F_0,v})$ be the stabilizer of $\Lambda_v$.
\item\label{item:G5} If $v$ is split in $F$ and $m_v>0$, then again $\mathrm{U}(V)(F_{0,v})\cong \mathrm{GL}_n(F_{0,v})$ and we assume $\Lambda_v\subset V_v$ is self-dual. Further assume that $v$ splits into $w$ and $\bar{w}$ in $F$ and  all places $\nu$ of $E$ above $v$ satisfy the following two conditions.
\begin{enumerate}
  \item the place $\nu$ of $E$ matches the CM type $\Phi$ (in the sense of \cite[\S 4.3]{RSZdiagonal}): if $\phi\in\Hom(F, \overline{\mathbb{Q}})$ induces the $p$-adic place $w$ of $F$ (via $\tilde\nu: \overline{\mathbb{Q}}\hookrightarrow \overline{\mathbb{Q}}_p$), then $\phi\in\Phi$.
  \item the extension $E_{\nu}/E_{r|_v}$ is unramified, where $E_{r|_v}$ is the local reflex field as defined in \cite[\S4.1]{RSZdiagonal}.  We remark that this condition holds automatically if all $p$-adic places of $F$ are unramified over $p$.
  \end{enumerate}
We remark here that condition $(a)$ is automatically true when $F=F_0\kappa$  for some imaginary quadratic $\kappa/\mathbb{Q}$ and the CM type $\Phi$ is induced from a CM type of $\kappa/\mathbb{Q}$ (e.g., when $F_0=\mathbb{Q}$), or when $v$ is of degree one over $p$. Let $K_{G,v}$ be the principal congruence subgroup modulo $\pi_v^{m_v}$ inside the stabilizer of $\Lambda_v$ where $\pi_v$ is a uniformizer of $F_{0,v}$.
\end{enumerate}
In the case when the above conditions are satisfied, we denote $K$ by $K^{\mathbf{m}}$. Also denote $K^{\mathbf{m}}$ by $K^\circ$ if $m_v=0$ for all $v$. In other words $K^\circ\subset \tilde{G}(\bA_f)$ is the stabilizer of $\Lambda\otimes_{O_F} \widehat{O}_F$. Define the moduli functor $\mathcal{M}_{K^\circ}$ as follows. For a locally noetherian $O_E$-scheme $S$, $\mathcal{M}_{K^\circ}(S)$ is the groupoid of tuples $(A_0,\iota_0,\lambda_0,A,\iota,\lambda,\cF)$ such that
\begin{enumerate}
    \item $A$ (resp. $A_0$) is an abelian scheme over $S$;
    \item $\iota$ (resp. $\iota_0$) is an action of $O_F$ on $A$ (resp. $A_0$) satisfying the Kottwitz condition of signature $\{(r_\phi,r_{\bar\phi})_{\phi\in \Phi}\}$ (resp. $\{(1,0)_{\phi \in \Phi}\}$):
    \begin{equation}\label{eq:global Kottwitz condition}
      \mathrm{charpol}(\iota(a)\mid \Lie A)=\prod_{\phi\in \Phi} (T-\phi(a))^{r_\phi}\cdot (T-\bar\phi(a))^{r_{\bar\phi}}
    \end{equation}
    for any $a\in O_F$;
    \item $\lambda$ (resp. $\lambda_0$) is a polarization of $A$ (resp. $A_0$) whose Rosati involution induces the automorphism given by the nontrivial Galois automorphism of $F/F_0$ via $\iota$ (resp. $\iota_0$);
    \item $\cF$ is locally a direct summand $\Oo_S$-submodule of $\Lie A$ which is stable under the $O_F$-action. Moreover $O_F$ acts on $\cF$ by the structural morphism and on $\Lie A/\cF$ by the Galois conjugate of the structural morphism.
\end{enumerate}
We further require the following conditions to be satisfied.
\begin{enumerate}[label=(H\arabic*)]
\item\label{item:H1} $(A_0,\iota_0,\lambda_0)\in \mathcal{M}_0^\xi$ where $\cM_0^\xi=\cM_0^{(1),\xi}$ in the notation of \cite[\S 4.1]{RSZsurvey} (where $(1)$ is the unit ideal in $O_F$) is an integral model of $\mathrm{Sh}_{K_{Z^\Q}}(Z^\Q,h_{Z^\Q})$ depending on the choice of a similarity class $\xi$ of $1$ dimensional $F/F_0$-Hermitian spaces.
\item For each finite place $v$ of $F_0$, $\lambda$ induces a polarization $\lambda_v$ on the $p$-divisible group $A[v^\infty]$. We require $\ker \lambda_v\subset A[\iota(\varpi_v)]$ and is of rank equal to the size of $\Lambda_v^\sharp/\Lambda_v$, where $\varpi_v$ is a uniformizer of $F_{0,v}$.
\item For each place $v$ of $F_0$, we require the \emph{sign condition} and \emph{Eisenstein condition} as explained in \cite[\S4.1]{RSZdiagonal}. We remark that the sign condition holds automatically when $v$ is split in $F$, and the Eisenstein condition holds automatically when the places of $F$ above $v$ are unramified over $p$.
\item\label{item:H4} We impose the \emph{Kr\"amer condition} on $\cF$ as explained in  \cite[Definition 6.10]{RSZsurvey}.
\end{enumerate}
A morphism $(A_0,\iota_0,\lambda_0,A,\iota,\lambda,\cF)\rightarrow (A'_0,\iota'_0,\lambda'_0,A',\iota',\lambda',\cF')$ is a pair $(f_0:A_0\rightarrow A'_0, f:A\rightarrow A')$  of $O_F$-linear isomorphism of abelian schemes over $S$ such that $f^*(\lambda')=\lambda$, $f^*_0(\lambda'_0)=\lambda_0$, $f_*(\cF)=\cF'$.
Let $\mathcal{V}_\mathrm{ram}$ (resp. $\mathcal{V}_\mathrm{asd}$) be the set of finite places $v$ of $F_0$ such that $v$ is ramified in $F$ (resp. $v$ is inert in $F$ and $\Lambda_v$ is almost self-dual). By \cite[Theorem 5.2]{RSZdiagonal},
the moduli problem $\mathcal{M}_{K^\circ}$ is representable by a Deligne-Mumford stack over $O_E$ which is regular and semistable at all places of $E$ above $\mathcal{V}_\mathrm{ram}\cup \mathcal{V}_\mathrm{asd}$.
The generic fiber of $\mathcal{M}_{K^\circ}$ is the Shimura variety $\Sh_{K^\circ}$. For a general $\mathbf{m}$, define $\mathcal{M}_{K^\mathbf{m}}$ to be the normalization of $\mathcal{M}_{K^\circ}$ in $\Sh_{K^\mathbf{m}}$. Then by \cite[Theorem 5.4]{RSZdiagonal}, $\mathcal{M}_{K^\mathbf{m}}$ is representable by a Deligne-Mumford stack over $O_E$ which is regular and semistable at all places of $E$ above $\mathcal{V}_\mathrm{ram}\cup \mathcal{V}_\mathrm{asd}$. Its localization at each finite place $\nu$ of $E$ agrees with the semi-integral models defined in and \cite[\S 4]{RSZdiagonal} or \cite[\S 11]{LZ}.

\subsection{Global main theorems}\label{subsec:global main theorems}
From now on we assume $K=K^{\mathbf{m}}$ and simply denote $\cM_{K^\mathbf{m}}$ by $\cM$.
Let $\mathbb{V}$ be the \emph{incoherent}
$\mathbb{A}_{F}/\mathbb{A}_{F_0}$-hermitian space associated to $V$, namely $\mathbb{V}$ is totally positive definite and $\mathbb{V}_v \cong V_v$ for all finite places $v$. Let $\varphi_K\in \sS(\bV^m_f)$ be a $K$-invariant (where $K$ acts on $\bV_f$ via the second factor $K_G$) Schwartz function. We say $\varphi_K$ is admissible if it is $K$-invariant and  $\varphi_{K,v}=\mathbf{1}_{(\Lambda_v)^m}$ at all $v$ nonsplit in $F$.

First, we consider a special admissible Schwartz function of the form  \begin{equation}\label{eq:testfunction}
  \varphi_K=(\varphi_i)\in \sS(\mathbb{V}^m_f),\quad \varphi_{i}=\mathbf{1}_{\Omega_i},\quad i=1,\ldots,m,
\end{equation}
where  $\Omega_i \subset \mathbb{V}_f$ is a $K$-invariant open compact subset such that $\Omega_{i, v}=\Lambda_v$ at all $v$ nonsplit in $F$.  Given $t_i\in F$ and $\varphi_i$, there exists a unique special divisor $\cZ(t,\varphi_i)$ over $\cM_K$ such that for each place $\nu$ of $E$ inducing a non-split place of $F_0$, the base change of $\cZ(t,\varphi_i)$ to $\text{Spec}\, O_{E,(\nu)}$ agrees with the  special divisor defined as in \cite[\S 13.3]{LZ}, and for each $\nu$ inducing a split place of $F_0$, it agrees with the Zariski closure of the special divisor over the generic fiber of $\mathcal{M}_K$. Then we have the following decomposition (cf. \cite[(11.2)]{KR2}),
$$
\mathcal{Z}\left(t_1, \varphi_1\right) \cap \cdots \cap \mathcal{Z}\left(t_m, \varphi_m\right)=\bigsqcup_{T \in \operatorname{Herm}_m(F)} \mathcal{Z}\left(T, \varphi_K\right),
$$
where $\cap$ denotes taking fiber product over $\mathcal{M}_K$, and the indexes $T$ have diagonal entries $t_1, \ldots, t_m$.

Let $T\in \Herm_n(F)$ be a nonsingular $F/F_0$-hermitian matrix of size $n$. Given $(T,\varphi_K)$, we can define an arithmetic degree as follows.    First, analogous to the local situation (\ref{eq:IntL}), we can define a local arithmetic intersection numbers $\Int_{T,\nu}(\varphi_K)$ for any place $\nu$ of $E$. First we assume $\nu$ is finite and let $v$ be the place of $F_0$ under $\nu$.
By the same proof of \cite[Proposition 2.22]{KR2}, it suffices to consider the case when $v$ is nonsplit in $F$.
When $\varphi_K$ is of the form \eqref{eq:testfunction}, define
\begin{equation}\label{eq:localint}
\Int_{T,\nu}(\varphi_K)\coloneqq \frac{1}{[E:F_0]} \cdot \chi(\mathcal{Z}(T,\varphi_K)_\nu, \mathcal{O}_{\mathcal{Z}(t_1,\varphi_1)_\nu} \otimes^\mathbb{L}\cdots \otimes^\mathbb{L}\mathcal{O}_{\mathcal{Z}(t_n,\varphi_n)_\nu})\cdot\log q_\nu,
\end{equation}
where $q_\nu$ denotes the size of the residue field $k_\nu$ of $E_{\nu}$, $\cZ(T,\varphi_K)_\nu$ and $\cZ(t,\varphi_i)_\nu$ denote the base change to $O_{E,(\nu)}$, $\mathcal{O}_{\mathcal{Z}(t_i,\varphi_i)_\nu}$ denotes the structure sheaf of the  Kudla--Rapoport divisor $\mathcal{Z}(t_i,\varphi_i)$, $\otimes^\mathbb{L}$ denotes the derived tensor product of coherent sheaves on $\mathcal{M}$, and $\chi$ denotes the Euler--Poincar\'e characteristic. For a general admissible function $\varphi_K$, we can extend the definition $\C$-linearly.
Using the star product of Kudla's Green functions, we can also define a local arithmetic intersection number  $\Int_{T,\nu}(\sy,\varphi_K)$ at infinite places (\cite[\S15.3]{LZ}), which depends on a parameter $\sy\in \Herm_n(F_{\infty})_{>0}$  where  $F_{\infty}=F \otimes_{\mathbb{Q}}\mathbb{R}$. Combining all the local arithmetic numbers together, define the \emph{global arithmetic intersection number}, or the \emph{arithmetic degree} of the special cycle $\mathcal{Z}(T,\varphi_K)$ in the arithmetic Chow group of $\cM$,
\[\wdeg_T(\sy,\varphi_K)\coloneqq \sum_{\nu\nmid\infty}\Int_{T,\nu}(\varphi_K)+\sum_{\nu\mid \infty}\Int_{T,\nu}(\sy,\varphi_K).\]

\begin{theorem}\label{thm:Arithmetic Siegel--Weil formula: nonsingular terms}
Let $\Diff(T, \mathbb{V})$ be the set of places $v$ such that $\mathbb{V}_{v}$ does not represent $T$.  Let $T\in\Herm_n(F)$ be nonsingular such that $\Diff(T,\mathbb{V})=\{v\}$ where $v$ is nonsplit in $F$ and not above 2. Assume $\varphi_K \in \sS(\mathbb{V}^m_f)$ is admissible. Then
\[\wdeg_T(\sy, \varphi_K)q^T=c_K\cdot \pEis_T(\sz,\varphi_K),\]
where $q^T\coloneqq\psi_\infty(\tr T\sz)$, $c_K=\frac{(-1)^n}{\vol(K)}$ is a nonzero constant independent of $T$ and $\vol(K)$ is the volume of $K$ under a suitable Haar measure. Finally,  $\pEis_T(\sz,\varphi_K)$ is the $T$-th coefficient of the modified central derivative of  Eisenstein series in  (\ref{eq:Fourier coefficients of central derivative})
\end{theorem}
\begin{proof}
When $v$ is finite and $v\notin \mathcal{V}_{\mathrm{ram}}\cup \mathcal{V}_\mathrm{asd}$, the theorem is proved in \cite[Theorem 13.6]{LZ}. For $v\in \mathcal{V}_\mathrm{asd}$, our definition of $\Int_{T,\nu}(\varphi_K)$ differs from that of \cite[(13.5.0.14)]{LZ}. Correspondingly on the analytic side, our definition of $\pEis_T(\sz,\varphi_K)$ is also modified, see \eqref{eq: modified asd term} and \eqref{eq:modified central derivative}. However using \cite[Theorem 10.5.1]{LZ}  instead of \cite[Theorem 10.3.1]{LZ}, the proof of \cite[Theorem 13.6]{LZ} works the same way in this case.
When $v$ is infinite, the theorem is proved in \cite[Theorem 4.17,4.20]{Liu2011} and independently in \cite[Theorem 1.1.2]{GS}. When $v$ is finite and $v\in \mathcal{V}_{\mathrm{ram}}$, the theorem is a corollary of Theorem \ref{thm: main thm} and can be proved in the same way as \cite[Theorem 12.3]{HSY3}.
\end{proof}

We say $\varphi_{v}\in \sS(\mathbb{V}_v^n)$ is \emph{nonsingular} if its support lies in $\{\mathbf{x}\in \mathbb{V}_v^n: \det T(\mathbf{x})\ne0\}$, see \cite[\S12.3]{LZ} or \cite[Proposition 2.1]{Liu2011}.
\begin{theorem}\label{thm:Arithmetic Siegel--Weil formula}
  Assume that $F/F_0$ is split at all places above 2. Further assume that $\varphi_K$ is admissible and nonsingular  at two places split in $F$. Then
  $$\wdeg(\sz, \varphi_K)=c_K\cdot \pEis(\sz,\varphi_K),$$
  where $\wdeg(\sz, \varphi_K)$ is defined in \eqref{eq:generating series of arithmetic degree}.
  In particular, $\wdeg(\sz, \varphi_K)$ is a nonholomorphic hermitian modular form of genus $n$.
\end{theorem}
\begin{proof}
The theorem can be derived from Theorem \ref{thm:Arithmetic Siegel--Weil formula: nonsingular terms} by the same way as \cite[Theorem 15.5.1]{LZ}.
\end{proof}

\bibliographystyle{alpha}
\bibliography{reference}
\end{document}